\documentclass[12pt]{amsart}

\usepackage{mathrsfs}
\usepackage{amssymb}
\usepackage{cite}
\usepackage[all]{xy}
\usepackage{graphicx}
\usepackage{textcomp}
\usepackage{charter}
\usepackage[vflt]{floatflt}
\usepackage{mathtools}
\usepackage{enumerate}
\usepackage{wasysym}
\usepackage{rotating}
\usepackage{pdflscape}
\usepackage{fullpage}

\usepackage[pdftex,colorlinks]{hyperref}
\hypersetup{%pdftitle={Lie Group Vessels},
pdfauthor={Shamovich, E.},
bookmarksnumbered,
pdfstartview={FitH}}%

\newtheorem{thm}{Theorem}[section]
\newtheorem{prop}[thm]{Proposition}
\newtheorem{lem}[thm]{Lemma}
\newtheorem{cor}[thm]{Corollary}
\newtheorem*{thm*}{Theorem}
\newtheorem*{cor*}{Corollary}
\newtheorem{quest}[thm]{Question}

\theoremstyle{definition}

\newtheorem{example}[thm]{Example}
\newtheorem{rem}[thm]{Remark}

\theoremstyle{remark}

\numberwithin{equation}{section}

\usepackage{mathrsfs}
\newcommand{\GL}{\operatorname{{\mathbf GL}}}

\newcommand{\cA}{{\mathcal A}}
\newcommand{\cB}{{\mathcal B}}
\newcommand{\cC}{{\mathcal C}}
\newcommand{\cD}{{\mathcal D}}
\newcommand{\cE}{{\mathcal E}}
\newcommand{\cF}{{\mathcal F}}

\newcommand{\cH}{{\mathcal H}}
\newcommand{\cI}{{\mathcal I}}
\newcommand{\cJ}{{\mathcal J}}

\newcommand{\cM}{{\mathcal M}}

\newcommand{\cO}{{\mathcal O}}

\newcommand{\cQ}{{\mathcal Q}}

\newcommand{\cV}{{\mathcal V}}

\newcommand{\fA}{{\mathfrak A}}
\newcommand{\fB}{{\mathfrak B}}
\newcommand{\fC}{{\mathfrak C}}

\newcommand{\fL}{{\mathfrak L}}

\newcommand{\fV}{{\mathfrak V}}
\newcommand{\fW}{{\mathfrak W}}

\newcommand{\C}{{\mathbb C}}

\newcommand{\N}{{\mathbb N}}

\newcommand{\M}{{\mathbb M}}

\newcommand{\W}{{\mathbb W}}
\newcommand{\D}{{\mathbb D}}

\newcommand{\bB}{\mathbb{B}}
\newcommand{\bC}{\mathbb{C}}
\newcommand{\bD}{\mathbb{D}}

\newcommand{\bM}{\mathbb{M}}
\newcommand{\bN}{\mathbb{N}}

\newcommand{\mlt}{\operatorname{Mult}}
\newcommand{\alg}{\operatorname{alg}}

\newcommand{\Rep}{\operatorname{Rep}}
\newcommand{\Aut}{\operatorname{Aut}}

\newcommand{\ph}{\delta_{cb}}
\newcommand{\phb}{\delta_{b}}
\newcommand{\CB}{\operatorname{CB}}

\newcommand{\cbh}{\widehat{\fC\fB_d}}

\newcommand{\ol}{\overline}

\def\mcc{M\raise.5ex\hbox{c}C}
\def\mccarthy{M\raise.5ex\hbox{c}Carthy}

\begin{document}

\title[The bounded and completely bounded isomorphism problem]{Algebras of noncommutative functions on subvarieties of the noncommutative ball: the bounded and completely bounded isomorphism problem}

\author{Guy Salomon}
\address{Department of Pure Mathematics, University of Waterloo, Waterloo, ON, Canada}
\email{gsalomon@uwaterloo.ac.il}
\author{Orr M. Shalit}
\address{Department of Mathematics\\
Technion --- Israel Institute of Technology\\
Haifa, Israel}
\email{oshalit@technion.ac.il}
\author{Eli Shamovich}
\address{Department of Pure Mathematics, University of Waterloo, Waterloo, ON, Canada}
\email{eshamovi@uwaterloo.ca}

\subjclass[2010]{47L80, 46L07,47L25}

\thanks{The first author was partially supported by the Clore Foundation.
The second author was partially supported by Israel Science Foundation Grant no. 195/16. The third author was partially supported by the Fields Institute for Research in the Mathematical Sciences.}

\begin{abstract}
%This paper continues our study of algebras of bounded, noncommutative (nc) holomorphic functions on nc subvarieties of the nc unit ball $\fB_d$ (where $d \in \bN \cup \{\infty\}$).
Given a noncommutative (nc) variety $\fV$ in the nc unit ball $\fB_d$, we consider the algebra $H^\infty(\fV)$ of bounded nc holomorphic functions on $\fV$. 
%We find that the finite dimensional and weak-$*$ continuous representations are parametrized by the similarity envelope $\widetilde{\fV}$ of $\fV$. 
%Therefore, $H^\infty(\fV)$ can be considered as a certain subalgebra of nc holomorphic functions on $\widetilde{\fV}$. 
We investigate the problem of when two algebras $H^\infty(\fV)$ and $H^\infty(\fW)$ are isomorphic.  
We prove that these algebras are weak-$*$ continuously isomorphic if and only if there is an nc biholomorphism $G : \widetilde{\fW} \to \widetilde{\fV}$ between the similarity envelopes that is bi-Lipschitz with respect to the free pseudo-hyperbolic metric. 
Moreover, such an isomorphism always has the form $f \mapsto f \circ G$, where $G$ is an nc biholomorphism.  
These results also shed some new light on automorphisms of the noncommutative analytic Toeplitz algebras $H^\infty(\fB_d)$ studied by Davidson--Pitts and by Popescu. 
In particular, we find that $\operatorname{Aut}(H^\infty(\fB_d))$ is a proper subgroup of $\operatorname{Aut}(\widetilde{\fB}_d)$. 

When $d<\infty$ and the varieties are homogeneous, we remove the weak-$*$ continuity assumption, showing that two such algebras are boundedly isomorphic if and only if there is a bi-Lipschitz nc biholomorphism between the similarity envelopes of the nc varieties. 
We provide two proofs.
In the noncommutative setting, our main tool is the  noncommutative spectral radius, about which we prove several new results. 
In the free commutative case, we use a new free commutative Nullstellensatz that allows us to bootstrap techniques from the fully commutative case. 

%We discuss completely bounded versions of the above classification results, and it turns out that in the homogeneous case, two algebras are boundedly isomorphic if and only if they are completely boundedly isomorphic. 
%
%We also briefly treat the algebras $A(\fV)$ of bounded nc holomorphic functions on $\fV$ that extend uniformly continuously to $\ol{\fV}$. 
%In the case of homogeneous varieties, we find the same classification results as for the algebras of bounded nc holomorphic functions. 
\end{abstract}

\maketitle
%\tableofcontents

%%%%%%%%%%%%%%%%%%%%%%%%%%%%%%%%%%%%%%%%%%%%%%%%%%%%%%%%%%%%%%%%%%%%%%
%%%%%%%%%%%%%%%%%%%%%%%%%%%%%%%%%%%%%%%%%%%%%%%%%%%%%%%%%%%%%%%%%%%%%%
\section{Introduction}
%
%I think JFA is a very good choice and it should be accepted to there (80\%). OK, let's try there. 
%Sure it's fine by me that you two make some changes. Following the referee report we just got, I suggest that when editing the introduction, try to make it a bit less modest. For example, in the last paragraph of section 1.1, not to start with "This paper is a continuation of [34]." but rather "In this paper our goal is to understand... The situation is markedly different from [34]..." etc. Really the situation is different and the usual varieties in the ball cannot be used to classify the algebras. In order to get a handle on the invariants we need 1) to develop new results on nc function theory, and 2) apply them in a clever way. Section 2,5,6, and 8 all have some interesting new ideas and results that should be highlighted. 

The study of holomorphic functions in one and several complex variables is an old and well developed subject with countless applications. 
One fruitful venue of research is the interplay between complex analysis and operator algebras, exemplified in the Sz. Nagy-Foias bounded analytic functional calculus \cite{SzNFo10} and in Taylor's functional calculus of several commuting operators \cite{Tay70}. 
The main limitation of this approach lies in the fact that in general operators do not commute. 
This led Taylor and Voiculescu to study analytic functions of several noncommuting variables \cite{Tay72frame,Tay73,Voic85,Voic86,Voic04,Voic10}. 
In fact, classical analytic functions can be viewed as shadows of their noncommutative (nc, for short) counterparts, the so-called {\em nc holomorphic functions}, under an appropriate quotient map. 
We will defer the formal definition of nc holomorphic functions to the next section. 
For now it suffices to view them as a generalization of polynomials in several noncommuting variables, i.e, elements of a free associative algebra. 

A natural choice for the domain of nc holomorphic functions in $d$ noncommuting variables is the {\em nc universe} $\bM_d: = \sqcup_{n=1}^{\infty} M_n(\bC)^d$, the graded set of all $d$-tuples of complex square matrices. 
One can view the matrix levels as capturing the noncommutative nature of our functions, in analogy with Kaplanski's theorem \cite[Theorem 2]{Kap48} that states that elements of the free associative algebra $\C\langle z_1,\ldots,z_d\rangle$ are determined by their values on $\bM_d$.

Not surprisingly, however, $\bM_d$ is in a sense too big to have a rich theory of holomorphic functions, so just like in the classical case, analysts usually consider only certain subsets of it. 
Every classical domain, such as a ball or a polydisc admits natural ``quantizations''. 
In particular, in this paper we will focus on the {\em nc ball} $\fB_d$: the set of all $d$-tuples $(X_1,\dots X_d)  \in \bM_d$ satisfying   $X_1X_1^* + \cdots + X_dX_d^* < 1$.
 
It is worth noting that for every $n$, the $n$th level of the nc universe admins a natural $\GL_n$-action given by $S \cdot (X_1,\ldots,X_d) = (S^{-1}X_1S,\dots,S^{-1}X_dS)$. 
Unfortunately, most domains, including our nc ball, are not invariant under this action. 
We therefore define, for every set $\Omega \subseteq \bM_d$, its {\em similarity orbit} $\widetilde{\Omega}$, which is just the orbit of $\Omega$ under the levelwise $\GL_n$-action.

% The first level of the nc ball is the classical unit ball $\bB_d \subseteq \bM_d$ and thus the nc ball can be viewed as a noncommutative analogue of the unit ball. 

The algebra of bounded nc holomorphic functions on the nc ball, $H^{\infty}(\fB_d)$ turns out to be the free semigroup algebra $\fL_d$ studied by Arias and Popescu and Davidson and Pitts, see for example \cite{AriasPopescu, DavPitts2, DavPitts1, Popescu89, Popescu06b, Popescu10}. 
The free semigroup algebra is the universal weak-$*$ closed algebra generated by a pure row contraction. 
Its quotients by weak-$*$ closed two sided ideals are thus universal weak-$*$ closed algebras generated by pure row contractions satisfying prescribed algebraic relations. 
We would like to understand when such algebras are isomorphic. 
Though isomorphism can be understood in many ways we will focus on continuous and completely bounded isomorphisms. 
Such a question of course begs the introduction of an invariant. 
An immediate candidate for an invariant is the vanishing locus inside the nc ball of a weak-$*$ closed two-sided ideal of $H^{\infty}(\fB_d)$. 
We will call these subsets $\fV \subseteq \fB_d$ nc varieties. 
Each variety $\fV$ comes equipped with its algebra of functions $H^{\infty}(\fV)$, namely the quotient of the free semigroup algebra by the ideal of functions that vanish on the variety.

%In this paper, we consider subsets of the nc ball that are cut out by bounded nc holomorphic functions on the nc ball, and refer to them as {\em nc subvarieties of the nc ball} or simply as {\em nc varieties}. Our goal is to classify the algebras $H^\infty(\fV)$ of bounded nc holomorphic functions that live on these nc varieties in terms of the geometry of the varieties. These algebras are not only Banach algebras: they carry a structure of an operator algebra, and in fact, they can be viewed as multiplier algebras of certain nc reproducing kernel Hilbert spaces (so they also have a natural weak-$*$ topology as dual Banach spaces). A well-known example of such an algebra is revealed when the variety $\fV$ is the whole nc ball $\fB_d$. In this case, $H^\infty(\fV)$ becomes 

It turns out that the information contained in our nc subvariety of the ball is not at all sufficient to answer the (completely) bounded isomorphism question. 
The question forces us to treat the geometry of $\widetilde{\fV}$, the similarity envelope $\fV$. 
Several delicate issues immediately arise. 
The first obstacle is the fact that the similarity envelope is, in general, unbounded and thus classical results on bounded domains are not readily available. Furthermore, there is an algebraic complication since similarity orbits of even single points can be quite complicated.

The ``geometry" required for the classification theorem, is encoded in certain pseudo-metrics defined on the similarity envelope. 
More precisely, we define two {\em free pseudo-hyperbolic distances} $\phb$ and $\ph$ on the similarity envelope of the closed nc ball $\overline{\fB_d}$ that measure the difference between point evaluations: the first in terms of the usual operator norm and the second in terms of the complete bounded norm.

We can now state the classification theorem for homogeneous nc varieties (i.e. nc varieties that are cut out by homogeneous nc functions). 

\begin{thm}[Theorem \ref{thm:isom_thm_for_homo} and Corollary \ref{cor:cb-isom<=>biholo_weakstar}] \label{thm:main_int}
Let $\fV \subseteq \fB_d$ and $\fW \subseteq \cB_e$ be two homogeneous nc varieties. 
The following statements are equivalent: 
\begin{enumerate}[(a)]
\item $H^\infty(\fV)$ and $H^\infty(\fW)$ are weak-$*$ continuously isomorphic. 
\item $H^\infty(\fV)$ and $H^\infty(\fW)$ are boundedly isomorphic. 
\item $H^\infty(\fV)$ and $H^\infty(\fW)$ are completely boundedly isomorphic. 
\item There exists a $\phb$-bi-Lipschitz nc biholomorphism mapping $\widetilde{\fW}$ onto $\widetilde{\fV}$. 
\item There exists a $\ph$-bi-Lipschitz nc biholomorphism mapping $\widetilde{\fW}$ onto $\widetilde{\fV}$. 
\item There exists a $\phb$-bi-Lipschitz linear map mapping $\widetilde{\fW}$ onto $\widetilde{\fV}$. 
\item There exists a $\ph$-bi-Lipschitz linear map mapping $\widetilde{\fW}$ onto $\widetilde{\fV}$. 
\end{enumerate}
In addition, any isomorphism that appears in (a)--(c) can be viewd as a precompostion with a $\ph$-bi-Lipschitz nc biholomorphism between the similarity envelopes.
\end{thm}

In the not necessarily homogeneous case, we lose the convenience of having a linear map --- meaning, we no longer have (f) and (g) --- and the rest of the equivalence list splits into two as follows: (a), (a)+(b), and (d) are equivalent, and (a)+(c) and (e) are equivalent (see Corollary \ref{cor:cb-isom<=>biholo_weakstar}).

One may rightly ask what about classifying these algebras up to an isometric  or completely isometric isomorphism. 
This question is somewhat easier to attack. In \cite{SalShaSha17} we showed that when it comes to homogeneous varieties such an isomorphism exists if and only if one of the varieties is the image of the other, under an nc automorphism of the nc ball $\fB_{\max\{d,e\}}$, and any such isomorphism is given by a precomposition with such an nc automorphism. 
The group of nc automorphisms of the nc ball is the well-known group of M\"obius transformations \cite{Popescu10}, so the class of isometric or completely isometric isomorphisms between the algebras $H^{\infty}(\fV)$ and $H^{\infty}(\fW)$ is somewhat poor.\footnote{In \cite{SalShaSha17}, we also examined the nonhomogeneous case, and we showed that these algebras are completely isometrically isomorphic if and only if the varieties $\fV$ and $\fW$ are nc {\em biholomorphic}. 
The main result of \cite{Sham18} then shows that, at least when the varieties contain a scalar point, such an nc biholomorphism is just a restriction of an nc automorphism of the nc ball.}

The situation in this paper is markedly different. For example, while the nc automorphism group of the nc ball is crystal-clear, the group of bi-Lipschitz nc automorphisms of the similarity envelope of the nc ball is far from being well understood. 
In fact, Theorem \ref{thm:main_int} together with the fact that algebraic automorphisms of $\fL_d$ are automatically weak-$*$ continuous \cite[Theorem 4.6]{DavPitts2} imply that it can be identified with the group of algebraic automorphisms of the algebra $\fL_d \cong H^\infty(\fB_d)$, which is mysterious in many ways (for example, it is not clear whether there are quasi-inner automorphisms which are not inner).

Some of the main tools that are developed to obtain Theorem \ref{thm:main_int} are interesting in their own right. For example, in Section \ref{sec:prelim_jspr} we prove an nc counterpart of the well-known Schwarz Lemma of the disc.

\begin{thm}[Lemma \ref{lem:free_spectral_Schwarz} and Proposition \ref{prop:coisometric_derivative}]\label{thm:Schwarz_int}
Let $f \colon \D \to \widetilde{\fB_d}$ be a holomorphic function mapping $0$ to $0$, and let $\rho$ denote the joint spectral radius of a $d$-tuple of matrices. Then
\begin{enumerate}[(a)]
\item $\rho(f(z)) \leq |z|$ for every $z \in \D$ and $\rho(f^{\prime}(0)) \leq 1$; and
\item if $f'(0)$ is an irreducible coisometry, then $f(z)$ is similar to $z f'(0)$ for very $z \in \D$.
\end{enumerate}
\end{thm}

The case where the varieties $\fV$ and $\fW$ contain only commuting tuples is of special interest. 
The isomorphism problem in the radical commutative case was treated by Davidson, Ramsey and the second author in \cite{DRS11} and \cite{DRS15}. 
We develop new machinery to deal with such ``commutative noncommutative'' varieties even in the non-reduced case, which also gives rise to a different proof of Theorem \ref{thm:main_int}, this time with a commutative flavor. 
One main ingredient in this machinery is a certain type of a Nullstellensatz.

\begin{thm}[Theorem \ref{thm:null_com_homo}]\label{thm:Nlstln_int}
Let $\fV \subseteq \fB_d$ be a homogeneous nc variety containing only commuting $d$-tuples and let $V = \fV(1)$ be the scalar level of $\fV$.
Then there exists an integer $N$ such that for every nc function $f \in H^\infty(\fB_d)$ that vanishes on $V$, the $N$th power $f^N$ of $f$ vanishes on the whole nc variety $\fV$.
\end{thm}

At the end of this paper we turn to another operator algebra related to an nc variety $\fV$: the algebra $A(\fV)$ of all bounded analytic functions that extend to uniformly continuous functions on $\ol{\fV}$. 
We prove the following analog of Theorem \ref{thm:main_int}.

\begin{thm}[Theorems \ref{thm:isom_thm_for_homoAV1} and \ref{thm:isom_thm_for_homoAV2}] \label{thm:main_int_cont}
Let $\fV \subseteq \fB_d$ and $\fW \subseteq \cB_e$ be two homogeneous nc varieties. 
The following statements are equivalent: 
\begin{enumerate}[(a)]
\item $A(\fV)$ and $A(\fW)$ are boundedly isomorphic. 
\item $A(\fV)$ and $A(\fW)$ are completely boundedly isomorphic. 
\item There exists a $\phb$-bi-Lipschitz nc biholomorphism mapping $\widetilde{\ol\fW}$ onto $\widetilde{\ol\fV}$. 
\item There exists a $\ph$-bi-Lipschitz nc biholomorphism mapping $\widetilde{\ol\fW}$ onto $\widetilde{\ol\fV}$. 
\item There exists a $\phb$-bi-Lipschitz linear map mapping $\widetilde{\ol\fW}$ onto $\widetilde{\ol\fV}$. 
\item There exists a $\ph$-bi-Lipschitz linear map mapping $\widetilde{\ol\fW}$ onto $\widetilde{\ol\fV}$. 
\end{enumerate}
In addition, any isomorphism that appears in (a)--(b) can be viewd as a precompostion with a $\phb$-bi-Lipschitz or a $\ph$-bi-Lipschitz nc biholomorphism mapping one similarity envelope of the variety's clousre onto the other.
\end{thm}

We note that in order to treat the algebras $A(\fV)$ for nonhomogeneous varieties one would have to face some nontrivial nc-function-theoretic issues, which are beyond the scope of this paper.  

Besides the introduction, this paper consists of nine sections, and we shall now describe their content. 
We start with a brief digression in which we treat the purely algebraic case, which quickly illustrates the utillity of the nc point of view. 
In Section \ref{sec:algebraic} we prove algebraic analogues of Theorems \ref{thm:main_int} and \ref{thm:main_int_cont}, which we believe are interesting in their own right. 

In Section \ref{sec:Dru-Arv_bdd_rep} we lay the foundations of the rest of the paper by providing the necessary preliminaries. 
In Section \ref{sec:prelim_jspr} we give two alternative descriptions of $\widetilde{\fB_d}$ and present the nc counterpart of the well-known Schwartz lemma of the disc as presented in Theorem \ref{thm:Schwarz_int}. Section \ref{sec:ph_dist} is dedicated to two pseudo-hyperbolic distances on the similarity envelope of the free ball, and in Section \ref{sec:isomorphisms} we use these distances to state and prove the isomorphism theorem for general nc varieties. 
In Section \ref{sec:homogeneous} we specialise to the homogeneous case, in which we obtain the sharper results described in Theorem \ref{thm:main_int}. Section \ref{sec:q-commutation} contains a detailed study of the family of nc varieties determined by the $q$-commutation relations.  
In Section \ref{sec:connections_to_comm} the results of Section \ref{sec:homogeneous} are derived again, this time in the case of commutative varieties, using commutative techniques that includes the Nullstellensatz presented in Theorem \ref{thm:Nlstln_int}. 
Then, in the las section, Section \ref{sec:continuous_case}, we turn our attention to the norm closed (instead of \textsc{wot}-closed) algebras generated by the free polynomial functions on homogeneous nc varieties and prove Theorem \ref{thm:main_int_cont}.

%%%%%%%%%%%%%%%%%%%%%%%%%%%%%%%%
\section{A digression -- the purely algebraic case}\label{sec:algebraic}
As motivation for our main investigations, we consider the purely algebraic analogues of our problems.
Let $\bC[z] := \bC[z_1, \ldots, z_d]$ denote the algebra of complex polynomials in $d$ commuting variables (here $d<\infty$).
With every ideal $I \triangleleft \bC[z]$ one naturally associates the corresponding affine variety
\[
Z(I) = Z_{\bC^d} (I) = \{z \in \bC^d : p(z) = 0 \,\, \textrm{ for all } \,\, p \in I\}.
\]
Together with this geometric object, there are two natural algebraic objects: the quotient $\bC[z]/I$ --- which is the universal unital algebra generated by $d$ commuting elements satisfying the relations in $I$ --- and the algebra of regular functions:
\[
\bC[Z(I)] = \left\{ p\big|_{Z(I)} : p \in \bC[z]\right\}.
\]
Consider two ideals $I, J \triangleleft \bC[z]$.
One may ask when are the quotients $\bC[z]/I$ and $\bC[z]/J$ isomorphic, as algebras.
When $I$ and $J$ are radical, then it follows from Hilbert's Nullstellensatz that $\bC[z]/I \cong \bC[Z(I)]$ and $\bC[z]/J \cong \bC[Z(J)]$, and it is then not hard to show that $\bC[z]/I$ and $\bC[z]/J$ are isomorphic if and only if there exist polynomial maps $F,G:\bC^d \to \bC^d$ that restrict to mutually inverse bijections between $Z(J)$ and $Z(I)$.

What happens if $I$ and $J$ are not radical ideals?
The concrete and geometric object $Z(I)$ is no longer a complete invariant for the quotient algebra $\bC[z]/I$.
Algebraic geometry offers some elaborate but opaque ``geometric" replacements for the variety.
A more simple-minded (and perhaps more satisfying) alternative is suggested to us by nc function theory.

We can consider $\bC[z]$ as an algebra of nc functions on $\fC\bM_d$ (recall that $\fC\bM_d$ is the set of all commuting $d$-tuples of matrices, of all sizes).
Given an ideal $I \triangleleft \bC[z]$, let
\[
Z_{\fC\bM_d}(I) = \{X \in \fC\bM_d : p(X) = 0 \,\, \textrm{ for all } \,\, p \in I\}.
\]
Points in $Z_{\fC\bM_d}(I)$ correspond bijectively to all finite dimensional representations of $\bC[z]/I$, via the map $\pi$ that sends every finite dimensional representation $\rho$ to its image on the coordinate functions in the quotient:
\[
\pi(\rho) = (\rho(z_1+I), \ldots, \rho(z_d+I)).
\]
The inverse of $\pi$ is given by
\[
\pi^{-1} : X \mapsto \rho_X,
\]
for all $X \in Z_{\fC\bM_d}(I)$, where $\rho_X$ is evaluation at $X$:
\[
\rho_X(p+I) = p(X).
\]

Suppose we are given a homomorphism $\alpha : \bC[z]/I \to \bC[z]/J$. 
Then $\alpha$ gives rise to a map between the spaces of representations, by $\alpha^* : \rho \mapsto \rho \circ \alpha$. 
Now, as 
\[
(\alpha(z_1+I), \ldots, \alpha(z_d+I)) \in \bC[z]/J\times \dots \times \bC[z]/J ,
\]
there exists $F = (F_1, \ldots, F_d) \in \bC[z] \times \cdots \times \bC[z]$ such that 
$\alpha(z_i+I)=F_i+J$ for every $i$.
Then, if $X \in Z_{\fC\bM_d}(J)$,
\[
\begin{split}
F(X)		&=	\left(\rho_X(F_1+J),\dots, \rho_X(F_d+J) \right)\\
		&=	\left(\rho_X(\alpha(z_1+I)),\dots, \rho_X(\alpha(z_d+I)) \right)\\
		&=	\left(\alpha^*(\rho_X)(z_1+I),\dots, \alpha^*(\rho_X)(z_d+I) \right)\\
		&=	\pi(\alpha^*(\rho_X)).
\end{split}
\]
Thus, we see that a homomorphism $\alpha : \bC[z]/I \to \bC[z]/J$ gives rise to a polynomial map $F \in \bC[z]^d$ mapping $Z_{\fC\bM_d}(J)$ into $Z_{\fC\bM_d}(I)$.

On the other hand, suppose we are given a polynomial map that restricts to a map from $Z_{\fC\bM_d}(J)$ into $Z_{\fC\bM_d}(I)$. Let
\[
\bC[Z_{\fC\bM_d}(I)] = \left\{p\big|_{Z_{\fC\bM_d}(I)} : p \in \bC[z]\right\}.
\]
Then $F$ clearly gives rise, via pre-composition, to a homomorphism from $\bC[Z_{\fC\bM_d}(I)]$ to $\bC[Z_{\fC\bM_d}(J)]$.

We therefore see that a homomorphism $\bC[z]/I \to \bC[z]/J$ always gives rise to a polynomial map that restricts to a map from $Z_{\fC\bM_d}(J)$ into $Z_{\fC\bM_d}(I)$, and such a map always gives rise to a homomorphism $\bC[Z_{\fC\bM_d}(I)] \to \bC[Z_{\fC\bM_d}(J)]$. To close the loop, we need a link, a Nullstellensatz, between the quotient $\bC[z]/I$ and the function algebra $\bC[Z_{\fC\bM_d}(I)]$.

Given a set $S \subseteq \fC\bM_d$, let
\[
I(S) = I_{\bC[z]}(S) = \{p \in \bC[z] : p(X) = 0 \,\, \textrm{ for all } \,\, X \in S\}.
\]
In \cite[Corollary 11.7]{SalShaSha17} we obtained that $I(Z_{\fC\bM_d}(J)) = J$ for every $J \triangleleft \bC[z]$.
We call this the {\em commutative free Nullstellensatz}, and it has been known to algebraists in one form or another (see \cite{EisHoch79}).
The commutative free Nullstellensatz implies at once that $\bC[z]/I$ is isomorphic to the function algebra $\bC[Z_{\fC\bM_d}(I)]$.

This shows, additionally, that homomorphisms $\alpha : \bC[Z_{\fC\bM_d}(I)] \to \bC[Z_{\fC\bM_d}(J)]$ are necessarily pre-composition with a polynomial map $F$ mapping $Z_{\fC\bM_d}(J)$ into $Z_{\fC\bM_d}(I)$. Indeed, after identifying $\bC[z]/I \cong \bC[Z_{\fC\bM_d}(I)]$, we saw before that the relation between $\alpha$ and $F$ is given by $F(X)=\pi(\alpha^*(\rho_X))$ for all $X \in Z_{\fC\bM_d}(J)$. Applying $\pi^{-1}$ to this equality, we obtain that $\alpha^*(\rho_X)=\rho_{F(X)}$, and therefore
\[
\alpha(p) = p \circ F \,\,  , \,\, \textrm{ for all } p \in \bC[Z_{\fC\bM_d}(I)].
\]
We summarize the conclusion of the above discussion, in the case of an isomorphism, as follows.
%%%%%%%%%%
\begin{thm}\label{thm:alg_com_iso}
Let $I$ and $J$ be two ideals in $\bC[ z ]$.
The algebras $\bC[z]/I$ and $\bC[z]/J$ are isomorphic if and only if $Z_{\fC\bM_d}(J)$ and $Z_{\fC\bM_d}(I)$ are isomorphic, in the sense that there exists polynomial maps $F$ and $G$ that restrict to bijections between $Z_{\fC\bM_d}(J)$ and $Z_{\fC\bM_d}(I)$.
Moreover, every homomorphism from $\bC[z]/I$ to $\bC[z]/J$ is implemented by a polynomial map $F: Z_{\fC\bM_d}(J)  \to Z_{\fC\bM_d}(I)$.
\end{thm}

One can consider ideals inside the algebra $\bC\langle z \rangle := \bC\langle z_1, \ldots, z_d \rangle$ of free polynomials in $d$ noncommuting variables, and given such an ideal $I \triangleleft \bC\langle z \rangle$, one can consider the noncommutative variety
\[
Z_{\bM_d}(I) = \{X \in \bM_d : p(X) = 0 \,\, \textrm{ for all } \,\, p \in I\}.
\]
If $I$ is a homogeneous ideal, then there is an appropriate noncommutative homogeneous Nullstellensatz \cite[Theorem 7.3]{SalShaSha17}, which says that (with obvious notation) 
\[
I_{\bC\langle z \rangle}(Z_{\bM_d}(J)) = J \,\,\textrm{  for every homogeneous }\,\, J \triangleleft \bC\langle z \rangle .
\]
If one replaces the commutative free Nullstellensatz with the noncommutative homogeneous Nullstellensatz, then the same argument as above (where polynomials are replaced by free polynomials) gives the following theorem:
%%%%%%%%%%
\begin{thm}\label{thm:alg_hom_iso}
Let $I$ and $J$ be two homogeneous ideals in $\bC\langle z \rangle$.
The algebras $\bC\langle z \rangle /I$ and $\bC\langle z \rangle /J$ are isomorphic if and only if $Z_{\bM_d}(J)$ and $Z_{\bM_d}(I)$ are isomorphic, in the sense that there exists free polynomial maps $F$ and $G$ that restrict to bijections between $Z_{\bM_d}(J)$ and $Z_{\bM_d}(I)$.
Moreover, every homomorphism $\bC\langle z \rangle/I$ and $\bC\langle z \rangle/J$ is implemented by a free polynomial map $F: Z_{\bM_d}(J)  \to Z_{\bM_d}(I)$.
\end{thm}

Our main goal in the remainder of this paper is to understand the analogue of the above results for algebras of bounded analytic nc functions. 
As such algebras are not finitely generated in an algebraic sense by the coordinate functions, there are interesting technical issues to overcome. 
%In \cite{SalShaSha17}, the main problem that we treated, was to describe when two algebras $H^\infty(\fV)$ and $H^\infty(\fW)$ are completely isometrically isomorphic, where $\fV, \fW \subseteq \fB_d$ and $d < \infty$. 
%One of our main results there was that there exists a completely isometric isomorphism $\alpha : H^\infty(\fV) \to H^\infty(\fW)$ if and only if there exist $F,G \in {\rm{End}}(\fB_d)$ such that $F\circ G|_{\fW}={\rm{id}}|_{\fW}$ and $G\circ F|_{\fV}={\rm{id}}|_{\fV}$; in this case, $\alpha$ is implemented by composition with $G$, i.e. $\alpha(f) = f \circ G$ (see \cite[Theorem 6.12]{SalShaSha17}). 
%This raised the interesting nc function theoretic question, whether this function $G$ (and therefore $F$, too) can be guaranteed to be an automorphism of the nc unit ball. 
%This was solved affirmatively in \cite[Theorem 8.4]{SalShaSha17} in the case that $\fV$ and $\fW$ are homogeneous varieties (and in this case it was shown that one variety is in fact the image of the other variety under a unitary linear transformation). 
%In \cite{Sham18} this was solved affirmatively in the much broader case where $\fV$ and $\fW$ contain a scalar point. 
%
%In this paper, we turn our attention to the problem of describing (completely) bounded isomorphisms $\alpha : H^\infty(\fV) \to H^\infty(\fW)$. 
%It turns out that the geometry of the varieties themselves can no longer serve as the classifying structure.
%Our results are described in the next subsection. 
%

In passing, we are happy to note that it was by considering the operator algebraic problems that the above purely algebraic results crystallized for us, and they seem to have been overlooked.
It is worth noting that Theorem \ref{thm:alg_com_iso} can be restated in the free setting as a theorem on ideals $I, J$ of $\bC\langle z \rangle$ that contain the commutant ideal. This point of view, together with Theorem \ref{thm:alg_hom_iso}, suggests that there might be a general theorem regarding {\em any} pair of ideals $I, J \triangleleft \bC\langle z \rangle$. Such a generalization, however, fails to be true (consider the trivial ideal and the ideal generated by the nc polynomial $z_1 z_2 - z_2 z_1 - 1$).

\section{Preliminaries}\label{sec:Dru-Arv_bdd_rep}

\subsection{Nc functions and nc varieties}

We study noncommutative (nc) function theory in $d$ complex variables, where $d \in \bN$ or $d = \infty$.
Let $M_n = M_n(\bC)$ denote the set of all $n \times n$ matrices over $\bC$, and let $M_n^d$ be the set of all $d$-tuples $X = (X_1, X_2, \ldots)$ of $n \times n$ matrices, such that the row $X$ determines a bounded operator from $\bC^n \oplus \bC^n \oplus \ldots$ to $\bC^n$.
We norm $M_n^d$ with the row operator norm $\|X\| = \|\sum_j X_j  X_j^*\|^{1/2}$, and endow $M_n^d$ with the induced topology. 
We define (the {\em nc universe})
\[
\bM_d = \sqcup_{n=1}^{\infty} M_n^d,
\]
and (the {\em commutative nc universe})
\[
\fC \bM_d = \{X \in \bM_d : X_i X_j = X_j X_i \,\, \textrm{ for all } \,\, i,j\}.
\]

A set $\Omega \subseteq \bM_d$ is said to be an {\em nc set} if it is closed under direct sums.
If $\Omega$ is an nc set, we denote $\Omega(n) = \Omega \cap M_n^d$. 
A subset $\Omega \subseteq \bM_d$ is said to be {\em open/closed} if for all $n$, $\Omega(n)$ is open/closed. 
This collection of open sets gives rise to a topology on $\M_d$ and its subsets, called the {\em disjoint union topology}. 
%An nc set $\Omega \subseteq \M_d$ will be said to be an {\em nc domain} if it is open and if every $\Omega_n$ is connected.
The {\em boundary of $\Omega$}, denoted $\partial \Omega$, is defined to be $\sqcup_{n=1}^\infty \partial \Omega(n)$. 
The principle nc open set that we shall consider is the {\em ($d$-dimensional) open matrix unit ball} $\fB_d$,  which is defined to be
\[
\fB_d = \left\{ X \in \bM_d : \|X\|^2 = \left\|\sum X_j X_j^*\right\| < 1\right\}.
\]

%%%%%%%%%%%%%%%%%%%%%%%%%%%%%%%%%%%
Let $\cV$ be a vector space.
A function $f$ from an nc set $\Omega \subseteq \bM_d$ to $\sqcup_{n=1}^\infty M_n(\cV)$ is said to be an {\em nc function (with values in $\cV$)} if
\begin{enumerate}
\item $f$ is graded: $X \in \Omega(n) \Rightarrow f(X) \in M_n(\cV)$; 
\item $f$ respects direct sums: $f(X \oplus Y) = f(X) \oplus f(Y)$; and
\item $f$ respects similarities: if $X \in \Omega(n)$ and $S \in M_n$ is invertible, and if $S^{-1} X S \in \Omega(n)$, then $f(S^{-1} X S) = S^{-1} f(X) S$.
\end{enumerate}
An nc function with values in $\bC$ is said to be a {\em scalar-valued} nc function. 
In this paper, we shall deal only with scalar-valued nc functions. 

A function $f$ defined on an nc open set $\Omega$ is said to be {\em nc holomorphic} if it is an nc function and, in addition, it is locally bounded. 
It turns out that an nc holomorphic function is really a holomorphic function when considered as a function $f : \Omega(n) \to M_n$, for all $n$, and moreover it has a ``Taylor series'' at every point (see \cite{KVBook}). 
We therefore allow ourselves to use the terms {\em holomorphic} and {\em analytic} interchangeably. 

We let
\[
H^\infty(\fB_d) = \left\{ f : \fB_d \to \M_d : f \textrm{ is a bounded nc function} \right\}, 
\]
denote the algebra of all bounded nc holomorphic functions on the nc unit ball. 
This algebra (with the $\sup$ norm) can be shown to be the algebra of multipliers of the noncommutative Drury-Arveson space, and as an operator algebra it is unitarily equivalent to the {\em noncommutative analytic Toeplitz algebra} studied by Davidson and Pitts, which was also studied by Popescu under the name {\em noncommutative Hardy algebra} (see \cite[Section 3]{SalShaSha17} and \cite{DavPitts1, DavPitts2, Popescu95b, Popescu06b}). 

A {\em noncommutative (nc) variety} in the unit ball is a set of the form
\[
V_{\fB_d}(S) = \{X \in \fB_d :  f(X) = 0 \,\, \textrm{ for all } f \in S\},
\]
where $S$ is a set of bounded nc holomorphic functions. 
We emphasize that although it makes sense to consider varieties determined by arbitrary nc holomorphic functions, such generality is beyond our scope, and we will assume that every variety is given by $S \subseteq H^\infty(\fB_d)$. 
Note that nc varieties are nc sets.

A {\em free polynomial} is an element in $\bC\langle z_1, \ldots, z_d\rangle$ (the free algebra in $d$ noncommuting variables).
Let $\W_d$ be the free monoid on $d$ generators $\{g_1, \ldots, g_d\}$. 
If $k = g_{i_1} \cdots g_{i_n} \in \W_d$ (in which case we write $|k| = n$), we define the {\em free monomial} $z^k = z_{i_1}  \cdots z_{i_n}$. 
A polynomial $p(z) = \sum_{k \in \W_d} a_k z^k$ can be written in a unique way as 
$p(z) = \sum_{n \in \bN} p_n(z)$
where
\[p_n(z) = \sum_{k\in \W_d, |k|=n} a_k z^k .
\]
The polynomial $p_n$ is called {\em the homogeneous component of degree $n$} of $p$.
Every free polynomial is a (scalar-valued) nc function on $\M_d$ in a natural way, by evaluation; moreover, this nc function is bounded on every uniformly bounded subset of $\M_d$.  
The inclusion of $\bC\langle z_1, \ldots, z_d\rangle$ in $H^\infty(\fB_d)$ is injective (this might not be entirely obvious, but follows from the known fact the matrix algebra $M_n$ satisfies no polynomial identity of degree less than $2n$; see \cite{AmiLev50}).

%%%%%%%%%%%%%%%%%%%%%%%%%%%%%%%%%%%

Given an nc variety $\fV \subseteq \fB_d$, we define $H^\infty(\fV)$ to be the algebra of bounded nc functions on $\fV$, and $A(\fV)$ to be the algebra of bounded nc functions that extend to uniformly continuous functions on  $\overline{\fV}$ (the closure is taken in the disjoint union topology). 
We give $H^\infty(\fV)$ and $A(\fV)$ the obvious operator algebra structure, where the matrix norm of $F \in M_n(H^\infty(\fV))$ is given by
\[
\|F\| = \sup_{z \in \fV} \|F(z)\| = \sup_{k \in \bN} \sup_{z \in \fV(k)} \|F(z)\|_{M_{nk}}.
\]
It follows from \cite[Corollary 5.6]{BMV15b} (see also \cite[Theorem 4.7]{SalShaSha17}) that if $\Omega = \fB_d$, and if
$f : \fV \to \bM_1$ is an nc function that is bounded on $\fV$, then there exists a bounded nc holomorphic function $F$ on $\fB_d$ such that $\|F\| = \|f\|$ and $f = F\big|_\fV$.

The algebra $H^\infty(\fV)$ can be identified with the multiplier algebra of an nc reproducing kernel Hilbert space \cite[Theorem 5.4]{SalShaSha17}. 
Given an nc variety $\fV$, we define
%\[
%\cI_\fV = \{p \in \bC\langle z_1, \ldots, z_d\rangle :   p(Z) = 0 \,\, \textrm{ for all } Z \in \fV\} 
%\]
\[
\cI_\fV = \{f \in A(\fB_d) :   f(Z) = 0 \,\, \textrm{ for all } Z \in \fV\}
\]
and
\[
\cJ_\fV = \{f \in H^\infty(\fB_d) :   f(Z) = 0 \,\, \textrm{ for all } Z \in \fV\}.
\]
Then we have that $H^\infty(\fV)$ is completely isometrically isomorphic to $H^\infty(\fB_d)/\cJ_{\fV}$ \cite[Theorems 5.2 and 5.4]{SalShaSha17}, and when $\fV$ is homogeneous, $A(\fV)$ is completely isometrically isomorphic to $A(\fB_d)/\cI_{\fV}$ \cite[Proposition 9.7]{SalShaSha17}.

%
%This paper continues the work \cite{SalShaSha17}.
%Our goal is to understand the structure of the algebras of the form $A(\fV)$ and $H^\infty(\fV)$, and in particular to classify them up to (completely) bounded isomorphisms in terms of the nc complex analytic geometry of the underlying variety $\fV$. 
%Before we state our main results --- for motivation as well as perspective --- let us discuss a similar problem in the purely algebraic case. 

%%%%%%%%%%%%%%%%%%%%%%%%%%%%%%%%%%%%
%%%%%%%%%%%%%%%%%%%%%%%%%%%%%%%%%%%%
%%%%%%%%%%%%%%%%%%%%%%%%%%%%%%%%%%%%
%%%%%%%%%%%%%%%%%%%%%%%%%%%%%%%%%%%%

%\section{The nc Drury--Arveson space and bounded representations of $H^\infty(\fV)$}

%%%%%%%%%%%%%%%%%%%%%%%%%%%%%%%%%%%

%Recall that $\W_d$ denotes the free monoid on $d$ generators ($d \in \bN \cup \{\infty\}$).

\subsection{The nc Szego kernel and the nc Drury--Arveson space}\label{subsec:ncDA}
The {\em nc Szego kernel} $K(Z,W)$ on the nc ball $\fB_d$ of $\M_d$ defined by
\[
K(Z,W)(T) = \sum_{k\in\W_d} Z^k T W^{*k} ,
\]
for $Z \in \fB_d(n)$, $W \in \fB_d(m)$ and $T \in M_{n \times m}(\C)$ \cite{BMV15b}.

Consider the nc function $K_{W,v,y}$ for $W \in \fB_d(m)$, $v, y \in \C^m$, defined for $Z \in \fB_d(n)$:
\[
K_{W,v,y}(Z)u = K(Z,W)(uv^*)y = \sum_{k \in \W_d} Z^k u v^* W^{*k} y = \sum_{k \in \W_d} \langle  y, W^k v \rangle Z^k u.
\]
Thus $K_{W,v,y}$ is an nc function given by the power series
\[
K_{W,v,y}(Z) = \sum_{k \in \W_d} \langle y, W^k v \rangle Z^k .
\]
The {\em nc Drury--Arveson space} $\cH^2_{d}$ is the nc reproducing kernel Hilbert space determined by the nc Szego kernel $K$, in the sense of \cite{BMV15a}.
In \cite[Section 3]{SalShaSha17} we saw that $\mlt \cH^2_d = H^\infty(\fB_d)$.
Moreover, we observed there that $H^\infty(\fB_d)$ can be naturally identified with the {\em noncommutative analytic Toeplitz algebra} $\fL_d$ that Davidson and Pitts treated in \cite{DavPitts2}, or, what is the same, Popescu's {\em noncommutative Hardy algebra} \cite{Popescu91}, which he denotes by $\mathscr F^\infty_d$.

Recall the Bunce--Frazho--Popescu dilation theorem \cite{Bunce,Frazho,Popescu89}, which says that if $T = (T_1, \ldots, T_d)$ is a pure row contraction on a Hilbert space $H$, then there is a Hilbert space $\cE$ of dimension $d$, an auxiliary Hilbert space $\cD$, and an isometry $V: H \to \cF(\cE) \otimes \cD$ such that $V H$ is a co-invariant subspace for the free shift $L \otimes I_\cD$, and such that
\[
T_i = V^* (L_i \otimes I_\cD) V \quad , \quad i=1,2, \ldots, d.
\]
Identifying $H^\infty(\fB_d)$ with $\fL_d$, this gives rise to a functional calculus: for every pure row contraction $T$, there is a weak-operator continuous, unital, completely contractive homomorphism
\[
\Phi_T : H^\infty(\fB_d) \to \overline{\alg}^{\textup{wot}}(T),
\]
given by $\Phi_T(f) = V^* (f(L) \otimes I_\cD) V$ (where $f(L)$ is the image of $f$ in $\fL_d$ under the isomorphism $H^\infty(\fB_d)  \cong \fL_d$).
If $T$ is a strict contraction ($\|T\| := \|T\|_{\textup{row}}<1$) then it is not hard to see that $\Phi_T$ becomes the evaluation at $T$, that is
\[
\Phi_T\left(\sum_{k \in \W_d} a_k z^k\right) = \sum_{k \in \W_d} a_k T^k.
\]
This gives rise to a functional calculus for multiplier algebras on nc varieties, versions of which were observed in \cite{Popescu06a,ShalitSolel}.

%%%%%%%%%%%%%%%%%%%%%%%%%%%%%%%%%%%

In a previous work \cite[Corollary 2.6]{SalShaSha17}, we showed that when $\Omega \subseteq \M_d$ is an nc set and $k$ is a completely positive nc kernel, then the nc multiplier algebra $\mathcal M(k)$ of the nc RKHS $\mathcal H(k)$ associated to $k$ is weak-$*$ closed and is therefore a dual algebra. In addition, to every nc subvariety $\fV \subseteq \fB_d$ we associated an nc RKHS $\cH^2_\fV$ --- a subspace of the nc Drury--Arveson space $\cH^2_d$ --- and its algebra of multipliers $\mlt \cH^2_\fV$, and we showed that $\mlt \cH^2_\fV$ is completely isometrically isomorphic to $H^{\infty}(\fV)$ (this relies on the deep result --- originally due to Agler and McCarthy \cite[Theorem 1.5]{AM15d} for an algebraic variety, and later extended to a much general setting by Ball, Marx and Vinnikov \cite[Theorem 3.1]{BMV15b} --- that every bounded nc function on $\fV$ extends to a bounded nc function of $\fB_d$ with the same norm); see \cite[Theorem 5.4]{SalShaSha17}. Thus, $H^\infty(\mathfrak{V})$ has a natural weak-$*$ topology.

Davidson and Pitts showed that when the variety $\fV$ is the whole nc ball $\fB_d$, then the weak-operator and weak-$*$ topologies of $H^{\infty}(\fV)=H^\infty(\fB_d)$ coincide \cite{DavPitts1}. 
Since for every nc variety $\fV \subseteq \fB_d$ the algebra $H^{\infty}(\fV)$ is a quotient of $H^{\infty}(\fB_d)$ by the $\textsc{wot}$-closed ideal
\[
\cJ_\fV = \{f \in H^\infty(\fB_d) :   f(Z) = 0 \,\, \textrm{ for all } Z \in \fV\}, 
\]
the weak-operator and weak-$*$ topologies coincide on $H^{\infty}(\fV)$ as well.
We record this for a later use.

\begin{thm}\label{thm:weak-$*$_and_wot_agree}
Let $\fV \subseteq \fB_d$ be an nc variety. Then the weak-operator and weak-$*$ topologies coincide on $H^\infty(\fV)$.
\end{thm}

As a corollary we obtained that for every nc subvariety $\fV \subseteq \fB_d$ and 
every a pure row contraction $T$ satisfying $\cJ_\fV \subseteq \ker \Phi_T$ (in particular, if $T \in \fV$), there is a weak-operator continuous, unital completely contractive homomorphism from $\mlt \cH^2_\fV$ to $\overline{\alg}^{\textup{wot}}(T)$ mapping $M_z$ to $T$ \cite[Corollary 5.3]{SalShaSha17} (see also \cite{AriasPopescu}).

%
%\begin{cor}\label{cor:functional_calculus}
%Let $\fV \subseteq \fB_d$ be an nc variety.
%Let $T$ be a pure row contraction.
%If $\cJ_\fV \subseteq \ker \Phi_T$ (in particular, if $T \in \fV$), then there is a weak-operator continuous, unital completely contractive homomorphism from $\mlt \cH^2_\fV$ to $\overline{\alg}^{\textup{wot}}(T)$ mapping $M_z$ to $T$.
%\end{cor}

\subsection{Bounded finite dimensional representations of $H^\infty(\fV)$} \label{subsec:fdim_rep}

Let $\Rep_{k}(\cA)$ denote the space of all unital bounded representations of an operator algebra $\cA$ on $\C^k$, and let $\Rep^{cc}_{k}(\cA)$ denote the completely contractive representations in $\Rep_{k}(\cA)$.

%%%%%%%%%%%%%%%%%%%%%%%%%%%%%%%%%%%
For $\fV \subseteq \fB_d$, we let $\ol{\fV}^{p}$ denote the set of all pure $T\in \ol{\fB}_d$ such that $\cJ_\fV \subseteq \ker \Phi_T$.
The completely contractive representations of $H^\infty(\fV)$ were determined in \cite[Theorem 6.3]{SalShaSha17}:
%%%%%%%%%%%%%%%%%%%%%%%%%%%%%%%%%%%
\begin{thm}\label{thm:finite_dim_ccreps}
Let $\fV \subseteq \fB_d$ be an nc variety.
For very $k \in \bN$, there is a natural continuous projection $\pi_{k}$ of $\Rep^{cc}_{k}(H^\infty(\fV))$ into the closed unit ball $\overline{\fB_d}(k)$, given by
\[
\pi_{k}(\Phi) = (\Phi(z_1), \ldots, \Phi(z_d)).
\]
For every $T \in \ol{\fV}^{p}$, there is a unique weak-$*$ continuous representation $\Phi_T \in \pi_k^{-1}(T)$, and these are the only weak-$*$ continuous elements in $\Rep^{cc}_{k}(H^\infty(\fV))$.
If $d < \infty$ and $T \in \fV$, then $\pi_{k}^{-1}(T)$ is the singleton $\{\Phi_T\}$.
Moreover, if $d < \infty$, then
\[
\pi_{k}(\Rep^{cc}_{k}(H^\infty(\fV)))\cap \fB_d(k) = \fV(k) .
\]
\end{thm}

We can now describe the bounded finite dimensional representations.
All that is required to pass from completely contractive to bounded representations, is to pass to the similarity invariant envelope of the variety.

%%%%%%%%%%%%%%%%%%%%%%%%%%%%%%%%%%%
\begin{lem}\label{lem:pure_con_sim_strict}
$\widetilde{\ol{\fV}^{p}} = \widetilde{\fV}$ for every nc variety $\fV \subseteq \fB_d$.
\end{lem}
\begin{proof}
Recall that $\ol{\fV}^{p}$ is the set of all pure $T\in \ol{\fB}_d$ such that $\cJ_\fV \subseteq \ker \Phi_T$.
It suffices to show that if $X \in \ol{\fV}^p$, then $X$ is similar to a strict row contraction $Y \in \fV$.
By Proposition \ref{prop:pure_iff_conj} below, since $X$ is pure, it is similar to some strict row contraction $Y$.
Being similar to $X$, we see that $Y$ annihilates every $f \in \cJ_\fV$, so $Y \in \fV$.
\end{proof}

%%%%%%%%%%%%%%%%%%%%%%%%%%%%%%%%%%%
\begin{thm}\label{thm:finite_dim_reps}
Let $\fV \subseteq \fB_d$ be an nc variety.
For very $k \in \bN$, there is a natural continuous projection $\pi_{k}$ of $\Rep_{k}(H^\infty(\fV))$ into the similarity invariant envelope $\widetilde{\overline{\fB_d}}(k)$ of the closed unit ball $\overline{\fB_d}(k)$, given by
\[
\pi_{k}(\Phi) = (\Phi(z_1), \ldots, \Phi(z_d)).
\]
For every $T \in \widetilde{\ol{\fV}^{p}} = \widetilde{\fV}$, there is a unique weak-$*$ continuous representation $\Phi_T \in \pi_{k}^{-1}(T)$, and these are the only weak-$*$ continuous elements in $\Rep_{k}(H^\infty(\fV))$.
If $d < \infty$ and $T \in \widetilde{\fV}$, then $\pi_{k}^{-1}(T)$ is the singleton $\{\Phi_T\}$.
Moreover, if $d < \infty$, then
\[
\pi_{k}(\Rep_{k}(H^\infty(\fV)))\cap \widetilde{\fB_d(k)} = \widetilde{\fV(k)}.
\]
\end{thm}
\begin{proof}
Let $\Phi \in \Rep_{k}(H^\infty(\fV))$.
By a theorem of Smith, $\Phi$, being a bounded map into $M_n(\C)$, is completely bounded (see \cite{PaulsenBook}, pp. 113--114).
By a theorem of Paulsen \cite[Theorem 9.1]{PaulsenBook},  there exists a completely contractive $\rho \in \Rep^{cc}_k(H^\infty(\fV))$ and an invertible matrix $S$ such that $\Phi(\cdot) = S \rho(\cdot) S^{-1}$.
By Theorem \ref{thm:finite_dim_ccreps}, $\pi_k(\rho) \in \overline{\fB_d}(k)$.
Thus,
\[
\pi_{k}(\Phi) := (\Phi(z_1), \ldots, \Phi(z_d)) = S\pi_k(\rho)S^{-1} \in \widetilde{\overline{\fB_d}}(k).
\]
Moreover, $\rho$ is weak-$*$ continuous if and only if $\Phi$ is, so the rest follows from Theorem \ref{thm:finite_dim_ccreps} and Lemma \ref{lem:pure_con_sim_strict}.
\end{proof}

Henceforth, we will identify $\widetilde{\fV}$ with the weak-$*$ continuous finite dimensional representations of $H^\infty(\fV)$, via the identification
\[
T \leftrightarrow \Phi_T .
\]

%%%%%%%%%%%%%%%%%%%%%%%%%%%%%%%%%%
\begin{quest}
Are finite dimensional representations automatically continuous?
\end{quest}

%%%%%%%%%%%%%%%%%%%%%%%%%%%%%%%%%%%%%%%%%%%%%%%%%%%%%%%%%%%%%%%%%%%%%%
%%%%%%%%%%%%%%%%%%%%%%%%%%%%%%%%%%%%%%%%%%%%%%%%%%%%%%%%%%%%%%%%%%%%%%
\section{The similarity invariant envelope and the joint spectral radius}\label{sec:prelim_jspr}

Recall that $\bM_d = \sqcup_{n=1}^{\infty} M_n^d$, the set of all $d$ tuples of all matrices, of all sizes. 
Unless we make the explicit assumption that $d$ is finite, we shall treat $d \in \bN \cup \{\infty\}$. 

%%%%%%%%%%%%%%%%%%%%%%%%%%%%%%%%%%%%%%%%%%%%%%%
\subsection{The similarity invariant envelope}

Let $\Omega \subseteq \M_d$.
The {\em similarity invariant envelope} (or simply the {\em similarity envelope}) of $\Omega$ is defined to be the set $\widetilde{\Omega}$ given by
\[
\widetilde{\Omega}(n) = \{S X S^{-1}  : X \in \Omega(n) , S \in \GL_n(\C)\}.
\]
Clearly, if $\Omega$ is open, then so is $\widetilde{\Omega}$.
In the appendix of \cite{KVBook}, it is shown that $\widetilde{\Omega}$ is an nc set if $\Omega$ is, and that every nc function $f$ on $\Omega$ extends uniquely to an nc function $\widetilde{f}$ on $\widetilde{\Omega}$.

Every $f \in H^\infty(\fB_d)$ extends uniquely to an nc holomorphic function $\widetilde{f}$ on $\widetilde{\fB_d}$, and likewise, if $\fV \subseteq \fB_d$ is an nc variety, then every $f \in H^\infty(\fV)$ extends uniquely to an nc holomorphic function $\widetilde{f}$ on $\widetilde{\fV}$. 
Of course, $\widetilde{f}$ need not be bounded. 
The nc holomorphic functions on $\widetilde{\fB_d}$ which are extensions of functions in $H^\infty(\fB_d)$ are precisely those whose restriction to $\fB_d$ is bounded, or, equivalently, whose restriction to every set bounded in the completely bounded norm, is bounded.

The similarity envelope of $\fB_d$ will be of particular interest in this paper and thus we will provide two descriptions of this set. 
For every $n \in \N$ and $X \in \M_d(n)$ we have the associated completely positive map on $M_n$, given by 
\[
\Psi_X(T) = \sum_{j=1}^d X_j T X_j^* .
\]
A point is in $\overline{\fB_d}$ if and only if $\Psi_X(I) \leq I$. 
Recall that a point $X \in \M_d(n)$ is called {\em pure} if $\lim_{k\to \infty} \Psi^k_X(I) = 0$. 
We will first show that every pure point is in fact similar to a strict row contraction. To prove this claim we need two lemmas.
The first lemma, Lemma  \ref{lem:conj_to_cont}, was proved in greater generality by Popescu in \cite[Theorem 3.8]{Popescu14} (see also \cite[Section 5]{Popescu03}), and so were Lemma \ref{lem:tilde_is_spectral_ball} and Proposition \ref{prop:pure_iff_conj}; see Remark \ref{Rk:PopescuWasFirst}.

\begin{lem}[Theorem 3.8 of \cite{Popescu14}] \label{lem:conj_to_cont}
For every $X \in \bM_d$, $X$ is similar to a strict row contraction if and only if there exists a strictly positive $A$, such that $\Psi_X(A) < A$.
\end{lem}
\begin{proof}
First let us assume that such an $A$ exists. 
If $S$ is the unique positive square root of $A$, then by assumption we have
\[
\sum_{j=1}^d \left( S^{-1} X_j S\right) \left( S^{-1} X_j S\right)^* = S^{-1} \Psi_X(A) (S^{-1})^* < I.
\]
In other words $S^{-1} X S$ is a strict row contraction.

Conversely, if $X$ is similar to a strict row contraction, then there exists $S$ invertible, such that $S^{-1} X S$ is a strict row contraction. Using the same consideration as above we deduce that
\[
\Psi_X(S S^*) < S S^*.
\]
Now set $A = S S^*$.
\end{proof}

\begin{lem} \label{lem:strict_unnecessary}
Suppose that $X \in \M_d(n)$ and $Y \in \M_d(m)$ are similar to strict row contractions.
Then, for every $d$-tuple of $n \times m$ matrices $Z$ we have that the point $\begin{bmatrix} X & Z \\ 0 & Y \end{bmatrix}$ is similar to a strict row contraction.
\end{lem}
\begin{proof}
If $S^{-1} X S$ and $T^{-1} Y T$ are strict row contractions, then by conjugating by $S \oplus T$ we may assume that in fact $X$ and $Y$ are strict row contractions, since $Z$ will be replaced by $S^{-1} Z T$.

Let $t > 0$ and consider the conjugation of our point by $t^{-1} I_n \oplus t I_m$
\[
\begin{bmatrix} t I_n & 0 \\ 0 & t^{-1} I_m \end{bmatrix} \begin{bmatrix} X & Z \\ 0 & Y \end{bmatrix} \begin{bmatrix} t^{-1} I_n & 0 \\ 0 & t I_m \end{bmatrix} = \begin{bmatrix} X & t^2 Z \\ 0 & Y \end{bmatrix}
\]
Note that $\lim_{t \to 0} \begin{bmatrix} X & t^2 Z \\ 0 & Y \end{bmatrix} = X \oplus Y \in \fB_d(n+m)$. Since $\fB_d(n+m)$ is open, there exists $t > 0$ such that $t^{-1} I_n \oplus t I_m$ implements the similarity of our matrix to a strict row contraction.
\end{proof}

A $d$-tuple $X=(X_1,\dots,X_d) \in \M_d(n)$ is called {\em irreducible} (or sometimes {\em generic}) if the only nontrivial subspace $K$ satisfying $\alg(X_1,\dots,X_d) K\subseteq K$ is $\mathbb C^d$, or equivalently, if the $X_1, \ldots, X_d$ have no joint nontrivial proper invariant subspace.
If $X$ is not irreducible, then it is called {\em reducible}.

It is not hard to see that every $X \in \M_d(n)$ is similar to a block upper triangular $d$-tuple whose diagonal blocks are all irreducible. 
Indeed, if $X$ is reducible, let $0 \subsetneq K \subsetneq \mathbb C^n$ be a joint invariant subspace of the $X_i$s. Then there exists an invertible matrix $S$ such that $S^{-1}XS=\begin{bmatrix} X'&*\\0&X'' \end{bmatrix}$ where both $X'$ and $X''$ are of size smaller than that of $X$. An induction argument now completes the proof. 

Alternatively, since it is finite dimensional, the $\alg(X_1,\dots,X_d)$-module $\mathbb C^n$ is both Noetherian and Artinian. 
Such modules (i.e., Nothereian and Artinian modules) are exactly the modules that have
composition series \cite[Proposition 3.2.2]{HazeGubarVv04}.

One way or another, the Jordan--H\"{o}lder theorem \cite[Theorem 3.2.1]{HazeGubarVv04}, implies that if
\[
\begin{bmatrix}
Y^{(1)}	& 			&	*	\\
 		&	\ddots	&		\\
0		&			&	Y^{(k)}	
\end{bmatrix}
\quad \text{and} \quad
\begin{bmatrix}
Z^{(1)}	& 			&	*	\\
 		&	\ddots	&		\\
0		&			&	Z^{(l)}	
\end{bmatrix}
\]
are two ``decompositions'' as above of $X$ --- namely, both are similar to $X$ and each of the diagonal blocks is irreducible --- then $k=l$, and there exists a permutation $\sigma$ of $\{1,\dots,k\}$ such that $Y^{(i)}$ is similar to $Z^{(\sigma(i))}$ for all $1\leq i \leq k$. 
Thus, up to similarity $Y^{(1)}, \dots, Y^{(k)}$ are unique. 
We call them the {\em Jordan--H\"{o}lder components} of $X$.

\begin{prop}[Theorem 3.8 of \cite{Popescu14}] \label{prop:pure_iff_conj}
A point $X \in \M_d$ is pure if and only if it is similar to a strict row contraction.
\end{prop}
\begin{proof}
Assume first that $X$ is similar to a strict row contraction, i.e., there exists an invertible $S$, such that $S^{-1} X S$ is a strict row contraction. 
Thus $\Psi_{S^{-1} X S}(I) < I$ and we conclude that $\lim_{k \to \infty}\Psi^k_{S^{-1} X S}(I) = 0$. 
One checks that for $A = S S^*$, we have the identity
\[
\Psi^k_{S^{-1} X S}(I) = S^{-1} \Psi^k_X(A) (S^{-1})^*.
\]
%Given $\epsilon > 0$, for $k$ sufficiently large we have that $\Phi^k_X(A) < \epsilon A$ and thus $\lim_{k \to \infty} \Phi_X(A) = 0$.
It follows that $\lim_{k \to \infty} \Psi^k_X(A) = 0$.
Since $A$ is strictly positive, we can choose $c > 0$, such that $c I \leq A$ and applying $\Psi^k$ we conclude that $\lim_{k\to\infty} \Psi^k_X(I) = 0$ or in other words that $X$ is pure.

Now assume that $X$ is pure.
First, let us assume that $X$ is also irreducible.
Let $r$ denote the spectral radius of $\Psi_X$.
By \cite[Theorem 2]{Far96} the map $\Psi_X$ is irreducible in the sense of \cite{EH78}.
By Theorems 2.3 and 2.5 in \cite{EH78}, there is a strictly positive $A \in M_n$, such that $\Psi_X(A) = r A$.
Since $\Psi_X$ is completely positive, $\|\Psi_X^k\| = \|\Psi_X^k(I)\|$ and thus $\lim_{k \to \infty} \Psi^k_X = 0$ in norm. We can conclude that $r < 1$, so $\Psi_X(A) = r A < A$. Thus by Lemma \ref{lem:conj_to_cont} we are done in the case that $X$ is irreducible.

To prove the general case we proceed by induction on the number of Jordan--H\"{o}lder components of $X$. If the number is $1$, then $X$ is irreducible and we have proved it.
Now if the number is $k$, then we can write $X = \begin{bmatrix} X^{\prime} & Z \\ 0 & X^{\prime\prime} \end{bmatrix}$, where $X^{\prime}$ is irreducible and $X^{\prime\prime}$ has $k-1$ Jordan--H\"{o}lder components.
Note that since $X$ is pure, $X^{\prime}$ and $X^{\prime\prime}$ must also be pure.
By induction, $X^{\prime}$ and $X^{\prime\prime}$ are similar to strict row contractions.
It remains to apply Lemma \ref{lem:strict_unnecessary}.
\end{proof}

%%%%%%%%%%%%%%%%%%%%%%%%%%%%%%%%%%%
\subsection{The joint spectral radius}\label{subsec:jnr}

Let $X \in \M_d(n)$. 
Given $k \in \bN$, write $X^{(k)}$ for the row comprising of the free monomials of degree $k$ in the entries of $X$. 
Following Popescu \cite{Popescu09a}, we define the {\em joint spectral radius} of the $d$-tuple $X$ as 
\[
\rho(X) = \lim_{k \to \infty} \|\Psi_X^k(I)\|^{1/2k} . 
\]
Since $\Psi_{X^{(k)}} = \Psi_X^k$ is a completely positive map we have that $\|\Psi_X^k\| = \|\Psi_X^k(I)\|$ and thus $\lim_{k \to \infty} \|\Psi_X^k(I)\|^{1/k}$ is the spectral radius of $\Psi_X$. 
Note that $\|\Psi_X^k(I)\|^{1/2k} = \|X^{(k)}\|^{1/k}$; in particular this notion of spectral radius reduces to the usual one when $d = 1$.  

The joint spectral radius will be crucial tool for showing that every bounded isomorphism gives rise to an nc biholomorphism between the similarity envelopes of varieties. 
The following lemma lists some properties of the joint spectral radius.

\begin{lem} \label{lem:irreducible_radius}
For every $X \in \M_d(n)$ we have that:
\begin{enumerate}
\item for every $S \in \GL_n$, $\rho(X) = \rho(S^{-1} X S)$;

\item if $X$ is irreducible, then $\rho(X) = \min\{\|S^{-1} X S\| \mid S \in \GL_n\}$; and

\item if $X$ is reducible and $X_1,\ldots,X_k$ are its Jordan--H\"{o}lder components, then $\rho(X) = \max\{\rho(X_1),\ldots,\rho(X_k)\}$, and thus it is the same as the joint spectral radius of the semi-simple elements in the closure of the similarity orbit of $X$.
\end{enumerate}
\end{lem}
\begin{proof}
Let $X \in \fB_d(n)$ and let $S \in \GL_n$. Observe that for every $k \in \N$, $(S^{-1} X S)^{(k)} = S^{-1} X^{(k)} S$. 
Therefore, $\rho(S^{-1} X S) \leq \rho(X)$. 
Applying the same consideration we see that $\rho(X)  = \rho(S (S^{-1} X S) S^{-1}) \leq \rho(S^{-1} X S)$. Hence we have that similarities preserve the joint spectral radius. 
%Alternatively, one can observe that $\rho(X)^2 = \rho(\Psi_X)$ and that $\Psi_{S^{-1} X S} = (\overline{S} \otimes S^{-1}) \Psi_X (\overline{S^{_1}} \otimes S)$, where $\overline{S}$ stands for the entrywise conjugation of $S$. Since they are similar their spectral radius is the same and we are done. 
In particular we have that $\rho(X) = \rho(S^{-1} X S) \leq \|S^{-1} X S\|$ for every $S \in \GL_n$.

On the other hand, if $X$ is irreducible, then --- as noted in the proof of Proposition \ref{prop:pure_iff_conj} --- there exists a strictly positive $A \in M_n$, such that $\Psi_X(A) = \rho(\Psi_X) A = \rho(X)^2 A$. 
Letting $T = \sqrt{A}$, we put $Y = T^{-1} X T$, and we find that $\Psi_Y(I) =  T^{-1} \Psi_X(A) (T^{-1})^* = \rho(X)^2 I$. 
On the other hand, $\Psi_Y(I) = \|Y\|^2 I$, and $\rho(X)^2 I = \rho(Y)^2 I$.  
We therefore get that $\rho(Y) = \|Y\|$ and additionally $\|Y\| = \min\{ \|S^{-1} X S\| \mid S \in \GL_n\}$.

For the last claim observe that applying a similarity we may assume that $X$ is block upper triangular, or in other words, that $X$ is similar to a sum $X_{ss} + N$ of block diagonal $d$-tuple $X_{ss}$ and jointly nilpotent $d$-tuple $N$. 
Now we can choose $S_k \in \GL_n$, such that $\lim_{k \to \infty} S_k^{-1} X S_k = X_{ss}$. 
Since $\rho(S_k^{-1} X S_k) = \rho(X) = \sqrt{\rho(\Psi_X)}$, and since the spectral radius of an operator on a finite dimensional space is a continuous function, we get that $\rho(X) = \rho(X_{ss})$. 
It is obvious that $\rho(X_{ss}) = \max\{\rho(X_1),\ldots,\rho(X_k)\}$.
\end{proof}

\begin{lem}[Theorem 3.8 of \cite{Popescu14}] \label{lem:tilde_is_spectral_ball}
Let $X \in \M_d$ and $k \in \N$. 
Then, $\rho(X) < 1$ if and only if $\lim_{k \to \infty} \|X^{(k)}\| = 0$, and this happens if and only if $X \in \widetilde{\fB_d}$.
\end{lem}
\begin{proof}
If $\rho(X) < 1$, then for $k \in \N$ sufficiently large $\|X^{(k)}\|^{1/k} \leq r < 1$, and thus $\|X^{(k)}\| \leq r^k \to 0$. 
If $\lim_{k \to \infty} \|X^{(k)}\| = 0$, then $\lim_{k \to \infty} \Psi_X^k(I) = 0$, so $X$ is pure. 
By Proposition \ref{prop:pure_iff_conj}, $X \in \widetilde{\fB_d}$. 
Finally, let $X$ be a strict row contraction. 
Then $\Psi_X$ is a strict contraction, so $\rho(X) = \rho(\Psi_X)^{1/2} < 1$. 
Since similar tuples have the same joint spectral radius, the proof is complete. 
\end{proof}

\begin{rem}\label{Rk:PopescuWasFirst}
In \cite[Theorem 3.8]{Popescu14}, Popescu proves in greater generality the equivalence of the first and the third conditions in Lemma \ref{lem:tilde_is_spectral_ball}, the equivalence in Lemma \ref{lem:conj_to_cont} and the one in Proposition \ref{prop:pure_iff_conj}.
More precisely, in his notations, if one sets
$k=1$, 
$n_1=d$, 
$m_1=1$,
$q_1=Z_{1,1}+\dots+Z_{1,d}$, and
$\cQ=\{0\}$,
then one obtains
$\Phi_{q_i,X_i}(Y)=\Psi_X(Y)$,
$\Delta^p_{q,X}=(id-\Psi_X)^p$, and
$D^m_q(\cH)=\ol{(B(\cH)^d)_1}$, where $(B(\cH)^{d})_1:=\{X \in B(\cH)^d : \sum_{i=1}^d X_i X_i^* < I_\cH\}$ is the open unit ball of $B(\cH)^{d}$. 

In this case, Theorem 3.8 in his paper becomes the equivalence of the following five conditions:
\begin{enumerate}[(i)]
\item $X$ is similar to some $Y \in (B(\cH)^d)_1$,
\item there exists $A \geq 0$ such that $\Psi_X(A)<A$,
\item $\rho(X)<1$,
\item $\lim_{k\to\infty}\|\Psi_X^k(I)\| = 0$, and
\item $\Psi_X$ is power-bounded and pure, and there exists an invertible positive $B \in B(\cH)$ such that the equation $\Psi_X(Z)=Z-B$ has a positive solution $Z \in B(\cH)$.
\end{enumerate}
Consider the case $\cH=\mathbb C^n$. Then the equivalence (i)$\iff$(ii) is our Lemma \ref{lem:conj_to_cont}, the equivalence (i)$\iff$(iv) is our Proposition \ref{prop:pure_iff_conj}, and the equivalence (i)$\iff$(iii) is the equivalence of the first and the third conditions in Lemma \ref{lem:tilde_is_spectral_ball}. Condition (v) is just an alternative version of condition (ii). Nevertheless, our methods are different than those of Popescu.
\end{rem}

The goal of the following discussion is to obtain a Schwarz type lemma for the joint spectral radius, by proving that it is ``subharmonic on discs". 
We will follow the ideas of Vesentini \cite{Ves70} and \cite{Ves79}.

\begin{lem} \label{lem:log_subharmonic}
If $f \colon \D \to \M_d(n)$ is an analytic function, then the function $\log \rho(f(z))$ is subharmonic.
\end{lem}
\begin{proof}
 It is obvious that the function $f_k(z) = f(z)^{(k)}$ is holomorphic. 
Therefore, since the norm of a holomorphic Banach-valued function is subharmonic (this follows from Cauchy's formula) the function $u_k(z) = \|f_k(z)\|$ is continuous and subharmonic. 
As in the proof of \cite[Theorem 1]{Ves70} we use the fact that a function $u$ is $\log$-subharmonic in $\D$ if and only if for every $a \in \C$ the function $|e^{az}| u(z)$ is subharmonic. Therefore $\log u_k(z)$ is subharmonic.

Let us write $T_n = I_{2^n} \otimes X^{(2^n)}$, then $X^{(2^{n+1})} = X^{(2^n)} T_n$ and thus $\|X^{(2^{n+1})}\| \leq \|X^{(2^n)}\|^2$. Therefore, $u_{2^{n+1}}(z)^{1/{2^{n+1}}} \leq u_{2^n}(z)^{1/2^n}$ for every $z \in \D$. Thus the sequence of functions $u_{2^n}(z)^{1/2^n}$ monotonically pointwise decreases to $\rho(f(z))$. As in part B of the proof of \cite[Theorem 1]{Ves70} we conclude that $\log \rho(f(z))$ is subharmonic.

\end{proof}

\begin{cor} \label{cor:subharmonic}
The function $\rho(f(z))$ is subharmonic for every $f \colon \D \to \M_d(n)$ analytic.
\end{cor}

The following lemma is a version of the Schwarz lemma for the joint spectral radius. 

\begin{lem}[Vesentini--Schwarz Lemma]\label{lem:free_spectral_Schwarz}
If $f \colon \D \to \widetilde{\fB_d}(n)$ is an analytic function and $f(0) = 0$, then $\rho(f(z)) \leq |z|$ for every $z \in \D$ and $\rho(f^{\prime}(0)) \leq 1$.
\end{lem}
\begin{proof}
Note that for every $z \in \D$, $\rho(f(z)) < 1$. Since the function $g(z) = \frac{f(z)}{z}$ is analytic, by Corollary \ref{cor:subharmonic} we get that $\rho(g(z))$ is subharmonic of the disc. Now for every $z_0 \in \D$ and every $|z_0| < r < 1$, since the joint spectral radius is homogeneous we get that $\rho(g(z_0)) \leq \frac{1}{r}$. Passing to the limit we see that $\rho(g(z_0)) \leq 1$. Thus $\rho(f(z)) \leq |z|$ for every $z \in \D$.

Finally, note that $g(0) = f^{\prime}(0)$ and thus $\rho(f^{\prime}(0)) = \rho(g(0))\leq 1$.
\end{proof}

%%%%%%%%%%%%%%%%%%%%%%%%%%%%%% This is the spectral Cartan piece %%%%%%%%%%%%%%%%%%%%%%%%%%%%%

We would like to obtain a spectral counterpart of the statement in the classical Schwartz lemma that if $f \colon \D \to \D$ is such that $f(0) = 0$ and $|f'(0)| = 1$, then $f(z) = z f'(0)$. We will break the preparation for this result into several lemmas.

\begin{lem} \label{lem:similar_to_coisom}
Let $X \in \M_d(n)$ be an irreducible point, such that $\rho(X) = 1$. Then $X$ is similar to a coisometry. 
\end{lem}
\begin{proof}
By Lemma \ref{lem:irreducible_radius} we know that $X$ is similar to a point $Y$, such that $\|Y\| = \rho(Y) = 1$, hence $Y$ is an irreducible row contraction. As in the proof of Proposition \ref{prop:pure_iff_conj} we note that $\Psi_Y$ is irreducible and thus there exists a strictly positive $A$, such that $\Psi_Y(A) = A$. Taking the similarity with $\sqrt{A}$, we get a row coisometry.
\end{proof}

\begin{lem} \label{lem:func_to_coisom_constant}
Let $g \colon \D \to \overline{\fB_d}(n)$ be an analytic function, such that $g(z)$ is a coisometry for some $z_0 \in \D$. Then $g$ is constant.
\end{lem}
\begin{proof}
Consider every $X \in \overline{\fB_d}(n)$ as a linear map $X \colon \C^d \otimes \C^n \to \C^n$. Let $\xi \in \C^n$ be a unit vector and define $g_{\xi}(z) = \langle g(z) g(z_0)^* \xi, \xi \rangle$. Note that $g_{\xi}$ is an analytic function on $\D$. Furthermore, by Cauchy--Schwarz $|g_{\xi}(z)| \leq 1$ on $\D$ and $g_{\xi}(z_0) = 1$, since $g(z_0)$ is a coisometry. Conclude that $g_{\xi}$ is the constant function $1$. Since this is true for every $\xi$ we see that for every $z \in \D$, the numerical range of the operator $g(z) g(z_0)^*$ is the singleton $\{1\}$, and thus this operator is the identity. Thus,
\[
0 \leq \left( g(z) - g(z_0) \right) \left( g(z)^* - g(z_0)^* \right) = g(z) g(z)^* - g(z) g(z_0)^* - g(z_0) g(z)^* + I = g(z) g(z)^* - I \leq 0.
\]
Therefore, $g(z) = g(z_0)$.
\end{proof}
\begin{lem}\label{lem:implicit}
Let $V$ be a finite dimensional real vector space, and let $z \mapsto T_z$ be a smooth function from $\mathbb D$ to $L(V)$ the space of real linear transformations over $V$. Assume that $0$ is a simple eigenvalue (i.e., of algebraic multiplicity $1$) of $T_z$ for every $z\in \mathbb D$. Then for every $0\neq v_0 \in \ker T_0$, there exists $s>0$ and a smooth function $v: s\mathbb D  \to V$ satisfying $v(0)=v_0$, such that 
\[
\ker T_z = {\rm{span}} \{v(z)\}\quad \text{ for all } z\in s \mathbb D.
\]
\end{lem}
\begin{proof}
We may identify $V$ as $\mathbb R^n$ and $L(V)$ as $M_n(\mathbb R)$.
We first show that there exists $s>0$ and
smooth functions $v: s\mathbb D  \to V$ and $\lambda: s\mathbb D \to \mathbb R$ satisfying $v(0)=v_0$ and $\lambda(0)=0$, such that 
$
T_z v(z)= \lambda(z) v(z) 
$
for all $z\in s \mathbb D$.
This is likely to be very well known and also appears in an unpublished note of Kazdan \cite{Kazdan}, but we include it here for completeness. 

Let
$
F(v,\lambda,z):=\left[
\begin{smallmatrix}
f(v,\lambda,z)\\
g(v,\lambda,z)
\end{smallmatrix}\right]
:=
\left[
\begin{smallmatrix}
T_z v - \lambda v\\
v_0^T v - 1
\end{smallmatrix}\right].
$
Then,
$
F'(v,\lambda,z):=\left[
\begin{smallmatrix}
\frac{\partial f}{\partial v} & \frac{\partial f}{\partial \lambda} \\
\frac{\partial g}{\partial v} & \frac{\partial g}{\partial \lambda} 
\end{smallmatrix}\right]
=\left[
\begin{smallmatrix}
T_z - \lambda & -v \\
v_0^T & 0 
\end{smallmatrix}\right].
$
To apply the implicit function theorem we must show that $F'(v(0),\lambda(0),0)$ is invertible, or equivalently that its kernel is trivial. Let $\bigl[\begin{smallmatrix}
u \\
\alpha 
\end{smallmatrix}\bigr] \in \ker F'(v(0),\lambda(0),0)$. Then  
$T_0 u - \alpha v_0 = 0$  and also $v_0^T u=0$.  If $\alpha \neq 0$, then 
by the first equality $T_0^2 u = \alpha T_0 v_0 =0$ while $T_0 u \neq 0$, which is impossible as $0$ is a simple eigenvalue of $T_0$.
Thus
$\alpha=0$, so by the first equality we have $u\in \ker T_0 = {\rm{span}}\{v_0\}$. Therefore, the second equality implies that $u=0$. The implicit function theorem now gives a neighborhood $s\mathbb D$ of the origin and smooth functions $v: s\mathbb D  \to V$ and $\lambda: s\mathbb D \to \mathbb R$ satisfying $v(0)=v_0$ and $\lambda(0)=0$, such that 
$
T_z v(z)= \lambda(z) v(z)$
for all $z\in s \mathbb D$.

Finally, we need to show that $\lambda$ must be the zero function. Let $p_{T_z}(x):={\rm{det}} (x I - T_z)$
be the characteristic polynomial of $T_z$, and set $p(z,x):=p_{T_z}(x)$. Since $0$ is a simple eigenvalue of $T_0$ we have that 
$\frac{\partial p}{\partial x}(0,0) \neq 0$. By the implicit function theorem, there exists a neighborhood $s\mathbb D$ of the origin and a function $\mu:s\mathbb D \to \mathbb R$ such that inside a neighborhood of $(0,0) \in \mathbb D \times \mathbb R$ the curve $(z,\mu(z))$ contains {\em all} solutions to $p(z,x)=0$, so $\lambda(z) = \mu(z)$ in a neighborhood of $0$. 
In particular, as we know that $(z,0)$ is a solution we find that $\mu = \lambda$ must be zero near the origin.
\end{proof}

In \cite[Theorem 2.3]{EH78}, the authors show --- as part of their generalization of the Perron--Frobenius theorem --- that for a positive irreducible linear map $\Theta$ on $M_n$ with spectral radius $r$, there exists a unique strictly positive matrix $A$ such that $\Theta(A)=rA$. It is interesting (and  also useful, as we shall see in the next Lemma) to note that $r$ is not only of geometric multiplicity $1$, but also of algebraic multiplicity $1$ (i.e., simple).

To see this, equip $M_n$ with the Hilbert--Schmidt inner product, namely set $\langle Z , W \rangle:={\rm{tr}}(W^*Z)$, and let $\Theta^*$ be the adjoint of $\Theta$ with respect to this inner-product. Since  $\Theta^*$ is positive and irreducible as well \cite{EH78}, there exists a unique strictly positive matrix $B$ such that $\Theta^*(B)=r'B$ where $r'$ is the spectral radius of $\Theta^*$. Note that $B^\perp:=\{Z\in M_n : \langle Z , B \rangle = 0\}$ is $\Theta$ invariant.

If we show that $M_n$ is the direct sum of the two $\Theta$-invariant spaces $\mathbb C A$ and $B^\perp$, then we are clearly done. To this end, note that 
$\langle A , B \rangle = {\rm{tr}}(B^*A) =  {\rm{tr}}\big((A^{\frac 12})^*BA^{\frac 12}\big)$. Since $B$ is strictly positive, all its eigenvalues are positive. By the complex version of Sylvester's law of inertia, all the eigenvalues of the latter matrix are positive, so that $\langle A , B \rangle >0$. Thus, $A \not \in B^\perp$.

\begin{lem} \label{lem:local_analytic_similarity}
Let $g \colon \D \to \M_d(n)$ be an analytic function, such that $g(z)$ is irreducible and similar to a coisometry for every $z \in \D$, then there exists an irreducible coisometry $X$, such that $g(z)$ is similar to $X$, for every $z \in \D$.
\end{lem}
\begin{proof}
Since for every $z \in \D$, $g(z)$ is irreducible and $\rho(g(z)) = 1$ we know that $\Psi_{g(z)}$ is irreducible with spectral radius $1$. The discussion prior to this Lemma now implies that $1$ is a simple eigenvalue of $\Psi_{g(z)}$ with a positive eigenvector. 
Note that $\Psi_{g(z)} - {\rm{id}}_{M_n}$ maps the real vector space of self-adjoint matrices $(M_n)_{sa}$ into itself, and set $T_z:=(\Psi_{g(z)} - {\rm{id}}_{M_n})|_{(M_n)_{sa}}$.  Choose a positive $A_0 \in \ker T_0$. By Lemma \ref{lem:implicit}, there exists $s>0$ and a smooth function $a: s\mathbb D \to (M_n)_{sa}$ satisfying $a(0)=A_0$ such that
\[
\Psi_{g(z)}a(z)=a(z) \quad \text{for all }z \in s\mathbb D.
\]

%
%Furthermore, since $g(z)$ is always similar to a coisometry we can consider the matrix-valued function $\Psi_{g(z)} - {\rm{id}}_{M_n}$. 
%Let $\left(\ker(\Psi_{g(z)} - {\rm{id}}_{M_n})\right)_{sa}$ be the set of all self-adjoint matrices in $\ker(\Psi_{g(z)} - {\rm{id}}_{M_n})$. 
%
%Choose a positive $A_0 \in \left(\ker (\Psi_{g(0)} - {\rm{id}}_{M_n})\right)_{sa}$.
%By \cite[]{LaxBook}
%
%By the above consideration this real vector space is of dimension $1$. We have thus obtained a real line bundle over $\D$ which is trivial since $\D$ is a contractible paracompact space \cite[Corollary 4.8]{Husemoller_Bundles}.
%
%Choose a positive $A \in \left(\ker (\Psi_{g(0)} - {\rm{id}}_{M_n})\right)_{sa}$. Since our line bundle is trivial there exists a smooth section $a(z)$ over $\mathbb D$ satisfying $a(0)=A$.
Since $a(0) > 0$ there exists some $r>0$ such that $a(z)>0$ for every $z\in \ol{r\D}$.
%$\det a(z) > 0$. Since the eigenvector associated to the spectral radius of $\Psi_{g(z)}$ is a scalar multiple of a positive definite matrix, we conclude that $a(z) > 0$, for every $z \in s\D$.
%Since the eigenvalue is simple the function $a(z)$ that maps $z$ to the a non-zero eigenvector is smooth on a neighbourhood of the origin. 
%Now consider $\Psi_{g(z)} - \rho(g(z)) I$ as a map of the trivial vector bundle $M_n(\C) \times s\D$ to itself, where $s > 0$ is such that the $\rho(g(z))$ is smooth on $s\D$. 
%Since $g(z)$ is always irreducible, the kernel is one dimensional, and in fact forms a line bundle on the disc, which is trivial. 
%Now choose a non-vanishing section of this line bundle. 
%Alternatively, one can apply the implicit function theorem again and shrinking $s$ further if necessary obtain the desired smooth function $a(z)$, which is strictly positive on a neighbourhood of the origin. 
%Now fix $0 < r < s$, where $a(z)$ is always strictly positive. 
Since the minimal eigenvalue of $a$ is bounded away from $0$ on the circle $|z| = r$, by \cite[Corollary III.2.1]{ClaGoh81} we can find a holomorphic function $h$ on the disc $r \D$ and continuous on the circle, such that $h(r e^{it}) h(r e^{it})^* = a(r e^{it})$ for every $t \in [0, 2\pi)$. 
There are two things to note regarding the result cited. First, the result discusses right factorization, but it is equivalent to the left factorization by just factoring $a(z)^T$ as Clancey and Gohberg indicate. Secondly, it is immediate that $a(z)$ has entries in the Wiener algebra since it is smooth.
%To see it we apply the Wiener-Masani theorem to $a$ and find an outer function $h$ that satisfies the above equality (see \cite{Hel64} for details and the original paper \cite{WieMas57}).

Define $\widetilde{g}(z) = h(z)^{-1} g(z) h(z)$. 
Note that for $|z|=r$, 
\[\widetilde{g}(z) \widetilde{g}(z)^* = h(z)^{-1} \Psi_{g(z)}(a(z)) (h(z)^*)^{-1} = I_n ,
\]
so $\widetilde{g}(z)$ is a coisometry. 
Now think of $\widetilde{g}$ as a function taking values in $n \times nd$ matrices, and complete $\widetilde{g}$ to a function $F$ with values in $nd \times nd$ matrices by adding rows of zeroes. Note that for every $z \in r\D$, the $n$ largest singular values of $F(z)$ are the singular values of $\widetilde{g}(z)$. By \cite[Theorem 1]{Aup97} the product of the $n$ largest singular values of $F(z)$ is a subharmonic function on $r\D$ that we shall denote by $\sigma(z)$. 
Since each singular value of $\widetilde{g}(z)$ is bounded by $1$ (the norm is the largest singular value), $0 \leq \sigma(z) \leq 1$. 
Furthermore, for every $t \in [0,2\pi)$ we have $\sigma(r e^{it}) = 1$, since the boundary values are coisometries. 
Since each $\widetilde{g}(z)$ is similar to a coisometry, we know that its singular values are non-zero (for otherwise, $\widetilde{g}(z)^*$ would have had a kernel, and similarity is implemented by a base change of specific form in the domain and range of the linear map). 
We conclude that $\sigma$ is bounded away from zero on $r \overline{\D}$. 
Therefore the function $1/\sigma$ is also subharmonic, since a composition of a subharmonic function with a convex function is subharmonic. 
However, $1/\sigma$ is also bounded by $1$ and thus $\sigma$ is identically $1$ on $r \overline{\D}$. 

Now recall that for each $z \in r\D$, the singular values of $\widetilde{g}$ are non-negative numbers bounded from above by $1$, and their product is $1$, hence they are all $1$, and by singular value decomposition, $\widetilde{g}(z)$ is a coisometry. 
Now it remains to apply Lemma \ref{lem:func_to_coisom_constant} to conclude that $\widetilde{g}$ is constant.

To conclude the proof, note that since $X = \widetilde{g}(0)$ is irreducible its similarity orbit is an algebraic subvariety of $\M_d(n)$, hence there is a finite number of polynomials $p_1,\ldots,p_k$, that cut out this orbit. 
By the above discussion the functions $p_1(g(z)),\ldots,p_k(g(z))$ are identically zero on $r \D$ and thus on all of $\D$. Hence, we can conclude that $g$ lands in the similarity orbit of $X$.
\end{proof}

The next proposition is a version of the equality clause in the classical Schwartz lemma.

\begin{prop} \label{prop:coisometric_derivative}
If $f \colon \D \to \widetilde{\fB_d}(n)$ is an analytic function is such that $f(0) = 0$ and $f'(0)$ is an irreducible coisometry, then $f(z)$ is similar to $z f'(0)$, for every $z \in \D$.
\end{prop}
\begin{proof}
Since $f'(0)$ is a coisometry, then $\rho(f'(0)) = 1$ and thus equality holds in the Vesentini--Schwartz lemma. Let $g(z) = f(z)/z$, then $\rho(g(z)) = 1$ for every $z \in \D$ and $g(0) = f'(0) = X$ is an irreducible coisometry. Since the irreducible points are open there exists $r > 0$, such that for every $|z| < r$ we have $g(z)$ is irreducible and thus similar to a coisometry by Lemma \ref{lem:similar_to_coisom}. Now by Lemma \ref{lem:local_analytic_similarity} $g$ takes values in the similarity orbit of $X$. Now since $f(z) = z g(z)$ we are done. 
\end{proof}

The following theorem is a spectral version of the Cartan uniqueness theorem. Let $X \in \fB_d$. we will define the {\em spectrum} of $X$ to be its Jordan--H\"{o}lder componenets and denote it by $\sigma_{JH}(X)$.

\begin{thm} \label{thm:spectral_cartan}
Let $G \colon \widetilde{\fB_d} \to \widetilde{\fB_d}$ be an analytic nc map, such that $G(0) = 0$ and $\Delta G(0,0) = I$. Then $\sigma_{JH}(G(X)) = \sigma_{JH}(X)$ for every $X \in \fB_d$. In particular, if $X$ is irreducible, then $G(X)$ is similar to $X$.
\end{thm}
\begin{proof}
Assume first that $X \in \fB_d$ is an irreducible point, such that $X/\|X\|$ is a coisometry. Consider the function $f \colon \D \to \widetilde{\fB_d}$, defined by $f(z) = G(z X/\|X\|)$. By our assumption $f'(0) = \Delta G(0,0)(X/\|X\|) = X/\|X\|$. Applying Proposition \ref{prop:coisometric_derivative} we have that $f(z)$ is similar to $z X/\|X\|$. 

Now if $X$ is an irreducible point, by Lemmas \ref{lem:irreducible_radius} and \ref{lem:similar_to_coisom}, we know that $X$ is similar to a point $Y$, such that $Y/\|Y\|$ is coisometric. Since $G$ is nc $G(X)$, is similar to $G(Y)$, but by the previous paragraph $G(Y)$ is similar to $Y$ and we are done with the case of irreducible points.

Now for arbitrary $X$, we know that $X$ is similar to a block upper triangular form with irreducible blocks on the diagonal. Since $G$ is nc the result follows from the irreducible case.
\end{proof}

\begin{rem}
The theorem implies that such a $G$ preserves the similarity orbit of the semi-simple part of every $X$ and thus induces the identity map on the GIT quotients $\widetilde{\fB_d}(n)//\operatorname{PGL}_n$. Hence the resemblance with Cartan's uniqueness theorem. The case for $d=1$ was proved in \cite{RanWhi91} in wider generality, namely without assuming equivariance.
\end{rem}

\section{The free pseudo-hyperbolic distance} \label{sec:ph_dist}

Let $X, Y \in \widetilde{\overline{\fB_d}}(n)$. We define the {\em free pseudo-hyperbolic distance} between $X$ and $Y$ by $\ph(X,Y)  = \|\Phi_X - \Phi_Y\|_{cb}$, where $\Phi_X$ and $\Phi_Y$ are norm continuous completely contractive representations of $A(\fB_d)$ on $\C^n$ associated to $X$ and $Y$, respectively. 
If $X \in \widetilde{\overline{\fB_d}}(n)$ and $Y \in \widetilde{\overline{\fB_d}}(m)$, then we set $\ph(X,Y) = \ph(X^{\oplus m}, Y^{\oplus n})$.

One can extend this distance to representations of $\cH^{\infty}(\fB_d)$ on finite dimensional Hilbert spaces. For the points in the interior of the ball that correspond to weak-$*$ continuous representations, these distances coincide.
Therefore, for the interior of the ball we may consider either $H^{\infty}(\fB_d)$ or $A(\fB_d)$.

\begin{prop} \label{prop:pseudhyp_properties}
\begin{enumerate}[(i)]
\item The free pseudo-hyperbolic distance is a pseudo-metric on $\widetilde{\overline{\fB_d}}$ that is independent of the choice of the level.

\item For every two points $X,Y \in \fB_d$ we have $\ph(X,Y) < 2$.
%\item Every free holomorphic function $F \colon \fB_d \to \fB_d$ is a contraction with respect to the free pseudo-hyperbolic distance. Furthermore, it is an isometry if and only if $F$ is an automorphism of the noncommutative ball.

\item We have the following nc properties of the pseudo-hyperbolic distance:
\begin{itemize}
\item $\ph(X^{\prime}\oplus X^{\prime\prime}, Y^{\prime} \oplus Y^{\prime\prime}) = \max\{\ph(X^{\prime}, Y^{\prime}), \ph(X^{\prime\prime}, Y^{\prime\prime})\}$.

\item If $X,Y \in \widetilde{\fB_d}(n)$ and $S \in \GL_n$, such that $S^{-1} X S, S^{-1} Y S \in \fB_d$, then:
\[
\ph(S^{-1} X S, S^{-1} Y S) \leq \|S\|\|S^{-1}\| \ph(X,Y).
\]
\end{itemize}
\end{enumerate}
\end{prop}
\begin{proof}

$(i)$ We think of $\overline{\fB_d}(n)$ as a subset of $\CB(A(\fB_d),M_n)$, the set of all completely bounded homomorphism from $A(\fB_d)$ into $M_n$, via the correspondence $X \mapsto \Phi_X$. Thus, in fact, each level is identified with a subset of the unit sphere in this Banach space and we use the operator space structure on the operator space dual of $A(\fB_d)$ to induce $\ph$.
By \cite[Proposition 7.14]{KVBook} we immediately see that it is a pseudo-metric, with only the direct sum ampliations of the same point identified.

$(ii)$ For every $f \in M_n(H^{\infty}(\fB_d))$, such that $f(0) = 0$ and $\|f\|_{\infty} < 1$, Popescu's free Schwarz lemma \cite{Popescu09} says that $\|f(X)\| \leq \|X\|$.
This implies immediately that $\ph(0,X) = \|X\|$ for every $X \in \fB_d$. Now from the triangle inequality we conclude that for every two points $X, Y \in \fB_d$ we have $\ph(X,Y) < 2$.

%$(iii)$ Every free holomorphic function $F \colon \fB_d \to \fB_d$ gives rise to a completely contractive homomorphism $\alpha_F \colon H^{\infty}(\fB_d) \to H^{\infty}(\fB_d)$ given by $f \mapsto f \circ F$ (this is a special case of Proposition \ref{prop:holo=>cb-homo}).
%Thus we have immediately that for every $X,Y \in \fB_d(n)$:
%\[
%\ph(F(X),F(Y)) = \|\Phi_{F(X)} - \Phi_{F(Y)}\|_{cb} = \|(\Phi_X - \Phi_Y)\circ \alpha_F\|_{cb} \leq \|\Phi_X - \Phi_Y\|_{cb} = \ph(X,Y)
%\]
%Now if $F$ is an automorphism then it induces a completely isometric automorphism on $H^{\infty}(\fB_d)$ and hence the above inequality is an equality.
%On the other hand, suppose that $F$ is an isometry.
%$F$ maps scalar points to scalar points, so by composing with an automorphism of the ball, we may assume that $F(0) = 0$.
%Since $\ph$ is a metric on every level and $F$ is graded, we see that $F$ restricts to an isometry of the compact metric space $\ol{r\fB_d(n)}$ into itself, for every $n$ and every $r<1$.
%As such, the restriction of $F$ to $\ol{r\fB_d(n)}$ is also surjective for all $n$ and $r$, and we conclude that $F$ is bijective.
%It is then straightforward that $F^{-1}$ must also be free holomorphic.

$(iii)$ This is immediate from the fact that the completely bounded norm is an operator space structure on the dual of $A(\fB_d)$. %$H^{\infty}(\fB_d)$.
\end{proof}

%\begin{rem}
%We can now define our analogue of the hyperbolic metric on the noncommutative ball as in the classical case by $d(X,Y) = 2 \log \frac{2 + \ph(X,Y)}{2 - \ph(X,Y)}$.
%\end{rem}

\begin{prop}\label{prop:3equiv}
Let $\fV \subseteq \fB_d$ and $\fW \subseteq \fB_e$ be nc varieties, and $G \colon \widetilde{\fW} \to \widetilde{\fV}$ an nc holomorphic function. Then the following statements are equivalent:
\begin{enumerate}[(i)]
\item $G$ is Lipschitz with respect to $\ph$,
\item $\sup_{W\in \fW}\|\Phi_{G(W)}\|_{cb} < \infty$, ~ and
\item the mapping $f \mapsto \widetilde{f} \circ G$ for $f \in \cH^\infty(\fV)$ is a well-defined completely bounded homomorphism $\alpha:H^\infty(\fV) \to H^\infty(\fW)$.
\end{enumerate}
Furthermore, in this case
\begin{enumerate}[(a)]
\item $\|\alpha \|_{cb} = \sup_{W\in \fW}\|\Phi_{G(W)}\|_{cb}$,
\item the value in (a) is a Lipschitz constant for $G$, ~ and
\item $\alpha$ is weak-$*$ continuous.
\end{enumerate}
\end{prop}
\begin{proof}
(i) $\implies$ (ii).
If $G$ is $K$-Lipschitz, choose some $W_0 \in \fW$. Then
\[
\|\Phi_{G(W)}\|_{cb} 
	\leq	\|\Phi_{G(W)} - \Phi_{G(W_0)}\|_{cb} + \|\Phi_{G(W_0)}\|_{cb} 
	\leq	K \|\Phi_{W} - \Phi_{W_0}\|_{cb} + \|\Phi_{G(W_0)}\|_{cb} 
	\leq	2K + \|\Phi_{G(W_0)}\|_{cb}.
\]
Thus, $\sup_{W \in \fW} \|\Phi_{G(W)}\|_{cb}<\infty$. 

(ii) $\implies$ (iii).
If $\sup_{W\in \fW}\|\Phi_{G(W)}\|_{cb} < \infty$, then for every $f =\big( f_{ij}\big) \in M_n(\cH^{\infty}(\fV))$
\[
\begin{split}
\sup_{W \in \fW} \|(\widetilde f\circ G)(W)\|
	&=		\sup_{W \in \fW} \left\|\big(\widetilde f_{ij}(G(W))\big)\right\| 		\\
	&=		\sup_{W \in \fW} \left\|\Phi_{G(W)}\big( f_{ij}\big)\right\|	\\
	&\leq	\sup_{W \in \fW} \|\Phi_{G(W)}\|_{cb} \|f\|_{M_n(H^{\infty}(\fV))}.
\end{split}
\]
Thus, $\widetilde{f} \circ G$ is in $M_n(H^\infty(\fW))$,  the map $\alpha: f \mapsto \widetilde{f} \circ G$ is a completely bounded homomorphism, and $\|\alpha\|_{cb} \leq \sup_{W\in \fW}\|\Phi_{G(W)}\|_{cb}$.

(iii) $\implies$ (i).
As $\Phi_{G(W)} = \alpha^*(\Phi_W)$ and $\alpha$ is completely bounded, $G$ is $\|\alpha\|_{cb}$-Lipschitz.\\

The last argument shows, in addition, that $\sup_{W \in \fW} \|\Phi_{G(W)}\|_{cb} \leq \|\alpha\|_{cb}$, so we have (a) and (b). To show (c), we need to use following two facts: (1) by Theorem \ref{thm:weak-$*$_and_wot_agree} the \textsc{wot} and weak-$*$ topologies coincide in $H^\infty(\fV)$ and $H^\infty(\fW)$ and (2) for bounded nets in $H^\infty(\fV)$, \textsc{wot}-convergence and pointwise convergence are the same \cite[Lemma 2.5]{SalShaSha17}. 
From the first fact, we see (using the Krein-Sm\u ulian Theorem) that it suffices to check that $\alpha$ is weak-$*$/\textsc{wot} continuous on the closed unit ball of $H^\infty(\fV)$. 
This follows readily from the second fact, since $\alpha$ is implemented by composition. 
\end{proof}

\begin{cor}\label{cor:holo=>contraction}
Let $\fV \subseteq \fB_d$ and $\fW \subseteq \fB_e$ be nc varieties, and $G : {\fW} \to {\fV}$ an nc holomorphic map. 
Then $G$ is a contraction with respect to the pseudo-hyperbolic metric. In particular, if $G$ is a biholomorphism, then it is isometric.
\end{cor}
\begin{proof}
Since $G({\fW})\subseteq {\fV}$, we have that $\sup_{W \in \fW} \|\Phi_{G(W)}\|_{cb}\leq 1$. By Proposition \ref{prop:3equiv}, $G$ is $1$-Lipschitz, so it is a contraction.
\end{proof}

\begin{prop}\label{prop:isometric_biholo}
Let $G:\fB_d \to \fB_d$ be an nc holomorphic map, and assume $d< \infty$. 
Then $G$ is an automorphism of $\fB_d$ if and only if it is isometric with respect to the pseudo-hyperbolic distance.
\end{prop}
\begin{proof}
If $G$ is an automorphism, by Corollary \ref{cor:holo=>contraction}, it is isometric.
On the other hand, suppose that $G$ is an isometry.
$G$ maps scalar points to scalar points, so by composing with an nc automorphism of the ball $\fB_d$, we may assume that $G(0) = 0$. 
To see that there is indeed no harm in this assumption, recall that $\operatorname{Aut}(\fB_d) = \operatorname{Aut}(\bB_d)$ (see \cite[Theorem 2.8]{Popescu10}) and every such automorphism gives rise to a completely isometric automorphism of $H^\infty(\fB_d)$ (see Propositions 6.5 and 6.8 in \cite{SalShaSha17}), therefore an automorphism of the ball is isometric with respect to the pseudo-hyperbolic distance. 

So suppose that $G$ is an isometry and that $G(0) = 0$. 
Since $\ph$ is a metric on every level and $G$ is graded, we see that $G$ restricts to an isometry of the compact metric space $\ol{r\fB_d(n)}$ into itself, for every $n$ and every $r<1$.
As such, the restriction of $G$ to $\ol{r\fB_d(n)}$ is also surjective for all $n$ and $r$, and we conclude that $G$ is bijective.
It is then straightforward that $G^{-1}$ must also be free holomorphic.
\end{proof}

One can define another free pseudo-hyperbolic distance: 
\[
\phb(X,Y)=\|\Phi_X - \Phi_Y\|.
\]
Evidently, $\phb \leq \ph$, and all results above have counterpart versions for $\phb$. 
It is interesting to note that, since Proposition \ref{prop:isometric_biholo} holds with $\phb$ instead of $\ph$, it follows that a self map of $\fB_d$ is isometric with respect to $\ph$ if and only if it is isometric with respect to $\phb$. 
We record the $\phb$-version of Proposition \ref{prop:3equiv} for later use. 
In fact, when the distance $\phb$ is used, we get a somewhat sharper result.

\begin{prop}\label{prop:4equiv_prime}
Let $\fV \subseteq \fB_d$ and $\fW \subseteq \fB_e$ be nc varieties, and $G \colon \widetilde{\fW} \to \widetilde{\fV}$ an nc holomorphic function. Then the following statements are equivalent:
\begin{enumerate}[(i)]
\item $G$ is Lipschitz with respect to $\phb$,
\item $\sup_{W\in \fW}\|\Phi_{G(W)}\| < \infty$,
\item the mapping $f \mapsto \widetilde{f} \circ G$ for $f \in \cH^\infty(\fV)$ is a well-defined bounded homomorphism $\alpha : H^\infty(\fV) \to H^\infty(\fW)$, ~ and
\item[(iii)'] the mapping $f \mapsto \widetilde{f} \circ G$ for $f \in \cH^\infty(\fV)$ is a well-defined homomorphism $\alpha :  H^\infty(\fV) \to H^\infty(\fW)$.
\end{enumerate}
Furthermore, in this case
\begin{enumerate}[(a)]
\item $\|\alpha \|= \sup_{W\in \fW}\|\Phi_{G(W)}\|$,
\item the value in (a) is a Lipschitz constant for $G$, ~ and
\item $\alpha$ is weak-$*$ continuous.
\end{enumerate}
\end{prop}
\begin{proof}
The proof of (i) $\implies$ (ii) $\implies$ (iii) $\implies$ (i), and the proof of $(a)$, $(b)$, and $(c)$ (in case (i)--(iii) holds, is the same as in Proposition \ref{prop:3equiv}. 
We also obviously have (iii) $\implies$ (iii)'. We shall show (iii)' $\implies$ (ii).
Suppose that $\alpha$ is a well defined map from $H^\infty(\fV)$ into $H^\infty(\fW)$ (it is clear in this case that it is a unital homomorphism).
Then $\widetilde{f} \circ G$ is a bounded function on $\fW$ for all $f \in H^\infty(\fV)$, so we have
\[
\sup_{W \in \fW}\|\Phi_{G(W)}(f)\| = \sup_{W \in \fW} \|\widetilde{f}(G(W))\| = \|\widetilde{f}\circ G\|_\infty < \infty .
\]
Since this holds for every $f \in H^\infty(\fV)$, the uniform boundedness principle implies (ii).
\end{proof}

\begin{rem}
It is tempting to consider a version of Gleason parts with respect to the free pseudo-hyperbolic metric. The immediate question is whether the relation $\ph(X,Y) <2$ is an equivalence relation. Classically, one can define the Gleason parts of a function algebra using either a variant of the pseudo-hyperbolic metric or a form of Harnack inequality. Harnack domination was defined by Suciu (\cite{Suc97}) for the case $d =1$ and by Popescu for $d > 1$ (\cite{Popescu09}). Therefore, it is also natural to ask whether there is a connection between the Harnack parts as defined by Popescu and the pseudo-hyperbolic metric. Unfortunately, the answer to both questions is negative.

To see that the ``Gleason parts'' are not parts at all consider the following example. Let $A = \left(\begin{smallmatrix} 0 & 1 \\ 0 & 0 \end{smallmatrix} \right)$ and $S = \left(\begin{smallmatrix} 2 & 0 \\ 0 & 1 \end{smallmatrix} \right)$, then $S^{-1} A S = \left(\begin{smallmatrix} 0 & 1/2 \\ 0 & 0 \end{smallmatrix} \right)$. Let $f \in M_k(A(\D))$, then
\[
\|f(S^{-1} A S) - f(A)\| = \|I_{M_k} \otimes(T_S - I_{M_2})(A)\| \leq \|T_S - I_{M_2}\|.
\]
Here $T_S$ is the operator on $M_2$ defined by $T_S(B) = S^{-1} B S$. 
Note that since $S$ is diagonal, $T_S$ is in fact a Schur multiplier, so is $T_S - I_{M_2}$ is too. 
In particular, $(T_S - I_{M_2})(B) = C \circ B$, where $C = \left(\begin{smallmatrix} 0 & -1/2 \\ 1 & 0 \end{smallmatrix} \right)$. Therefore, $\|T_S - I_{M_2}\| = 1$, and we can conclude that for every $k \in \N $ and $f \in M_k(A(\D))$ we get
\[
\| \Phi_{S^{-1} A S}(f) - \Phi_A(f)\| \leq \|T_S - I_{M_2}\| < 2.
\]
So $A$ is ``equivalent'' to an interior point and so is $-A$, however, the distance between them is at least $2$, as can be observed by taking $f(z)=z$. Therefore, this is not an equivalence relation. 

Harnack domination is an equivalence relation, as proved by Popescu, and the open ball is a part. One can show, just as in the classical case, that if $X$ Harnack dominates $Y$, then $\ph(X,Y) < 2$, however as we have seen above, the converse is false.
\end{rem}

\begin{rem}
Recently, a more general approach to noncommutative hyperbolic geometry was formulated by Belinschi and Vinnikov in \cite{BelVin18}. They define a noncommutative version of the Lempert function for an nc domain $\Omega$. Let $X \in \Omega(n)$ and $Y \in \Omega(m)$ and $Z \in M_{n,m}(\C) \otimes \C^d$, then 
\[
L(X,Y)(Z)^{-1} = \sup \left\{t \in [0,\infty] ~:~ \begin{bmatrix} X & sZ \\ 0 & Y \end{bmatrix} \in \Omega, \text{ for all } s \in [0,t] \right\}.
\]
Unfortunately, in the case of $\Omega = \widetilde{\fB}_d$ we get that $L$ is identically $0$, since the similarity envelope of the ball consists precisely of points of spectral radius strictly less than $1$ and by Lemma \ref{lem:irreducible_radius} we know that the spectral radius of an block upper triangular matrix is the same as the maximum of the spectral radii of its diagonal blocks. 
This observation also tells us that there are many copies of $\C$ embedded into  every level except the first. 
This is an indication of the fact that the nice hyperbolic structure of the ball that is used in the previous works cannot be extended to the similarity envelope. 
That being said, the results of Section \ref{sec:prelim_jspr} show that a modicum of hyperbolic-like properties is preserved, when one passes to the similarity envelope, if one uses the joint spectral radius. 
\end{rem}

%%%%%%%%%%%%%%%%%%%%%%%%%%%%%%%%%%%%%%%%%%%%%%%
%%%%%%%%%%%%%%%%%%%%%%%%%%%%%%%%%%%%%%%%%%%%%%%
\section{Classification up to weak-$*$ continuous isomorphism} \label{sec:isomorphisms}

Let $\fV \subseteq \fB_d$ and $\fW \subseteq \fB_{e}$ be nc varieties.
In this section, we allow for any $d,e \in \mathbb N \cup\{\infty\}$, unless we state otherwise.
An nc holomorphic map $G:\widetilde{\fB_{e}} \to \mathbb M_d$ is called an {\em nc biholomorphism} from $\widetilde{\fW}$ onto $\widetilde{\fV}$ if there exists an nc holomorphic map $F:\widetilde{\fB_{d}} \to \mathbb M_e$ such that
\[
F\circ G|_{\widetilde{\fW}}={\rm id}_{\widetilde{\fW}} \quad \text{and} \quad
G\circ F|_{\widetilde{\fV}}={\rm id}_{\widetilde{\fV}}.
\]
In this case, and whenever there is no confusion, we let $G^{-1}$ denote the map $F$.
The goal of this section is to classify the algebras of bounded analytic functions on nc subvarieties of $\fB_d$ in terms of biholomorphisms between the respective similarity invariant envelopes of the varieties.

If $f$ is an nc holomorphic function on $\fV$, then we write $\widetilde{f}$ for the unique extension of $f$ to its similarity envelope $\widetilde{\fV}$.

\begin{rem}
If $\alpha : H^\infty(\fV) \to H^\infty(\fW)$ is a unital and bounded homomorphism, then the adjoint $\alpha^* : H^\infty(\fW)^* \to H^\infty(\fV)^*$ is also bounded.
Moreover, there is also a natural adjoint mapping $\alpha^*$ between representation spaces,
\[
\alpha^* : \sqcup_n \Rep_{n}(H^\infty(\fW)) \to \sqcup_n \Rep_{n}(H^\infty(\fV)),
\]
given by $\alpha^* (\rho) = \rho\circ \alpha$.
If $\alpha$ is completely contractive/bounded, then $\alpha^*$ preserves completely contractive/bounded representations.
\end{rem}
%%%%%%%%%%%%%%%%%%%%%%%%%%%%%%%%
\begin{prop}\label{prop:cb-homo=>holo}
Let $\fV \subseteq \fB_d$ and $\fW \subseteq \fB_{e}$ be nc varieties.
Let $\alpha : H^\infty(\fV) \to H^\infty(\fW)$ be a unital bounded homomorphism.
If $\alpha^*$ maps weak-$*$ continuous finite dimensional representations to weak-$*$ continuous representations, then there exists an nc holomorphic map $G: \widetilde{\fW} \to \widetilde{\fV}$ 
which implements $\alpha$ by the formula
\[
\alpha(f) = \widetilde{f} \circ G.
\]
In this case, $\|\alpha\|=\sup_{W\in \fW}\|\Phi_{G(W)}\|$.
If $\alpha$ is, in addition, completely bounded, then $\|\alpha\|_{cb}=\sup_{W\in \fW}\|\Phi_{G(W)}\|_{cb}$.
\end{prop}
\begin{proof}
Define $G:\widetilde{\fW} \to \widetilde{\ol{\fB_d}}$ by
\[
G(W)=\pi_k(\alpha^*(\Phi_W)), \quad W \in \widetilde{\fW}(k) .
\]
By the assumption on $\alpha^*$ and Theorem \ref{thm:finite_dim_reps}, we find that $G$ maps $\widetilde{\fW}$ into $\widetilde{\fV}$. Note that
\[
\widetilde{f}(G(W)) = \Phi_{G(W)}(f) = \alpha^*(\Phi_W)(f)=\Phi_W(\alpha(f))=\alpha(f)(W)
\]
for all $W \in \fW$, so that $\alpha(f) = \widetilde{f} \circ G$. Propositions \ref{prop:3equiv} and \ref{prop:4equiv_prime} show that $\|\alpha\|=\sup_{W\in \fW}\|\Phi_{G(W)}\|$ and that $\|\alpha\|_{cb}=\sup_{W\in \fW}\|\Phi_{G(W)}\|_{cb}$, if $\alpha$ is completely bounded.
\end{proof}

We record the isomorphism equivalence implied by Propositions \ref{prop:3equiv}, \ref{prop:4equiv_prime} and \ref{prop:cb-homo=>holo}.
\begin{cor}\label{cor:cb-isom<=>biholo_weakstar}
Let $\fV \subseteq \fB_d$ and $\fW \subseteq \fB_{e}$ be nc varieties. Then the following statements are equivalent:
\begin{enumerate}[(i)]
\item $H^\infty(\fV)$ and $H^\infty(\fW)$ are weak-$*$ continuously and completely boundedly isomorphic,
\item the similarity envelopes $\widetilde{\fV}$ and $\widetilde{\fW}$ are nc biholomorphic via a $\ph$-bi-Lipschitz biholomorphism, 
\item the similarity envelopes $\widetilde{\fV}$ and $\widetilde{\fW}$ are nc biholomorphic via an nc biholomorphism $G: \widetilde{\fW} \to \widetilde{\fV}$ satisfying 
\[
\sup_{W\in \fW}\|\Phi_{G(W)}\|_{cb} < \infty \quad \text{and} \quad \sup_{V\in \fV}\|\Phi_{G^{-1}(V)}\|_{cb}<\infty.
\]
\end{enumerate}
Furthermore, in this case $\sup_{W\in \fW}\|\Phi_{G(W)}\|_{cb} = \|\alpha\|_{cb}$ and $\sup_{W\in \fW}\|\Phi_{G^{-1}(W)}\|_{cb} = \|\alpha^{-1}\|_{cb}$, and
\[
\|\alpha^{-1}\|_{cb}^{-1} \ph(W_1,W_2) \leq \ph(G(W_1),G(W_2)) \leq \|\alpha\|_{cb} \ph(W_1,W_2).
\]

Likewise, the following statements are equivalent:
\begin{enumerate}[(i)]
\item $H^\infty(\fV)$ and $H^\infty(\fW)$ are weak-$*$ continuously and boundedly isomorphic,
\item[(i)'] $H^\infty(\fV)$ and $H^\infty(\fW)$ are weak-$*$ continuously isomorphic,
\item the similarity envelopes $\widetilde{\fV}$ and $\widetilde{\fW}$ are nc biholomorphic via a $\phb$-bi-Lipschitz biholomorphism, 
\item the similarity envelopes $\widetilde{\fV}$ and $\widetilde{\fW}$ are nc biholomorphic via an nc biholomorphism $G: \widetilde{\fW} \to \widetilde{\fV}$ satisfying 
\[
\sup_{W\in \fW}\|\Phi_{G(W)}\| < \infty \quad \text{and} \quad \sup_{V\in \fV}\|\Phi_{G^{-1}(V)}\|<\infty.
\]
\end{enumerate}
Furthermore, in this case $\sup_{W\in \fW}\|\Phi_{G(W)}\| = \|\alpha\|$ and $\sup_{W\in \fW}\|\Phi_{G^{-1}(W)}\| = \|\alpha^{-1}\|$, and
\[
\|\alpha^{-1}\|^{-1} \phb(W_1,W_2) \leq \phb(G(W_1),G(W_2)) \leq \|\alpha\|\phb(W_1,W_2)
\]
\end{cor}
%
%\begin{cor}\label{cor:cb-isom<=>biholo_weakstar}
%Let $\fV \subseteq \fB_d$ and $\fW \subseteq \fB_{e}$ be nc varieties. 
%Then  $H^\infty(\fV)$ and $H^\infty(\fW)$ are weak-$*$ continuously isomorphic if and only if the similarity envelopes $\widetilde{\fV}$ and $\widetilde{\fW}$ are nc biholomorphic via an nc biholomorphism $G: \widetilde{\fW} \to \widetilde{\fV}$ satisfying
%\[
%\sup_{W\in \fW}\|\Phi_{G(W)}\| < \infty \quad \text{and} \quad \sup_{V\in \fV}\|\Phi_{G^{-1}(V)}\|<\infty.
%\]
%Furthermore, in this case $\sup_{W\in \fW}\|\Phi_{G(W)}\| = \|\alpha\|$. 
%
%Likewise, $H^\infty(\fV)$ and $H^\infty(\fW)$ are completely boundedly and weak-$*$ continuously isomorphic if and only if the similarity envelopes $\widetilde{\fV}$ and $\widetilde{\fW}$ are nc biholomorphic via an nc biholomorphism $G: \widetilde{\fW} \to \widetilde{\fV}$ satisfying
%\[
%\sup_{W\in \fW}\|\Phi_{G(W)}\|_{cb} < \infty \quad \text{and} \quad \sup_{V\in \fV}\|\Phi_{G^{-1}(V)}\|_{cb}<\infty, 
%\]
%and in this case $\sup_{W\in \fW}\|\Phi_{G(W)}\|_{cb} = \|\alpha\|_{cb}$.
%\end{cor}

\begin{rem}
Note that when $\alpha$ is a completely isometric isomorphism, or equivalently when $\|\alpha\|_{cb}=\|\alpha^{-1}\|_{cb}=1$, then Corollary \ref{cor:cb-isom<=>biholo_weakstar} implies that
$\sup_{W\in \fW}\|\Phi_{G(W)}\|_{cb} = 1$ and $\sup_{V\in \fV}\|\Phi_{G^{-1}(V)}\|_{cb} = 1$.
Since $\|T\| \leq \|\Phi_{T}\|_{cb}$ for every $T$, we obtain that $G$ must map $\fW$  injectively into $\ol\fV$ and $G^{-1}$ maps $\fV$ injectively into $\ol\fW$.
Recalling the nc maximum principle \cite[Lemma 6.11]{SalShaSha17}, it follows that $G$ is an nc biholomorphism between $\fW$ and $\fV$.
Conversely, whenever $G$ is an nc biholomorphism between $\fW$ and $\fV$, we have that
$\sup_{W\in \fW}\|\Phi_{G(W)}\|_{cb} = 1$ and $\sup_{V\in \fV}\|\Phi_{G^{-1}(V)}\|_{cb} = 1$.
Corollary \ref{cor:cb-isom<=>biholo_weakstar} yields a completely bounded isomorphism $\alpha$ with $\|\alpha\|_{cb}=\|\alpha^{-1}\|_{cb}=1$, so that $\alpha$ is in fact a completely isometric isomorphism.
This recovers \cite[Corollary 6.14]{SalShaSha17} (in \cite{SalShaSha17} it was shown that when $d,e <\infty$, then the weak-$*$ continuity condition is redundant).
\end{rem}

\begin{example} \label{ex:DHS}
In \cite[Example 6.6]{DHS14} the authors show that given a separated, but not strongly separated Blaschke sequence $V \subseteq \D$, there exists a bi-Lipschitz biholomorphism  of $V$ onto $W$, where $W$ is an interpolating sequence. 
Since $V$ is not interpolating, the associated multiplier algebras are not isomorphic. The above theorem demonstrated what fails in this case, since bi-Lipschitz on the first level is not enough to conclude that the multiplier algebras are isomorphic.
\end{example}

The following example shows that the assumption $\sup_{W\in \fW}\|\Phi_{G(W)}\| < \infty$ does not hold for every holomorphic map $G$ between similarity invariant envelopes of varieties, and, in fact, it may fail even for automorphisms of $\widetilde{\fB}_d$.
In particular, it shows that not every automorphism of $\widetilde{\fB}_d$ gives rise to an automorphism of $H^\infty(\fB_d)$.

%%%%%%%%%%%%%%%%%%%%%%%%%%%%%%%%%
\begin{example}\label{ex:holo_not_bounded}
On $\fB_2$, consider the function $f(x) = f(x_1, x_2) = (1-x_1)^2$.
The inverse of $f$ is a well defined unbounded nc function, and it extends to $\widetilde{\fB}_2$.
Define $G : \widetilde{\fB}_2 \to \widetilde{\fB}_2$ by
\[
G(x) = f(x)^{-1} x f(x).
\]
It is easy to check that $G$ is biholomorphism of $\widetilde{\fB}_d$.

Now we consider pairs of matrices $X = (X_1, X_2) \in \fB_2(2)$, given by
\[
X_1 = \begin{bmatrix} \lambda & 0\\
0 & \mu \end{bmatrix}
\quad , \quad
X_2 = \begin{bmatrix}  0 & a \\
0 &  0\end{bmatrix} ,
\]
where all entries are assumed positive.
In order to be a strict row contraction, it is necessary and sufficient that $\lambda, \mu <1$ and $a < \sqrt{1-\lambda^2}$.
For definiteness, choose $a = \frac{1}{2} \sqrt{1-\lambda^2} > \frac{1}{2}(1 - \lambda)$.

Now we observe $G(X) = f(X)^{-1} X f(X)$, the second component of which is
\begin{align*}
f(X)^{-1} X_2 f(X) &= \begin{bmatrix} (1- \lambda)^{-2} & 0\\
0 & (1 - \mu)^{-2} \end{bmatrix}
\begin{bmatrix}  0 & a \\
0 &  0\end{bmatrix}
\begin{bmatrix} (1- \lambda)^{2} & 0\\
0 & (1 - \mu)^{2} \end{bmatrix} \\
&= \begin{bmatrix}  0 & a\left(\frac{1 - \mu}{1-\lambda}\right)^{2}  \\
0 &  0\end{bmatrix} .
\end{align*}
By the choice of $a$, we find that
\[
\|f(X)^{-1} X_2 f(X) \| > \frac{1}{2}\frac{(1-\mu)^2}{1 - \lambda}.
\]
Fixing $\mu$ small and letting $\lambda \to 1$, we see that $G$ is not bounded on $\fB_2(2)$.
\end{example}

We record a consequence of our results that adds to what is known (see, for example, \cite{DavPitts2,Popescu10}) about the automorphism group of $H^\infty(\fB_d)$. 

\begin{thm}\label{thm:auto_of_Ld}
Let $d \in \bN$. 
Every automorphism $\alpha$ of $H^\infty(\fB_d)$ is given by 
\[
\alpha(f) = \widetilde{f} \circ G, 
\]
where $G : \widetilde{\fB}_d \to \widetilde{\fB}_d$ is an nc biholomorphism of $\widetilde{\fB}_d$ such that 
\[
\sup_{W\in \fB_d}\|\Phi_{G(W)}\| < \infty \quad \text{and} \quad \sup_{V\in \fB_d}\|\Phi_{G^{-1}(V)}\|<\infty. 
\]
\end{thm}
\begin{proof}
By \cite[Theorem 4.6]{DavPitts2}, every automorphism of $H^\infty(\fB_d)$ is \textsc{wot}-continuous. 
Thus Corollary \ref{cor:cb-isom<=>biholo_weakstar} applies. 
\end{proof}

Recall that in \cite{DavPitts2} an automorphism of $H^{\infty}(\fB_d)$ was called {\em quasi-inner}, if it is in the kernel of the homomorphism to $\Aut(\fB_d(1))$. In terms of the preceding theorem, this means that the induced nc biholomorphism of $\fB_d$ is the identity when restricted to the first level. We can therefore apply the spectral Cartan theorem (Theorem \ref{thm:spectral_cartan}) to obtain the following corollary, that can be interpreted as stating that it is very difficult to distinguish a quasi-inner automorphism from an inner one.

%%%%%%%%%%%%%%%%%%%%%%%% This is a consequence of spectral Cartan %%%%%%%%%%%%%%%%%%%%%%%%%%
\begin{cor} \label{cor:quasi-inner}
If $\alpha \colon H^{\infty}(\fB_d) \to H^{\infty}(\fB_d)$ is a quasi-inner automorphism, then there exists an nc  biholomorphism $G \colon \widetilde{\fB}_d \to \widetilde{\fB}_d$ such that $\alpha(f) = \widetilde{f} \circ G$ for all $f \in H^\infty(\fB_d)$, and such that for every irreducible $X \in \fB_d$, $G(X)$ is similar to $X$.
\end{cor}
%%%%%%%%%%%%%%%%%%%%%%%%%%%%%%%%%%%%%%%%%%%%%%%%%%%%%%%%%%%%%%%%%%%%%%%%%%%%%%%%%%%%%%%%%%%%

%%%%%%%%%%%%%%%%%%%%%%%%%%%%%%%%%%%%%%%%%%%%%%%
\section{The homogeneous case} \label{sec:homogeneous}
%%%%%%%%%%%%%%%%%%%%%%%%%%%%%%%%%%%%%%%%%%%%%%%

In this section, we focus our attention on the case where the underlying varieties are homogeneous. 
Our first result shows that when $\fV$ and $\fW$ are homogeneous, the weak-$*$ continuity requirement can be dropped in Proposition \ref{prop:cb-homo=>holo} and Corollary \ref{cor:cb-isom<=>biholo_weakstar}. 
Throughout this section, we shall assume that $d< \infty$. 

\begin{thm} \label{thm:homog_no_need_w*}
Let $\fV \subseteq \fB_d$ and $\fW \subseteq \fB_e$ be two homogeneous nc varieties. 
Let $\alpha \colon H^{\infty}(\fV) \to H^{\infty}(\fW)$ be a bounded isomorphism. 
Then the induced nc function $G = \pi \circ \alpha^* : \widetilde{\fW} \to \widetilde{\ol{\fB_d}}$ maps $\widetilde{\fW}$ into $\widetilde{\fV}$. 
Consequently, every (completely) bounded isomorphism between $H^\infty(\fV)$ and $H^\infty(\fW)$ arises as the composition with an nc biholomorphism $G : \widetilde{\fW} \to \widetilde{\fV}$. 
\end{thm}
\begin{proof}
We first note that if $\fV = 0$ or $\fW = 0$ (understood as the nc variety over the scalar point $0$), then the problem trivializes. 
Thus we assume henceforth that neither variety consists of just the point $0$. 

For every $n \in \N$ we have that $\alpha^* \colon \Rep_n(H^{\infty}(\fW)) \to \Rep_n(H^{\infty}(\fV))$ is a bijection and that $G = \pi \circ \alpha^*$. 
Suppose towards a contradiction that for some $X \in \fW$, $\alpha^*(\Phi_X)$ is not weak-$*$ continuous. Then $G(X) \in \widetilde{\ol\fB_d}\setminus \widetilde{\fB_d}$, so $\rho(G(X)) = 1$ by Lemma \ref{lem:tilde_is_spectral_ball}. 
Consider the function $f^{(X)}(z) = G(z X/\|X\|)$ defined and analytic on the disc. 
By Corollary \ref{cor:subharmonic} the function $\rho(f^{(X)}(z))$ is subharmonic and thus by the maximum modulus principle it is the constant function $1$. 
In particular $\rho(G(0)) = 1$, i.e, $\|G(0)\| = 1$ as a point in $\ol{\bB_d}$. 
% By the classical maximum modulus, we see that $G$ restricted to the first level is constant. 
This implies that $G$ is constant on $\fW$ as follows. 

Applying a unitary we may assume that $G(0) = (1, 0, \ldots, 0)$. 
Now, let $n \in \mathbb N$, let $W$ be a point in $\fW(n)$, and consider $f^{(W)}_1(z)$ --- the first coordinate of the function $f^{(W)}(z) = (f^{(W)}_1(z), \ldots, f^{(W)}_d(z))$, where $f^{(W)}(z) = G(z W/\|W\|)$.
Note that since $G$ respects direct sums $f^{(W)}_1(0) = I_n$. 
Next, since $G(\fW) \subseteq \widetilde{\ol{\fB_d}}$, we have that $f^{(W)}_1(z)$ is similar to a contraction for every $z \in \mathbb D$. 
Thus, for each $z \in \mathbb D$, the eigenvalues of $f^{(W)}_1(z)$ are all of absolute value less than or equal to $1$, so that $|{\rm{tr}}(f^{(W)}_1(z))| \leq n$. 
As ${\rm{tr}}(f^{(W)}_1(0)) = {\rm{tr}}(I_n) = n$, the classical maximum modulus implies that ${\rm{tr}}(f^{(W)}_1(z))=n$ for every $z\in \mathbb D$, and thus all of the eigenvalues of $f^{(W)}_1(z)$ are all equal to $1$. 
Since for every $z \in \mathbb D$ the matrix $f^{(W)}_1(z)$ is similar to a contraction, it is power bounded. 
Since all of its eigenvalues are $1$, its Jordan form cannot contain a non-trivial Jordan block. 
Thus, $f^{(W)}_1(z)=I_n$ for every $z \in \mathbb D$.
 Since this holds for every point $W \in \fW$, we have that the function $G_1 = \alpha(z_1|_{\fV})$ is identically equal to the constant nc function $1$. 
Furthermore, for every $n \in \N$ and every $W \in \fW(n)$, we have that $G(W)$ is similar to a row contraction, i.e, there exists $S \in \operatorname{GL}_n(\C)$, such that
\[
0 \leq \sum_{j=1}^d (S^{-1} G_j(W) S) (S^{-1} G_j(W) S)^* \leq I_n.
\]
However, $G_1(W) = I$ and thus this forces $G_j(W) = 0$, for $j = 2, \ldots, d$. 
Therefore, $G$ is constant, as we claimed. 

Now recall that from the definition of $G$ as $\pi \circ \alpha^*$, the nc function $G_j$ is equal to $\alpha(z_j|_\fV)$ for every $j = 1, \ldots, d$. 
It follows from the previous paragraph that $\alpha(z_j|_{\fV})$ is equal to a constant nc function on $\fW$, for $j = 1, \ldots, d$. 
Since $\alpha$ is an isomorphism and preserves the unit, we find that the tuple of coordinate functions $(z_1|_\fV, \ldots, z_d|_\fV)$ is  equal to the constant tuple $G(0)$. 
This implies that $\fV$ is the nc variety over the scalar point $G(0)$, which is a contradiction. 

It follows that $G$ maps $\widetilde{\fW}$ into $\widetilde{\fV}$, so by Theorem \ref{thm:finite_dim_reps}, $\alpha^*$ maps weak-$*$ continuous representations in $\Rep_{k}(H^\infty(\fW))$ to weak-$*$ continuous representations in $\Rep_{k}(H^\infty(\fV))$. 
By Proposition \ref{prop:cb-homo=>holo}, the isomorphism $\alpha$ is given by composition with $G$. 
Finally, since we can argue similarly about $\alpha^{-1}$, the map $G$ must be a biholomorphism between $\widetilde{\fW}$ and $\widetilde{\fV}$. 
\end{proof}

We shall now show that for any two homogeneous nc varieties $\fV \subseteq \fB_d$ and $\fW \subseteq \fB_e$, if the similarity envelopes $\widetilde\fV$ and $\widetilde\fW$ are nc biholomorphic, then $\widetilde\fV$ is the image of $\widetilde\fW$ under an invertible linear transformation.
We start with the following lemma.

\begin{lem}\label{lem:DiscToDisc}
Let $\fV \subseteq \fB_d$ and $\fW \subseteq \fB_e$ be two nc homogeneous varieties such that $\widetilde\fV$ and $\widetilde\fW$ are nc biholomorphic.
If $0$ is not mapped to $0$, then there exist
two discs $D_1 \subseteq \fV(1)=\widetilde\fV(1)$ and $D_2 \subseteq \fW(1)=\widetilde\fW(1)$, both containing $0$, such that $D_1$ is mapped by the biholomorphism onto $D_2$.
\end{lem}
\begin{proof}
Let $G: \widetilde\fV \to \widetilde\fW$ be an nc biholomorphism mapping $\widetilde\fV$ onto $\widetilde\fW$. 
If $0$ is not mapped to $0$, then both $V:=\fV(1)=\widetilde\fV(1)$ and $W:=\fW(1)=\widetilde\fW(1)$ are non-trivial homogeneous varieties and $G|_{V}:V \to W$ is a biholomorphism. The first paragraph of the proof of \cite[Lemma 5.9]{SalomonShalit} (see also \cite{DRS11})  shows --- using an analysis of the singular nuclei of $V$ and $W$ --- that there are two discs $D_1  \subseteq V$ and $D_2 \subseteq W$, both containing $0$, such that $D_1$ is mapped by $G|_{V}$ onto $D_2$.
\end{proof}

\begin{prop}\label{prop:biholo=>0-biholo}
Let $\fV \subseteq \fB_d$ and $\fW \subseteq \fB_e$ be two nc homogeneous varieties such that $\widetilde\fV$ and $\widetilde\fW$ are nc biholomorphic.
Then there exists an nc biholomorphism $F$ of $\widetilde\fV$ onto $\widetilde\fW$ that maps $0$ to $0$.
\end{prop}

\begin{proof}
We can import the ``disc trick" used in \cite[Proposition 4.7]{DRS11} to the current setting (see also \cite[Lemma 5.9]{SalomonShalit}). 
Since the argument is just a couple of paragraphs long, we include it for completeness. 

Let $G$ be an nc biholomorphism mapping $\widetilde\fV$ onto $\widetilde\fW$.
If $0$ is mapped by $G$ to $0$, we are done.
Assume that $G(0) \neq 0$.
We will prove that there exists an nc biholomorphism $F$, mapping $\widetilde\fV$ onto $\widetilde\fW$, such that $F(0) = 0$.

Lemma \ref{lem:DiscToDisc} implies there exist two discs $D_1 \subseteq \widetilde\fV(1)$ and $D_2 \subseteq \widetilde\fW(1)$ such that
$G(D_1) = D_2$.
Define
\[
\cO(0;\widetilde\fV):=\{z \in D_1 : z=F(0) \text{ for some automorphism $F$ of $\widetilde\fV$} \},
\]
and
\[
\cO(0;\widetilde\fV,\widetilde\fW):=\left\{z \in D_2~ :~
\begin{minipage}{0.43\linewidth}
\text{$z=F(0)$ for some nc biholomorphism}\\
\text{$F$ of $\widetilde\fV$ onto $\widetilde\fW$}
\end{minipage}
\right\}.
\]
Since homogeneous varieties are invariant under multiplication by complex numbers, it is easy to check that these sets are circular, that is, for every $\mu \in \cO(0;\widetilde\fV)$ and $\nu \in \cO(0;\widetilde\fV,\widetilde\fW)$, it holds that $C_{\mu,D_1}:=\{z \in D_1: |z|=|\mu|\} \subseteq \cO(0;\widetilde\fV)$ and $C_{\nu,D_2}:=\{z \in D_2: |z|=|\nu|\} \subseteq \cO(0;\widetilde\fV,\fW)$.

Now, as $G(0)$ belongs to $\cO(0;\widetilde\fV,\widetilde\fW)$, we obtain that $C := C_{G(0),D_2}\subseteq \cO(0;\widetilde\fV,\widetilde\fW)$.
Therefore, the circle $G^{-1}(C)$ is a subset of $\cO(0;\widetilde\fV)$; note that this circle passes through the point $0 = G^{-1}(G(0))$. 
As $\cO(0;\widetilde\fV)$ is circular, every point of the interior of the circle $G^{-1}(C)$ lies in $\cO(0;\widetilde\fV)$.
Thus, the interior of the circle $C$ must be a subset of $\cO(0;\widetilde\fV,\widetilde\fW)$. 
But the interior of $C$ contains $0$. 
We conclude that $0 \in \cO(0;\widetilde\fV,\widetilde\fW)$.
\end{proof}

\begin{prop} \label{prop:preserves_radius}
Let $\fV \subseteq \fB_d$ and $\fW \subseteq \fB_e$ be two homogeneous nc varieties. Let $G \colon \widetilde{\fV} \to \widetilde{\fW}$ be an nc biholomorphism, such that $G(0) = 0$. 
Then for every $X \in \widetilde{\fV}$ we have that $\rho(X) = \rho(G(X))$.
\end{prop}
\begin{proof}
Since similarities preserve the joint spectral radius, it suffices to prove that for every $X \in \fV$, $\rho(G(X)) = \rho(X)$. Let $X \in \fV$ be irreducible. 
By Lemma \ref{lem:irreducible_radius} we know that $\rho(X) = \min\{ \|S^{-1} X S\| \mid S \in \GL_n\}$. For the same reason, we may choose $T \in \GL_n$ that realizes the minimum and replace $X$ by $T^{-1} X T$ and assume that $\rho(X) = \|X\|$.

Now consider the function $f(z) = G(z X/\|X\|)$ defined on the disc. Since $G(0) = 0$, by Lemma \ref{lem:free_spectral_Schwarz} we know that $\rho(G(X)) = \rho(f(\|X\|)) \leq \|X\| = \rho(X)$. Applying the same consideration for $G^{-1}$ we get that for every irreducible $X \in \fB_d$, $\rho(X) = \rho(G(X))$.

By Lemma \ref{lem:irreducible_radius} we know that if $X$ is reducible, then $\rho(X) = \max\{ \rho(X_1),\ldots,\rho(X_k)\}$, where $X_1, \ldots, X_k$ are the Jordan--H\"{o}lder components of $X$. Since $G$ is an nc automorphism the Jordan--H\"{o}lder components of $G(X)$ are precisely $G(X_1),\ldots,G(X_k)$ and since we know that for each of them the spectral radius is preserved, we are done.
\end{proof}

%For a multiplier $f \in H^{\infty}(\fB_d)$, let us write $\rho(f)$ for the spectral radius of the multiplication operator $M_f$. Note that for every $X \in \fB_d$ the subspace $\bigvee_{v,y \in \C^n} K_{X,v,y}$ is $M_f^*$-invariant and the spectrum of the restriction of $M_f^*$ to this subspace is precisely the spectrum of $f(X)^*$ we conclude that $\rho(f) = \sup\{ \rho(f(X)) \mid X \in \fB_d \}$. Note that the above discussion accounts only for the point spectrum, but since the kernels are dense, we get the equality for the spectral radius.

The following is the free analogue of \cite[Lemma 7.5]{DRS11}.

\begin{cor} \label{cor:linear_isom}
Let $\fV, \fW \subseteq \fB_d$ be two homogeneous nc varieties. Let $A \colon \C^d \to \C^d$ be an invertible linear map, such that $A$ maps $\widetilde{\fV}$ bijectively onto $\widetilde{\fW}$. 
Then for every $X \in \fV$, $\rho(X) = \rho(A(X))$.
\end{cor}

Before stating the next proposition, we require some notation. 
For an ideal $J \triangleleft \bC\langle z_1, \ldots, z_d\rangle$, we write
\[
Z(J) = Z_{\bM_d}(J) = \{X \in \bM_d : p(X) = 0 \,\, \textrm{ for all } \,\, p \in J\}.
\]
Given a homogeneous variety $\fV \subseteq \fB_d$, we let $J_\fV$ be the corresponding homogeneous ideal in $\bC\langle z_1, \ldots, z_d\rangle$ consisting of polynomials vanishing on $\fV$. 

\begin{prop}\label{prop:zeros_to_zeros}
Let $\fV \subseteq \fB_d$ and $\fW \subseteq \cB_e$ be two homogeneous nc varieties, and let $G: \widetilde{\fW} \to \widetilde{\fV}$ be a $0$-preserving nc holomorphic map. 
Then the linear transformation $\Delta G (0,0)$ maps $\widetilde\fW$ into $\widetilde\fV$. 
\end{prop}
\begin{proof}
First, we prove that $\Delta G (0,0)$ maps $Z(J_\fW)$ into $Z(J_\fV)$.
For every $W\in \fW(n)$ define $\gamma_W: \mathbb D \to \mathbb M_d(n)$ by
\[
\gamma_W(z):= G(zW).
\]
Then $\gamma_W$ is a holomorphic map satisfying $\gamma_W(\mathbb D) \subseteq \widetilde\fV(n)$. As $G$ maps $0$ to $0$, $\gamma_W$ maps $0$ to $0$ as well. Thus, there exists a holomorphic map $\beta_W: \mathbb D \to \mathbb M_d(n)$ such that
\[
\gamma_W(z)=z\left(\gamma_W'(0)+z\beta_W(z)\right).
\]
Now let $f \in J_\fV$ be an $m$-homogeneous function. As $\gamma_W(\mathbb D) \subseteq \widetilde\fV(n)$ we have
\[
0
	= f\left(\gamma_W(z)\right) 
	= f\left( z\left(\gamma_W'(0)+z\beta_W(z)\right) \right)
	= z^m f\left(\gamma_W'(0)+z\beta_W(z)\right)
\]
for every $z \in \mathbb D$, and thus $f\left(\gamma_W'(0)+z\beta_W(z)\right)=0$ for every $z \in \mathbb D$. Evaluating this at $0$, we have $f\left(\gamma_W'(0)\right)=0$.
Since this is true for every homogeneous $f \in J_\fV$, we must have that $\gamma_W'(0) \in Z(J_\fV)$.
But
\[
\gamma_W'(0)=\lim_{z\to 0}\frac{G(zW)}{z}= \Delta G(0,0) (W),
\]
so $\Delta G (0,0)$ maps $\fW$, and hence $Z(J_\fW)$, to $Z(J_\fV)$.

Now, $\widetilde\fV=Z(J_\fV) \cap \widetilde\fB_d$, so in order to prove that $\Delta G(0,0)$ maps $\widetilde{\fW}$ into $\widetilde{\fV}$, it suffices to show that $\Delta G(0,0)$ maps $\widetilde \fW$ into $\widetilde \fB_d$. 
Let $X \in \fW(n)$ and consider the function $f(z) = G(zX/\|X\|)$. 
Since $G(zX) \in \widetilde{\fV}$ for every $z \in \D$ we know that $\rho(f(z)) < 1$ for every $z \in \D$. 
By Lemma \ref{lem:free_spectral_Schwarz} we get that $\frac{1}{\|X\|}\rho(\Delta G(0,0)(X)) = \rho(f^{\prime}(0))  \leq 1$, and, in particular, $\rho(\Delta G(0,0)(X)) < 1$. 
Thus by Lemma \ref{lem:tilde_is_spectral_ball} we conclude that $\Delta G(0,0)(X) \in \widetilde{\fB}_d$. 
\end{proof}

\begin{prop}\label{prop:bdd_linear=>cb_linear}
Let $\fV \subseteq \fB_d$ and $\fW \subseteq \fB_e$ be homogeneous varieties and
let $A$ be an invertible linear transformation mapping $\widetilde \fW$ onto $\widetilde \fV$.
Then the following statements are equivalent:
\begin{enumerate}[(i)]
\item $A$ is bi-Lipschitz with respect to $\ph$;
\item $\sup_{W\in \fW}\|\Phi_{A(W)}\|_{cb} < \infty$ and $\sup_{V\in \fV}\|\Phi_{A^{-1}(V)}\|_{cb} < \infty$;
\item the mapping $f \mapsto \widetilde{f} \circ A$, for $f \in \cH^\infty(\fV)$, is a completely bounded isomorphism $H^\infty(\fV) \to H^\infty(\fW)$;
\item $A$ is bi-Lipschitz with respect to $\phb$;
\item $\sup_{W\in \fW}\|\Phi_{A(W)}\| < \infty$ and $\sup_{V\in \fV}\|\Phi_{A^{-1}(V)}\| < \infty$;
\item the mapping $f \mapsto \widetilde{f} \circ A$, for $f \in \cH^\infty(\fV)$, is a bounded isomorphism $H^\infty(\fV) \to H^\infty(\fW)$;  and
\item the mapping $f \mapsto \widetilde{f} \circ A$, for $f \in \cH^\infty(\fV)$, is an algebraic isomorphism $H^\infty(\fV) \to H^\infty(\fW)$.
\end{enumerate}
\end{prop}

\begin{proof}
The statements (i)--(iii) are equivalent by Proposition
\ref{prop:3equiv}, and the statements (iv)--(vii) are equivalent by Proposition \ref{prop:4equiv_prime}. In addition, the equivalent statements (i)--(iii) clearly imply (iv)--(vii), so we must show the opposite implication. Let $\alpha:H^\infty(\fV) \to H^\infty(\fW)$ be the bounded homomorphism of item (vi); we will show it is completely bounded. 

Recall that by \cite[Lemma 7.10]{SalShaSha17}, every nc holomorphic function $f$ on an nc variety $\fV$ which is $n$-homogeneous --- namely, $f(\lambda Z)=\lambda^n f(Z)$ for every $Z \in \fV$ and $\lambda \in \mathbb D$ --- is a restriction to the variety of an $n$-homogeneous polynomial $p$, and $\|f\|_{H^\infty(\fV)} = \|f\|_{\mathcal H_\fV} = \|p\|_{H^\infty(\fB_d)} = \|p\|_{\mathcal H^2_d}$.

Let $T$ be the densely-defined linear map $\mathcal H_\fV \to \mathcal H_\fW$ defined on a polynomial $f$ by $Tf=\alpha(f)$. 
As $A$ is linear, $T$ must be graded, namely, it maps $n$-homogeneous functions to $n$-homogeneous functions. Let $f \in \mathcal H_\fV$ be a finite sum of homogeneous functions $f=\sum_n f_n$. 
Then obviously $\sum_n Tf_n$ is the homogeneous decomposition of $Tf$ and
\[
\begin{split}
\|Tf\|_{\mathcal H_\fW}^2	&=		\sum_n \|Tf_n\|_{\mathcal H_\fW}^2
					=		\sum_n \|Tf_n\|_{H^\infty(\fW)}^2\\
					&\leq		\sum_n \|\alpha\|^2 \|f_n\|_{H^\infty(\fV)}^2
					=		\|\alpha\|^2 \sum_n \|f_n\|_{\mathcal H_\fV}^2\\
					&=		\|\alpha\|^2 \|f\|_{\mathcal H_\fV}^2.
\end{split}
\]
Thus, $T$ extends to a well defined bounded map $\mathcal H_\fV \to \mathcal H_\fW$.

Now let $k_{W,v,y}$ be a kernel function in $\mathcal H_\fW$. Then, for every $g \in \mathcal H_\fV$ we have
\[
\langle g , T^*k_{W,v,y} \rangle	=	\langle (Tg)(W)v,y \rangle
						=	\langle g(AW)v,y \rangle
						=	\langle g,k_{AW,v,y} \rangle.
\]
Thus, $T^*k_{W,v,y}=k_{AW,v,y}$.

Now, as $H^\infty (\fV) \cong \mlt(\mathcal H_\fV)$ for every $f \in \mathcal H^\infty(\fV)$ there exists a bounded operator $M_f \in B(\mathcal H_\fV)$ such that $fg=M_fg$ for every $g \in \mathcal H_\fV$. So for every kernel function $k_{W,v,y} \in \mathcal H_\fW$ we have
\[
(T^*)^{-1} M_f^* T^* k_{W,v,y}	=	(T^*)^{-1}M_f^* k_{AW,v,y}
					%	=	(T^*)^{-1} k_{AW,v,f(AW)^*y}
						=	k_{W,v,f(AW)^*y}
						=	M_{\alpha(f)}^* k_{W,v,y}.
\]
Therefore
\[
M_{\alpha(f)}	=	T M_f T^{-1}.
\]
Thus, $\alpha$ is given by similarity and is therefore completely bounded.
\end{proof}

Combining Theorem \ref{thm:homog_no_need_w*} and Propositions \ref{prop:zeros_to_zeros} and \ref{prop:bdd_linear=>cb_linear}, we obtain the main result of this section. 

\begin{thm}\label{thm:isom_thm_for_homo}
Let $\fV \subseteq \fB_d$ and $\fW \subseteq \cB_e$ be two homogeneous nc varieties. 
The following statements are equivalent: 
\begin{enumerate}[(i)]
\item $H^\infty(\fV)$ and $H^\infty(\fW)$ are weak-$*$ continuously isomorphic. 
\item $H^\infty(\fV)$ and $H^\infty(\fW)$ are boundedly isomorphic. 
\item $H^\infty(\fV)$ and $H^\infty(\fW)$ are completely boundedly isomorphic. 
\item There exists a $\phb$-bi-Lipschitz linear map mapping $\widetilde{\fW}$ onto $\widetilde{\fV}$. 
\item There exists a $\ph$-bi-Lipschitz linear map mapping $\widetilde{\fW}$ onto $\widetilde{\fV}$. 
\end{enumerate}
\end{thm}
\begin{proof} 
By Corollary \ref{cor:cb-isom<=>biholo_weakstar} and Proposition \ref{prop:bdd_linear=>cb_linear}, we need to prove only one implication. 

If $\alpha$ is a bounded isomorphism, then by Theorem \ref{thm:homog_no_need_w*} it induces a bi-Lipschitz biholomorphism $G : \widetilde{\fW} \to \widetilde{\fV}$. 
By applying the analogue of Proposition \ref{prop:biholo=>0-biholo} in the bi-Lipschitz category, we may assume that $G : \widetilde{\fW} \to \widetilde{\fV}$ is a bi-Lipschitz biholomorphism that takes $0$ to $0$. 
By Proposition \ref{prop:4equiv_prime}, we have 
\[
M:= \sup_{W \in \fW} \|\Phi_{G(W)}\| < \infty .
\]
Let $A = \Delta G(0,0)$. 
By Proposition \ref{prop:zeros_to_zeros}, $A$ maps $\widetilde{W}$ into $\widetilde{V}$. 
Our goal is to show that $\sup_{W \in \fW} \|\Phi_{A(W)}\| < \infty$; in fact we will show that $\sup_{W \in \fW} \|\Phi_{A(W)}\| \leq M$. 
  
Fix $W \in \fW(n)$.
We define $u: \ol{\bD} \to \bM_d$ by $u(z) = G(zW)/z$. 
Since $\fW$ is homogeneous, for every $z \in \ol{\bD}$ we have that $zW \in \fW(n)$, and therefore, $G(zW) \in \widetilde \fV \subseteq \widetilde \fB_d$, so by Lemma \ref{lem:tilde_is_spectral_ball} $\rho(G(zW))<1$. Since $\ol{\bD}$ is compact $s:=\sup_{z\in\ol{\bD}}\rho(G(zW))<1$, so we can choose some $s<r<1$.
As $z \mapsto \frac{1}{r}G(zW)$ is an analytic function mapping $\mathbb D$ into $\widetilde \fB_d$ and $0$ to $0$, by the Vesentini--Schwarz Lemma (Lemma \ref{lem:free_spectral_Schwarz}) we have that $\rho(\frac{1}{r}G(zW)) \leq |z|$ so that $\rho(u(z)) < r<1$. Thus, by Lemma \ref{lem:tilde_is_spectral_ball} we see that $u(z) \in \widetilde{\fB_d}$ for all $z \in \bD$. 
Since $\widetilde \fV$ is homogeneous, $u(z) = G(zW)/z$ belongs to $Z(J_{\fV})$, so we conclude that $u(z) \in Z(J_{\fV}) \cap \widetilde{\fB_d} = \widetilde \fV$ for all $z \in \mathbb D$.

Next, we define a function $F$ from $\ol{\bD}$ to the Banach space $B(H^{\infty}(\fV),M_n)$ by $F(z) = \Phi_{u(z)}$. We will show that this function is holomorphic. By applying some M\"obius transformation it suffices to show that $F$ is differentiable at the origin. 
Let $f \in H^{\infty}(\fV)$ with $\|f\| \leq 1$, and set 
$g(z):=\frac{1}{2}\left(f(u(z)) - f(u(0))\right)$. Then $g:\mathbb D \to M_n$ is a holomorphic function with $\|g\| \leq 1$ and $g(0)=0$. By the Schwarz lemma, $\|g(z)\| \leq |z|$, so we obtain
\[
\|\Phi_{u(z)}-\Phi_{u(0)}\| \leq 2|z|.
\]

Now, the function $f \circ u$ is holomorphic in $\ol{\bD}$ so by Cauchy's integral formula for every $z\neq w \in \frac{1}{2}\bD$
\[
\begin{split}
\frac{1}{|z-w|} \Bigg\|& \frac{\Phi_{u(z)}(f)-\Phi_{u(0)}(f)}{z} 
				-  \frac{\Phi_{u(w)}(f)-\Phi_{u(0)}(f)}{w}\Bigg\|\\
	&=\frac{1}{|z-w|}\left\| \frac{f(u(z))-f(u(0))}{z} - \frac{f(u(w))-f(u(0))}{w} \right\| \\
	&=\frac{1}{|z-w|}\left\| \frac{\int_{|\zeta|=1}f(u(\zeta))\left(\frac{1}{\zeta-z} - \frac{1}{\zeta}\right)d\zeta}{z} -  \frac{\int_{|\zeta|=1}f(u(\zeta))\left(\frac{1}{\zeta-w} - \frac{1}{\zeta}\right)d\zeta}{w}\right\| \\
	&=\left\|\int_{|\zeta|=1}f(u(\zeta))\left(\frac{1}{\zeta(\zeta-z)(\zeta-w)} \right)d\zeta\right\|\leq 4,
\end{split}
\]
where $d \zeta$ is the normalized Lebesgue measure of $\mathbb T$.
Since this is true for every $\|f\| \leq 1$, we have that for every $z,w \in \frac{1}{2}\bD$
\[
\left\|\frac{F(z)-F(0)}{z}-\frac{F(w)-F(0)}{w}\right\|
=\left\| \frac{\Phi_{u(z)}-\Phi_{u(0)}}{z} 
				-  \frac{\Phi_{u(w)}-\Phi_{u(0)}}{w}\right\| \leq 4 |z-w|,
\]
and in particular $\lim_{z \to 0}\frac{F(z)-F(0)}{z}$ exists, so $F$ is differentiable at the origin. 

Now let $\beta_{\theta} \in \Aut(H^{\infty}(\fV))$ be the (completely isometric) rotation automorphism
\[
\beta_{\theta}(f)(V):=f(e^{-i\theta} V),\quad f \in H^{\infty}(\fV),~V \in \fV.
\]
Then $F(e^{i\theta})=\Phi_{e^{-i\theta}G(e^{i\theta} W)} =\Phi_{G(e^{i\theta} W)} \circ \beta_\theta.$ Thus, $\sup_{|z|=1}\|F(z)\| \leq M$.
But as $F$ is holomorphic in $\ol{\bD}$, it satisfies Cauchy's integral formula in the closed disc, which implies a maximum principle: 
\[
\|F(0)\| \leq \max_{|z|=1} \|F(z)\| \leq M. 
\]
However, since $F(0) = \Phi_{A(W)}$, we have that $\|\Phi_{A(W)}\| \leq M$ for all $W \in \fW$. 
This completes the proof.
\end{proof}

%%%%%%%%%%%%%%%%%%%%%%%%%%%%%%%%%%
\begin{quest}\label{quest:bilipLin}
Is a linear map between similarity envelopes of nc homogeneous varieties automatically bi-Lipschitz?
\end{quest}

%%%%%%%%%%%%%%%%%%%%%%%%%%%%%%%%%%

%%%%%%%%%%%%%%%%%%%%%%%%%%%%%%%%%%

\begin{example} \label{ex:x^2}

Let us consider a simple example that illuminates the rigidity of the noncommutative case.  
Let $\fV \subseteq \fB_2$ be the subvariety cut out by the single polynomial $x^2$. 
Let $A$ be a $2 \times 2$ matrix mapping $\widetilde{\fV}$ isomorphically onto some subvariety of $\widetilde{\fB}_2$. 
Note that $\fV(1)$ is the single line corresponding to the $y$-axis. 
Therefore, by \cite[Lemma 7.5]{DRS11} $A$ maps the $y$-axis isometrically onto some line. Multiplying $A$ on the left by a unitary we may assume that $A$ has the form $A = \left( \begin{smallmatrix} a & 0 \\ c & 1 \end{smallmatrix}\right)$. We will show that in this case $A$ is a diagonal unitary.

First consider the point
\[
Z = \begin{bmatrix} 0 & 1 \\ 0 & 0 \end{bmatrix}, \, W = \begin{bmatrix} 0 & 0 \\ 1 & 0 \end{bmatrix}.
\]
This point is a coisometry and thus has spectral radius $1$. By Corollary \ref{cor:linear_isom} we know that $A$ preserves the spectral radius and thus the image must also have spectral radius $1$. The matrix of the associated completely positive map with respect to the standard basis of $M_2$ is
\[
\begin{bmatrix} 0 & 0 & 0& |a|^2 + |c|^2 \\ 0 & 0 & c &0 \\ 0 & \overline{c} & 0 & 0 \\ 1 & 0 & 0 & 0\end{bmatrix}.
\]
Its spectral radius is $\sqrt{|a|^2 + |c|^2}$ and thus we have that $|a|^2 + |c|^2 = 1$.

Now, let $\epsilon, \delta > 0$, be such that $\epsilon^2 + \delta^2 = 1$. Write $c = r e^{i \theta}$ and consider the point 
\[
X = \begin{bmatrix} 0 & \epsilon e^{- i \theta} \\ 0 & 0 \end{bmatrix}, \, Y = \begin{bmatrix} 0 & \delta \\ \delta & 0 \end{bmatrix}.
\]
It is immediate that $X X^* + Y Y^* = \left(\begin{smallmatrix} 1 & 0 \\ 0 & \delta^2 \end{smallmatrix}\right)$ and therefore this point lies on the boundary of the ball. Furthermore, a short computation of the associated completely positive map shows that the spectral radius of this point is $0 < \sqrt{\delta} < 1$. By Corollary \ref{cor:linear_isom} $A$ maps the point $(X,Y)$ to a point with spectral radius $\sqrt{\delta}$. The image of this point is
\[
A(X,Y) = \left( \begin{bmatrix} 0 & a \epsilon e^{-i \theta} \\ 0 & 0 \end{bmatrix}, \, \begin{bmatrix} 0 & \delta + r \epsilon \\ \delta & 0 \end{bmatrix} \right).
\]
%Note that for every $n \in \N$, the $2n$-th power of the second coordinate is $\delta^n (\delta + r\epsilon)^n I$. Since the spectral radius is at most one, we have that for every $\eta > 0$, $\delta (\delta + r \epsilon) < (1 + \eta)^4$ and thus $\delta^2 + r \epsilon \delta \leq 1$. Since $\delta^2 = 1 - \epsilon^2$ and $\epsilon > 0$, we have that $r \leq \epsilon/\delta$. Letting $\epsilon$ tend to $0$ we see that it can only hold if $r = 0$. We conclude that $A$ is in fact diagonal.
Computing again the associated completely positive map and its spectral radius we get that the spectral radius of the point is $\sqrt{\delta} \sqrt[4]{(\delta+r \epsilon)^2 + |a|^2 \epsilon^2}$. Thus $(\delta+r \epsilon)^2 + |a|^2 \epsilon^2 = 1$. Opening brackets and using our assumptions and the fact that $|a|^2 + r^2 = 1$ we get that:
\[
1 = (\delta+r \epsilon)^2 + |a|^2 \epsilon^2 = 1 - \epsilon^2 + 2 r \epsilon \delta + r^2 \epsilon^2 + |a|^2 \epsilon^2 = 1 + 2 r \epsilon \delta.
\]
This is possible if and only if $r = 0$ and thus $A$ is a diagonal unitary.

Combining this with Theorem \ref{thm:isom_thm_for_homo} we see that the only varieties that are isomorphic to the variety cut out by $x^2$ are those that are conformally equivalent to it via a unitary.
\end{example}

%%%%%%%%%%%%%%%%%%%%%%%%%%%%%%%%%%%%%%%%%%%%%%%
\section{Example: the $q$-commutation varieties} \label{sec:q-commutation}
%%%%%%%%%%%%%%%%%%%%%%%%%%%%%%%%%%%%%%%%%%%%%%%

In quantized versions of classical algebras, relations such as $xy = qyx$ --- also known as $q$-commutation --- naturally appear. Let $q \in \mathbb C$ and let $\fV_q \subseteq \fB_2$ be the homogeneous variety generated by the ideal $\langle z_1z_2 - qz_2 z_1\rangle$.

We will first show that unless $p \in \{q,\frac{1}{q}\}$, there is  
no $2 \times 2$ matrix mapping $\widetilde{\fV_q}$ to $\widetilde{\fV_p}$, and therefore for $q \in \mathbb D \cup \{e^{i\theta}: 0\leq \theta \leq \pi\}$ we get a continuum of mutually non-isomorphic Banach algebras $H^\infty(\fV_p)$. 

\begin{thm}
Let $q \neq p$ be two complex numbers. 
If $p = \frac{1}{q}$, then
\begin{enumerate}[(i)]
\item the varieties $\fV_q$ and $\fV_p$ are conformally equivalent, ~and
\item the operator algebras $H^{\infty}(\fV_q)$  and $H^{\infty}(\fV_p)$ are unitarily equivalent;
\end{enumerate}
and otherwise,
\begin{enumerate}[(i)]
\item the varieties $\widetilde{\fV_q}$ and $\widetilde{\fV_p}$ not biholomorphically equivalent, ~and
\item the Banach algebras $H^{\infty}(\fV_q)$  and $H^{\infty}(\fV_p)$ are not boundedly isomorphic.
\end{enumerate}
\end{thm}
\begin{proof}
For the first case, note that for every $q \neq 0$, the unitary matrix
$
\begin{bmatrix}
0	&	1\\
1	&	0
\end{bmatrix}
$
--- an automorphism of $\fB_2$ --- maps $\fV_q$ onto $\fV_{\frac{1}{q}}$. Thus, $\fV_q$ and $\fV_{\frac{1}{q}}$ are conformally equivalent, and as a result $H^\infty(\fV_q)$ and $H^\infty(\fV_{\frac{1}{q}})$ are unitarily equivalent.

We now move to the second case. Note that if $q=1$, then the first level of $\fV_q$ --- a variety that is considered in the next section and is called the commutative nc ball --- is $\mathbb B_2$ while the first level of every other $\fV_q$ is only the cross $C:=\{(x,y) \in \mathbb B_2: xy=0 \}$. So we may assume that $q \neq 1$.

Now let $p,q \in \mathbb D \cup \{e^{i\theta}: 0< \theta \leq \pi\}$, and assume there is a $2 \times 2$ invertible matrix $A$ mapping $\fV_q$ onto $\fV_p$. A simple computation shows that a matrix $A$, mapping the cross $C$ onto itself, must be of the form
\[
\begin{bmatrix}
\lambda_1	&	0\\
0		&	\lambda_2
\end{bmatrix}
~\text{ or }~
\begin{bmatrix}
0		&	\lambda_1\\
\lambda_2&	0
\end{bmatrix}
\]
for some $\lambda_1,\lambda_2 \in \mathbb T$. 
The first option maps $\fV_q$ onto itself, and the second option maps $\fV_q$ onto $\fV_{\frac{1}{q}}$, so we only need to show that $\fV_q \neq \fV_p$ for $q \neq p$. 
Indeed, the point
\[
\frac{1}{2\sqrt{q^2+1}}
\left(
\begin{bmatrix}
q	&	0\\
0	&	1
\end{bmatrix}
,
\begin{bmatrix}
0	&	1\\
0	&	0
\end{bmatrix}
\right) \in \fB_2
\] 
belongs to $\fV_q$ but not to $\fV_p$.
\end{proof}

We will now examine the existence of irreducible points in the varieties $\fV_q$. 
This will later help us understand maps of the corresponding operator algebras. 

\begin{prop}\label{prop:q-comm-irreducible}
If $q$ is not a root of unity, then $\fV_q$ has no non-scalar irreducible points.
If $q$ is a root of unity, primitive of order $k$, then it has non-scalar irreducible points only in the $k$-th level. In the latter case, they are all similar to a point of the form
\[
\left(
\lambda\begin{bmatrix}
1		& 0		& \cdots	& \cdots	& 0 		\\
0		& q		&  \ddots 	& 		& \vdots	\\
\vdots	& \ddots	& q^2  	& \ddots	& \vdots 	\\
\vdots	& 		& \ddots  	& \ddots 	& 0	 	\\
0		& \cdots	& \cdots	& 0		& q^{k-1}	\\
\end{bmatrix} 
,
\begin{bmatrix}
0		& \cdots	& \cdots	& 0		& \mu	\\
1		& \ddots	&   		& 		& 0 		\\
0		& \ddots	& \ddots  	& 		& \vdots 	\\
\vdots	& \ddots	& \ddots  	& \ddots 	& \vdots 	\\
0		& \cdots	& 0		& 1		& 0 		\\
\end{bmatrix} \right)
\]
for some non zero $\lambda$ and $\mu$.
\end{prop}
\begin{proof}
Before we start, for $A \in M_n$ and $v \in \mathbb C^n$, let $\cC(A;v)$ denote the $A$-cyclic subspace of $\mathbb C^n$, namely, the linear subspace generated by $\{v, Av, A^2v, \dots, A^{n-1}v\}$.

Let $n>1$ and $(X,Y) \in \fV_q(n)$. 
First note that if $(X,Y)$ is irreducible, then $X$ cannot have $0$ as an eigenvalue. 
Indeed, if $v \neq 0$ is an eigenvector of $X$ corresponding to an eigenvalue $0$, then $\cC(Y;v) \subseteq \ker X$. 
So if $\dim \cC(Y;v)=n$, then $X=0$, and otherwise $\cC(Y;v)$ is a nontrivial invariant subspace for both $X$ and $Y$; in any case, this contradicts the irreducibility of $(X,Y)$. If $q \neq 0$, a symmetric argument shows that $Y$ cannot have $0$ as an eigenvalue too. 

If $q=0$, then as $X$ is invertible, the equality $XY=0$ implies $Y=0$, which contradicts the irreducibility of $(X,Y)$. Thus, the only irreducible points of $\fV_0$ are the ones in the first level. 
So henceforward, we assume that $q \neq 0$.

Suppose first that $1,q,q^2,\dots, q^{n}$ are $n+1$ distinct numbers, namely, that $q$ is not a root of unity of order less than or equal to $n$. 
Let $\lambda$ be a nonzero eigenvalue of $X$ with an eigenvector $v$. 
Then it is easy to check inductively that for every $m$, $Y^m v$ is an eigenvector of $X$ for the eigenvalue $\lambda  q^m$, which is of course impossible. 

In particular, if $q$ is not a root of unity, then the only irreducible points of $\fV_q$ are the ones in the first level, and if $q$ is a root of unity, say primitive of order $k>1$, then for every $1< l <k$, the $l$'th level $\fV_q(k)$ does not contain irreducible points.

We will now show that in the latter case --- where $q$ is a root of unity primitive of order $k$ --- there exist irreducible points in the $k$-th level, and we will give a full description of them. 
Let $\lambda$ be an eigenvalue of $X$, and recall that for every $0 \leq m \leq k-1$, $Y^m v$ is an eigenvector of $X$ with an eigenvalue $\lambda q^m$. 
Thus, the vectors $v, Yv, Y^v, \dots, Y^{k-1}v$ form a basis. 
Let $S$ be the invertible matrix whose rows are the latter vectors (in the order they were written). Additionally, note that as $Y^k v$ is an eigenvector of $X$ with an eigenvalue $\lambda q^k = \lambda$, there must be a scalar $\mu$ such that $Y Y^{k-1} v= \mu v$. 
Thus,

\[
S^{-1}XS=\lambda\begin{bmatrix}
1		& 0		& \cdots	& \cdots	& 0 		\\
0		& q		&  \ddots 	& 		& \vdots	\\
\vdots	& \ddots	& q^2  	& \ddots	& \vdots 	\\
\vdots	& 		& \ddots  	& \ddots 	& 0	 	\\
0		& \cdots	& \cdots	& 0		& q^{k-1}	\\
\end{bmatrix}
\quad\text{ and }\quad
S^{-1}YS=\begin{bmatrix}
0		& \cdots	& \cdots	& 0		& \mu	\\
1		& \ddots	&   		& 		& 0 		\\
0		& \ddots	& \ddots  	& 		& \vdots 	\\
\vdots	& \ddots	& \ddots  	& \ddots 	& \vdots 	\\
0		& \cdots	& 0		& 1		& 0 		\\
\end{bmatrix},
\]
and $\mu$ must be non-zero. 

%On the other hand, it is evident that if $\lambda$ and $\mu$ are small, then the above pair of matrices constitutes a point in $\widetilde{\fV_q}$. 
On the other hand, it is evident that the above pair of matrices constitutes a point in $Z(J_{\fV_q})$. 
To see that $(X,Y)$ is irreducible one may argue as follows. 
Begin by observing that $\det(XY-YX)=(q-1)^k\det X \det Y \neq 0$. 
Since for each two upper triangular matrices $R$ and $T$ the matrix $RT-TR$ is strictly upper triangular its determinant is $0$, so the pair $(X,Y)$ cannot be similar to a triangular pair $(R,T)$. As there are no irreducible points in levels $1<m<k$, the point $(X,Y)$ itself must be irreducible.

Finally, we will show --- still in the case that $q$ is a primitive root of unity of order $k$ --- that levels higher than $k$ contain no irreducible points.

Indeed, assume that $n>k$ and $(X,Y) \in \fV_q(n)$ is irreducible. 
Let $\lambda$ be an eigenvalue of $X$ (so $\lambda \neq 0$) with an eigenvector $v$. 
As before, consider the $Y$-cyclic subspace $\cC(Y,v)$. 
If its dimension is not $n$, then this subspace is a nontrivial invariant subspace for both $X$ and $Y$. Otherwise, the vectors $v,Yv,Y^2v, \dots, Y^{n-1}v$ form a basis.
Rearrange them such that all the eigenvectors in the list corresponding to the eigenvalue $\lambda$ (i.e., $v, Y^kv, \dots$) comes first, then the ones corresponding to the eigenvalue $\lambda q $ (i.e., $Yv, Y^{k+1}v, \dots$) , and so on. 
For each eigenvalue, keep them arranged in their original increasing form. 
Let $S$ be the invertible matrix whose rows are the latter vectors. 
Note also that $Y^nv$ is an eigenvector of $X$ with an eigenvalue $\lambda q^{n} $. 
Thus, if $k$ does not divide $n$, the first row of $S^{-1}YS$ must be zero, which is impossible.
Thus, $k | n$. 

Let $l=n/k$. Since $Y^nv$ is an eigenvector of $X$ with an eigenvalue $\lambda$, there are scalars $\mu_0, \dots \mu_{l-1} \in \mathbb C$ such that $Y^n v= \sum_{j=0}^{l-1} \mu_j Y^{jk+1}v$. Let
\[
A:=\begin{bmatrix}
0		& \cdots	& \cdots	& 0		& \mu_0	\\
1		& \ddots	&   		& \vdots	& \mu_1	\\
0		& \ddots	& \ddots  	& \vdots	& \vdots 	\\
\vdots	& \ddots	& \ddots  	& 0 		& \vdots 	\\
0		& \cdots	& 0		& 1		& \mu_{l-1}\\
\end{bmatrix} \in M_{l}(\mathbb C).
\] 

Then  
\[
S^{-1}XS=\lambda\begin{bmatrix}
I_l		& 0		& \cdots	& \cdots	& 0 		\\
0		& q I_l	&  \ddots 	& 		& \vdots	\\
\vdots	& \ddots	& q^2 I_l  	& \ddots	& \vdots 	\\
\vdots	& 		& \ddots  	& \ddots 	& 0	 	\\
0		& \cdots	& \cdots	& 0		& q^{k-1}I_l\\
\end{bmatrix}
\quad\text{ and }\quad
S^{-1}YS=\begin{bmatrix}
0		& \cdots	& \cdots	& 0		& A		\\
I_l		& \ddots	&   		& \vdots	& 0 		\\
0		& \ddots	& \ddots  	& \vdots	& \vdots 	\\
\vdots	& \ddots	& \ddots  	& 0	 	& \vdots 	\\
0		& \cdots	& 0		& I_l		& 0 		\\
\end{bmatrix}.
\]
Now let $0\neq v \in \mathbb C^n$ be some eigenvector of $A$.
Then, $S(\mathbb C v)^{\oplus k} \subseteq \mathbb C^n$ is a $k$-dimensional subspace invariant under both $X$ and $Y$. 
\end{proof}

Let us now consider a change of angle in the varieties $\fV_q$.
More precisely, let $\fW_q \subseteq \fB_2$ be the nc variety generated by the ideal $\langle (z_1- z_2)z_2 -q z_2(z_1-z_2) \rangle$, and let
\[
A=\begin{bmatrix}
1 & \frac{1}{\sqrt{2}}\\
0 & \frac{1}{\sqrt{2}}
\end{bmatrix}.
\]
Then obviously, for every $q \neq 1$, $A$ maps $\fV_q(1)$ onto $\fW_q(1)$ and $Z(J_{\fV_q})$ onto $Z(J_{\fW_q})$. 

Suppose first that $q$ is not a root of unity. 
By Proposition \ref{prop:q-comm-irreducible}, the irreducible points in $\widetilde{\fV_q}$ are only the scalar points. 
Now, since every $(X,Y) \in \M_2(n)$ is similar to a block upper triangular pair whose diagonal blocks are all irreducible, every irreducible $(X,Y) \in \widetilde{\fV_q}$ must be similar to an upper triangular matrix whose diagonal entries are in $\fV_q(1)$. 
Therefore, as $A$ maps $\fV_q(1)$ onto $\fW_q(1)$, Lemma \ref{lem:irreducible_radius} implies that for every  $(X,Y) \in \widetilde{\fV_q}$, we have $\rho(A(X,Y))<1$, so in view of Lemma \ref{lem:tilde_is_spectral_ball}, $A(X,Y) \in \widetilde{\fB_2}$, so $A$ maps $\widetilde{\fV_q}$ into $\widetilde{\fW_q}$. 

This argument works in reverse, and we conclude that for every $q$ which is not a root of unity $A$ maps $\widetilde{\fV_q}$ bijectively onto $\widetilde{\fW_q}$.
Unfortunately, we couldn't prove that $A$ is bi-Lipschitz; if we could, this would have shown that $H^{\infty}(\fV_q)$ and $H^\infty(\fW_q)$ are completely boundedly isomorphic.

If $q$ is a root of unity, say primitive of order $k$, then since irreducible points exist not only in the first level, it might happen that $A$ does not map $\widetilde{\fV_q}$ bijectively onto $\widetilde{\fW_q}$. 

Indeed, consider the case $q=-1$. Let
\[
X_0=\begin{bmatrix}
1		& 0		\\
0		& -1		\\
\end{bmatrix}
\quad\text{ and }\quad
Y_0=\begin{bmatrix}
0		& i		\\
1		& 0		\\
\end{bmatrix},
\]
which by Proposition \ref{prop:q-comm-irreducible}, is an irreducible point in $Z(J_{\fV_q})$.
We will show that $\rho(A(X_0,Y_0)) \neq \rho(X_0,Y_0)$. Corollary \ref{cor:linear_isom} would then imply that $\widetilde{\fV_{-1}}$ is not mapped bijectively onto $\widetilde{\fW_{-1}}$.

To this end, recall that for a pair of two $n \times n$ matrices $(A,B)$, the linear transformation on $\Theta_{(A,B)}:M_n\to M_n$, defined by $\Theta_{(A,B)}(T)=ATB$, is represented, with respect to the standard basis $E_{1,1}, E_{1,2}$, \dots, $E_{n,n}$, by the matrix $A \otimes B^T$. Thus, the matrix representing $\Psi_{(X,Y)}$ --- which is by definition $\Theta_{(X,X^*)}+\Theta_{(Y,Y^*)}$ --- is
\[
X \otimes \ol{X} + Y \otimes \ol{Y}.
\]
The square root of the largest number among all eigenvalue absolute values of the latter $n^2 \times n^2$ matrix is its spectral radius, and hence is equal to $\rho(X,Y)$.  

Now, the matrix representing $\Psi_{(X_0,Y_0)}$ is
\[
X_0 \otimes \ol{X_0} + Y_0 \otimes \ol{Y_0}=\begin{bmatrix}
1			& 0			& 0			& 1			\\
0			& -1			& i			& 0			\\
0			& -i			& -1			& 0			\\
1			& 0			& 0			& 1			\\
\end{bmatrix},
\]
and the matrix representing $\Psi_{A(X_0,Y_0)}$ is
\[
(X_0 + \frac{1}{\sqrt{2}}Y_0) \otimes \ol{(X_0 + \frac{1}{\sqrt{2}}Y_0)} + \frac{1}{\sqrt{2}}Y_0 \otimes\frac{1}{\sqrt{2}} \ol{Y_0}=\begin{bmatrix}
1 				& -\frac{1}{\sqrt{2}}i	&\frac{1}{\sqrt{2}}i	& 1				\\
\frac{1}{\sqrt{2}}		& -1				& i				& -\frac{1}{\sqrt{2}}i	\\
\frac{1}{\sqrt{2}}		& -i				& -1				& \frac{1}{\sqrt{2}}i	\\
1				& -\frac{1}{\sqrt{2}}	& -\frac{1}{\sqrt{2}}	& 1				\\
\end{bmatrix}.
\]
The largest number among all eigenvalue absolute values of the first matrix is $2$, and of the second --- $\sqrt{3}$.
Thus, $\rho(A(X_0,Y_0)) \neq \rho(X_0,Y_0)$. We conclude that $H^\infty(\fV_{-1})$ is not boundedly isomorphic to $H^\infty(\fW_{-1})$

Similar computations show that $\fV_{e^{\frac{2\pi i}{3}}}$ and $\fV_i$ also contain irreducible points whose joint spectral radius is not preserved under $A$.
Thus, $H^\infty(\fV_{e^{\frac{2\pi i}{3}}})$ and $H^\infty(\fV_{i})$ are not boundedly isomorphic to $H^\infty(\fW_{e^{\frac{2\pi i}{3}}})$ and $H^\infty(\fW_{i})$. %We could not prove the general case for any root of unity $q$.

%The matrix representing $\Psi_{(X,Y)}$ is
%\[
%\begin{bmatrix}
%|\lambda|^2	& 0			& 0			& |\mu|^2		\\
%0			& -|\lambda|^2	& \mu		& 0			\\
%0			& \ol{\mu}		& -|\lambda|^2	& 0			\\
%1			& 0			& 0			& |\lambda|^2	\\
%\end{bmatrix},
%\]
%and the matrix representing $\Psi_{A(X,Y)}$ is
%\[
%\begin{bmatrix}
%|\lambda|^2				& \frac{1}{\sqrt{2}}\ol{\mu}\lambda	& \frac{1}{\sqrt{2}}\mu\ol{\lambda}	& |\mu|^2						\\
%\frac{1}{\sqrt{2}}\lambda		& -|\lambda|^2					& \mu						& -\frac{1}{\sqrt{2}}\mu\ol{\lambda}	\\
%\frac{1}{\sqrt{2}}\ol{\lambda}	& \ol{\mu}						& -|\lambda|^2					& -\frac{1}{\sqrt{2}}\ol{\mu}\lambda	\\
%1						& -\frac{1}{\sqrt{2}}\ol{\lambda}		& -\frac{1}{\sqrt{2}}\lambda		& |\lambda|^2					\\
%\end{bmatrix}.
%\]

%%%%%%%%%%%%%%%%%%%%%%%%%%%%%%%%%%%%%%%%%%%%%%%
%%%%%%%%%%%%%%%%%%%%%%%%%%%%%%%%%%%%%%%%%%%%%%%
\section{The commutative case} \label{sec:connections_to_comm} 
In this section, we study the isomorphism problem for the algebras of bounded analytic functions on commutative nc varieties.
This problem was treated extensively in the fully commutative case, that is, when the algebra $H^\infty(\fV)$ lives on a variety $\fV$ which is minimal, in the sense that it is the minimal nc variety in $\fB_d$ which contains $V = \fV(1)$; this is referred to as the {\em isomorphism problem for complete Pick algebras}, see \cite{DHS14,DRS11,DRS15,Hartz16,KerMcSh13,McSh16,SalomonShalit}.
In \cite[Section 11]{SalShaSha17} we explained how this problem can be investigated in the nc commutative setting (cf. Remark \ref{rem:commutative_radical} below).

As in Section \ref{sec:homogeneous}, we assume $d< \infty$.
Let us define
\[
\fC\bM_d = \{ X \in \bM_d ~:~ X_i X_j = X_j X_i\, , \, i,j=1, \ldots, d\}
\]
and
\[
\fC\fB_d = \{ X \in \fB_d ~:~ X_i X_j = X_j X_i\, , \, i,j=1, \ldots, d\} .
\]
Our goal is, as before, to classify the algebras $H^\infty(\fV)$, but this time only for nc subvarieties $\fV \subseteq \fC\fB_d$; we shall refer to such nc varieties as {\em commutative nc varieties}. Obviously, the operator algebra $H^\infty(\fV)$ of a commutative nc variety $\fV$ is commutative. 
Note that $H^\infty(\fC\fB_d) = \cM_d$ --- the multiplier algebra on the Drury--Arveson space;
see \cite[Section 11]{SalShaSha17} for a detailed explanation of this fact.

In the setting of commutative nc varieties, we relax somewhat the weak-$*$ continuity conditions of our basic classification results from Section \ref{sec:isomorphisms}. 
When we specialize further to homogeneous commutative nc varieties, we will recover the results of Section \ref{sec:homogeneous} by other techniques. 
An important ingredient for the proof of the classification result is a Nullstellensatz for homogeneous commutative nc varieties, which may be of independent interest. 

Let us define a new nc set $\widehat{\fC\fB_d}$ by
\[
\widehat{\fC\fB_d} := \{X \in \fC \bM_d : \sigma(X) \subseteq \bB_d\} .
\]
Here $\sigma(X)$ could mean Taylor spectrum, but since $X$ is a tuple acting on a finite dimensional space, $\sigma(X)$ can be defined by any other reasonable definition of spectrum.
Our working definition for the joint spectrum $\sigma(X)$ will be the set of $d$-tuples in $\C^d$ made from taking elements from the diagonal of $(X_1, \ldots, X_d)$ after they are realized in a simultaneous upper triangular form.

For every $X \in \widehat{\fC\fB_d}$, the mapping $f \mapsto f(X)$, which is well defined on $\bC[z]$, extends to a continuous representation of the topological algebra $\cO(\bB_d)$ consisting of all analytic functions on $\bB_d$ \cite[Theorem 4.3]{Tay70}.

\begin{thm}\label{thm:rep_similar}
Every $X \in \widehat{\fC\fB_d}$ is simultaneously similar to some $Y \in \fC\fB_d$.
In other words, $\widehat{\fC\fB}_d = \widetilde{\fC\fB_d}$.
\end{thm}
\begin{proof}
This theorem follows from the results of Section \ref{subsec:jnr}, but we wish to present a different proof that is also interesting. 

Fix $X \in \widehat{\fC\fB_d}(n)$.
Since the map $EV_X : \cO(\bB_d) \to M_n$ given by evaluating at $X$ is continuous, there must be a compact set $K \subseteq \bB_d$ and a constant $C$ for which
\[
\|f(X)\| \leq C \|f\|_{\infty,K},
\]
where $\|f\|_{\infty,K} = \sup\{|f(x)| : x \in K\}$.
Considering $A$, the closure of $(\cO(\bB_d), \|\cdot\|_{\infty,K})$, as a subalgebra of $C(K)$, it is evident that it is a unital operator algebra.
The evaluation map $EV_X: \cO(\bB_d) \to M_n$ extends to a bounded unital map from $A$ into $M_n$.
Since $M_n$ is finite dimensional, the evaluation map must be completely bounded \cite[Corollary 2.2.4]{EffrosRuan}.
By Paulsen's Theorem \cite[Theorem 9.1]{PaulsenBook}, $EV_X$ is similar to a completely contractive homomorphism $\rho$. Since $(\bC[z], \|\cdot\|_{\infty,K})$ is dense in $A$, the homomorphism $\rho$ must be of the form $EV_Y$ for $Y=(\rho(z_1),\dots,\rho(z_d)) \in \fC\bM_d$.
Finally, to see that $Y$ must be in $\fC \fB_d$, we note that the row norm of $(z_1, \ldots, z_d)$ is
\[
\|(z_1, \ldots, z_d)\| = \sup_{x \in K}\sqrt{\sum |x_i|^2} < 1,
\]
so, because $\rho$ is a complete contraction, $\|(Y_1, \ldots, Y_d)\| <1$.
\end{proof}

\begin{rem}
The above proof exposes the connection between the operator space completions of $\cO(\bB_d)$ and the representations at hand. 
However, one can provide a more elementary proof by induction as follows. 
For every $X \in \cbh(1) = \bB_d$ the claim is trivial. 
Now let $X \in \cbh(n)$ and choose a unitary, such that $U^* X U$ is upper triangular. 
Write $X = \left( \begin{smallmatrix} \alpha & v \\ 0 & X' \end{smallmatrix} \right)$, where $\alpha \in \bB_d$, $v$ is some vector and $X' \in \cbh(n-1)$. 
By induction, there exists a matrix $S$, such that $S^{-1} X' S$ is a strict contraction. 
Take for $t > 0$ the invertible matrix $T_t = \left( \begin{smallmatrix} t & 0 \\ 0 & t^{-1} S^{-1} \end{smallmatrix} \right)$, then $T_t U^* X U T_t^{-1} = \left( \begin{smallmatrix} \alpha & t^2 v S \\ 0 & S^{-1} X' S \end{smallmatrix} \right)$. Since $\alpha \oplus (S^{-1} X' S)$ is a strict contraction, for $t$ small enough, so is $T_t U^* X U T_t^{-1}$.
\end{rem}

%%%%%%%%%%%%%%%%%%%%%%%%%%%
\subsection{The isomorphism problem in the commutative case}

%%%%%%%%%%%%%%%%%%%%%%%%%%%
%\begin{lem}\label{lem:comm_var_spec1}
%Let $\fV \subseteq \fC\fB_d$ be a commutative nc variety, and write $V$ for $\fV(1)$.
%Then for all $X \in \fV$, $\sigma(X) \subseteq V$.
%\end{lem}
%\begin{proof}
%Suppose that $X \in M_n^d$, and put $X_1, \ldots, X_d$ in upper triangular form.
%Then $\sigma(X)$ is precisely the set of points on the diagonals of $X$:
%\[
%\sigma(X) = \{((X_1)_{ii}, \ldots, (X_d)_{ii}) : i = 1, \ldots, n\}.
%\]
%If $f \in \cJ_\fV$, then the $i$th element on the diagonal of $f(X)$ is equal to $f((X_1)_{ii}, \ldots, (X_d)_{ii})$.
%Since $X \in \fV = V(\cJ_\fV)$, for such $f$ we have $f(X) = 0$.
%Thus, for all $\lambda \in \sigma(X)$ and all $f \in \cJ_\fV$, $f(\lambda) = 0$.
%Consequently, $\sigma(X) \subseteq V$.
%\end{proof}

%%%%%%%%%%%%%%%%%%%%%%%%%%%
%\begin{lem}\label{lem:comm_var_spec2}
%Then
%\[
%\widetilde{\fV} = \widetilde{\fC\fB_d} \cap \pi(\sqcup_k\Rep_k(H^\infty(\fV))).
%\]
%\end{lem}
%\begin{proof}

%If $X \in \widetilde{\fV} \subseteq \widetilde{\fC \fB_d}$ then $X = \pi(\Phi_X)$, and so $\widetilde{\fV} \subseteq \widetilde{\fC\fB_d} \cap \pi(\sqcup_k\Rep_k(H^\infty(\fV)))$.
%
%On the other hand, suppose that $X \in \pi(\Rep_k(H^\infty(\fV)))$ and $X \in \widetilde{\fC\fB_d}(k) \subseteq \widetilde{\fB_d}(k)$.
%By Theorem \ref{thm:finite_dim_reps}, $\Phi_X$ is the unique representation in $\pi^{-1}(X)$ and so $X \in \widetilde{\fV}(k)$.
%This shows $\widetilde{\fV} \supseteq \widetilde{\fC\fB_d} \cap \pi(\sqcup_k\Rep_k(H^\infty(\fV)))$.
%\end{proof}

%%%%%%%%%%%%%%%%%%%%%%%%%%%
\begin{prop}\label{prop:cb-homo=>weakstar_commutative}
Let $\fV \subseteq \fC\fB_d$ and $\fW \subseteq \fC \cB_e$ be two commutative nc varieties, and let $\alpha : H^\infty(\fV) \to H^\infty(\fW)$ be a bounded homomorphism.
Suppose that $\alpha^*$ maps every evaluation functional $\Phi_w : f \mapsto f(w)$ ($f \in H^\infty(\fW)$, $w \in W = \fW(1)$) to an evaluation functional on $H^\infty(\fV)$.
Then $\alpha^*$ maps $\widetilde{\fW}$ into $\widetilde{\fV}$.
\end{prop}
\begin{proof}
Define $G : \widetilde{\fW} \to \pi(\sqcup_k\Rep_k(H^\infty(\fV)))$ by
\[
G(T) = \pi_k (\alpha^*(\Phi_T))~, \quad T \in \widetilde{\fW}(k).
\]
Since $\alpha^*$ is a similarity preserving and a direct sum preserving map between the nc spaces of representations, it follows that $G$ is an nc map from $\widetilde{\fW}$ into $\pi(\sqcup_k\Rep_k(H^\infty(\fV))) = \ol{\widetilde{\fV}}$. 
Let $T \in \widetilde{\fW}$. 
We know that $G(T) \in \ol{\widetilde{\fV}}$, and we need to show that $G(T) \in \widetilde{\fV}$. 
Since $T$ is simultaneously upper-triangular in some basis, and since nc maps act entry-wise on the diagonal, we find that
\[
\sigma(G(T)) = G(\sigma(T)).
\]
But as $T \in \widetilde{\fW}(k)$, we know that $\sigma(T) \subseteq W = \fW(1)$.
By assumption, $G(\sigma(T)) \subseteq V = \fV(1) \subseteq \bB_d$, so Theorem \ref{thm:rep_similar} implies that $G(T) \in \widetilde{\fC \fB_d}$. %$G(T) \in \widetilde{\fV}(k)$.
Thus $\alpha^*(\Phi_T) = \Phi_{G(T)}$, and the proof is complete.
\end{proof}

Recall that a {\em character} of an algebra is a non-zero linear multiplicative functional.

%%%%%%%%%%%%%%%%%%%%%%%%%%%
\begin{prop}\label{prop:commutative_cb-homo=>holo}
Let $\fV \subseteq \fC\fB_d$ and $\fW \subseteq \fC \cB_e$ be two commutative nc varieties, and let $\alpha : H^\infty(\fV) \to H^\infty(\fW)$ be a unital and bounded homomorphism.
Suppose that $\alpha^*$ maps weak-$*$ continuous characters to weak-$*$ continuous characters.
Then there exists an nc holomorphic map $G: \widetilde{\fW} \to \widetilde{\fV}$
which implements $\alpha$ by the formula
\[
\alpha(f) = \widetilde{f} \circ G
\]
for all $f \in H^\infty(\fV)$ where $\widetilde{f}$ denotes the unique extension of $f$ to $\widetilde{\fV}$.
\end{prop}
\begin{proof}
Applying Proposition \ref{prop:cb-homo=>weakstar_commutative} to $\alpha$, we find that $\alpha^*$ restricts to an nc holomorphic map $G: \widetilde{\fW} \to \widetilde{\fV}$.
The conclusion now follows from Proposition \ref{prop:cb-homo=>holo} and Theorem \ref{thm:finite_dim_reps}.
\end{proof}

Putting the conclusion of the previous Proposition with Propositions \ref{prop:4equiv_prime} and \ref{prop:cb-homo=>holo}, we obtain the following strengthening of Corollary \ref{cor:cb-isom<=>biholo_weakstar}.

%%%%%%%%%%%%%%%%%%%%%%%%%%%
\begin{cor}\label{cor:cb-isom<=>biholo_commutative}
Let $d,e \in \mathbb N$, and $\fV \subseteq \fC\fB_d$ and $\fW \subseteq \fC\fB_{e}$ be commutative nc varieties.
Then  $H^\infty(\fV)$ and $H^\infty(\fW)$ are continuously isomorphic via an isomorphism that preserves weak-$*$ continuous characters if and only if the similarity envelopes $\widetilde{\fV}$ and $\widetilde{\fW}$ are nc biholomorphic via an nc biholomorphism $G: \widetilde{\fW} \to \widetilde{\fV}$ satisfying
\[
\sup_{W\in \fW}\|\Phi_{G(W)}\| < \infty \quad \text{and} \quad \sup_{V\in \fV}\|\Phi_{G^{-1}(V)}\|<\infty.
\]
Furthermore, in this case $\sup_{W\in \fW}\|\Phi_{G(W)}\|= \|\alpha\|$.

Likewise, $H^\infty(\fV)$ and $H^\infty(\fW)$ are completely boundedly isomorphic via an isomorphism that preserves weak-$*$ continuous characters if and only if the similarity envelopes $\widetilde{\fV}$ and $\widetilde{\fW}$ are nc biholomorphic via an nc biholomorphism $G: \widetilde{\fW} \to \widetilde{\fV}$ satisfying
\[
\sup_{W\in \fW}\|\Phi_{G(W)}\|_{cb} < \infty \quad \text{and} \quad \sup_{V\in \fV}\|\Phi_{G^{-1}(V)}\|_{cb}<\infty, 
\]
and in this case $\sup_{W\in \fW}\|\Phi_{G(W)}\|_{cb} = \|\alpha\|_{cb}$ and $\sup_{V\in \fV}\|\Phi_{G^{-1}(V)}\|_{cb} = \|\alpha^{-1}\|_{cb}$.
\end{cor}

\begin{rem}\label{rem:commutative_radical}
Let $\bB_d$ and $\bB_e$ be the commutative (classical) open unit balls in $d$ and $e$ dimensions, respectively. 
Let $V\subseteq \bB_d$ and $W \subseteq \bB_e$ be two subvarieties, meaning that $V$ and $W$ are both given as zero sets of multipliers of the Drury--Arveson space. 
Let $\fV$ and $\fW$ be the minimal nc varieties containing $V$ and $W$ as their first level. 
The ``quantized" varieties $\fV$ and $\fW$ are clearly commutative nc varieties, and the algebras $H^\infty(\fV)$ and $H^\infty(\fW)$ are the multiplier algebras $\cM_V$ and $\cM_W$ studied in \cite{DRS11} and \cite{DRS15}. 
In \cite{DRS11} the authors show that if $\cM_V$ and $\cM_W$ are isomorphic via a map that preserves weak-$*$ continuous characters, then the varieties $V$ and $W$ are {\em multiplier biholomorphic}, in the sense that the varieties are biholomorphic via maps $G$ and $G^{-1}$ whose coordinates are multipliers. 
Multiplier biholomorphism, however, is not even an equivalence relation; see \cite[Remark 6.7]{DHS14}. 
This raises the interesting open problem of describing an equivalence relation on varieties that encodes the isomorphism classes of algebras of the form $\cM_V$; see \cite[Subsection 7.3]{SalomonShalit}. 

Before handling this problem from our perspective, we wish to understand how ``variety quantization" (i.e., the passage from a variety $W \subseteq \bB_d$ to the smallest nc variety $\fW \subseteq \fB_d$ containing it as its first level $\fW(1) = W$) behaves with respect to maps between varieties. 
(It is perhaps worth stressing that here, as in the rest of the paper, by {\em nc variety} we mean an nc variety that is cut out by bounded nc holomorphic functions.)
Let $g:W \to V$ be a holomorphic map. 
Since the spectrum $\sigma(X)$ of every $X\in \widetilde\fW$ is a subset of $W$, the function $g$ extends to an nc holomorphic map $ G : \widetilde \fW \to \widetilde {\fC\fB_d}$ satisfying $\sigma( G(\widetilde W) ) \subseteq V$. 
We claim that if the latter map $G$ is $\phb$-Lipschitz, or equivalently, if $\sup_{X\in \fW}\|\Phi_{G(X)}\| < \infty$, then in fact $G$ maps $\widetilde \fW$ into $\widetilde \fV$, where $\fV$ is the smallest nc variety which has $\fV(1) = V$. 
To see why this is true, consider $X \in \widetilde \fW$. 
We wish to show that $f(G(X)) = 0$ every bounded nc function $f \in H^\infty(\fB_d)$ that vanishes on $V$. 
But for every such $f$, the composition $f \circ G$ is in $H^\infty(\fW)$ and vanishes on $W$ (because $G(W)  = g(W) \subseteq V$). 
From the definition of $\widetilde{\fW}$ it follows that $f \circ G$ vanishes on $\widetilde{\fW}$, whence $f(G(X)) = 0$, as required. 

Categorically speaking, if the morphisms in the category of commutative radical varieties are assumed to be Lipschitz on a quantized variety, then the quantized mapping maps the quantized variety to the quantization of the range, and so this quantization becomes a functor into the category of nc varieties with Lipschitz nc holomorphic maps. 

We now see, due to Corollary \ref {cor:cb-isom<=>biholo_commutative}, that the ``right" equivalence relation between varieties $V$ and $W$ corresponding to weak-$*$ isomorphism between the algebras $\cM_V$ and $\cM_W$, is a bi-Lipschitz biholomorphism of their quantizations. 
We do not know whether there exists a characterization of this equivalence relation that can be read directly from the original varieties $V$ and $W$.
\end{rem}

%%%%%%%%%%%%%%%%%%%%%%%%%%%
\subsection{A Nullstellensatz for homogeneous commutative nc varieties}

We now prove a Nullstellensatz that will be important for our classification result.
These results were essentially obtained already in \cite[Section 2.1.6]{RamseyThesis} (and in \cite[Section 6]{DRS11} in the norm closed setting), but there are some modifications we need, so we expand.
The reader is referred to \cite[Section 7]{SalShaSha17} for a similar discussion in the noncommutative setting.

We define, given a subset $\mathfrak{S} \subseteq \fC\fB_d$,
\[
\cJ_\mathfrak{S} = \{f \in H^\infty(\fB_d) :   f(X) = 0 \,\, \textrm{ for all } X \in \mathfrak{S}\}.
\]
Since we are dealing with the free commutative case, it is convenient to define
\[
\cJ_\mathfrak{S}^c = \{f \in \cM_d = H^\infty(\fC \fB_d):   f(X) = 0 \,\, \textrm{ for all } X \in \mathfrak{S}\},
\]
as well as the polynomial ideal
\[
I(\mathfrak{S}) = \{p \in \bC[z_1, \ldots, z_d]  :   p(X) = 0 \,\, \textrm{ for all } X \in \mathfrak{S}\}.
\]

%%%%%%%%%%%%%%%%%%%%%%%%%%%
\begin{thm}\label{thm:null_com_homo}
Let $\fV \subseteq \fC\fB_d$ be a homogeneous commutative nc variety and let $V = \fV(1)$.
There exists an integer $N$ such that for every $f \in \cM_d$,
\[
f \big|_{V} = 0 \Longrightarrow f^N \big|_\fV = 0 .
\]
Likewise, for every $f \in H^\infty(\fB_d)$
\[
f \big|_{V} = 0 \Longrightarrow f^N \big|_\fV = 0 .
\]
In other words, $f \in \cJ^c_V$ implies that $f^N \in \cJ^c_\fV$ and similarly for $\cJ_V$ and $\cJ_\fV$.
\end{thm}
\begin{proof}
The ideal $\cJ^c_\fV$ is homogeneous.
Therefore, if $f \in \cJ^c_\fV$ and $f = \sum_{n=0}^\infty f_n$, then every $f_n$ is in the polynomial ideal $I(\fV) \triangleleft \bC[z_1, \ldots, z_d]$.
It holds that
\[
f_n(X) = \int_0^{2\pi} f(e^{int}X) e^{-int} dt,
\]
and a familiar argument using the Fej\'er kernel shows that the Ces\`aro sums of the series $\sum f_n$ are bounded and converge pointwise to $f$.
Thus every $f \in \cJ_{\fV}^c$ is the weak-$*$ limit of a bounded sequence in $I(\fV)$.

By Hilbert's Basis Theorem, the radical $I(V) = \sqrt{I(\fV)}$ of the ideal $I(\fV) \triangleleft \bC[z]$ is finitely generated.
Combining this with Hilbert's Nullstellensatz, it is easy to see that there is some $N$ such that for every $p \in \bC[z_1, \ldots, z_d]$,
\[
p\big|_{V} = 0  \Longrightarrow p^N \in I(\fV) .
\]
If $f \in \cJ^c_V$, that is, if $f \in \cM_d$ vanishes on $V$, then, by the first paragraph of the proof, $f$ is the pointwise limit of a bounded sequence of polynomials $q_k \in I(V)$.
It follows that $f^N$ is the bounded pointwise limit of $q_k^N \in I(\fV)$, and therefore $f^N \in \cJ^c_\fV$.

That proves the statement for $f \in \cM_d$.
Now let $f \in H^\infty(\fB_d)$, and suppose that $f \big|_V = 0$.
Letting $q : H^\infty(\fB_d) \to \cM_d = H^\infty(\fC\fB_d)$ be the quotient map, we find that $q(f)\big|_V = 0$, so $q(f^N) = q(f)^N \in \cJ^c_{\fV}$.
Therefore, $f^N \in q^{-1}(\cJ^c_{\fV}) = \cJ_{\fV}$.
\end{proof}

%Given an algebra $\cB$ and an ideal $\cI$ in $\cB$, we define
%\[
%\operatorname{rad}_{\cB}(\cI) = \{f \in \cB :  f^N \in \cI \,\, \textrm{ for some } \,\, N \in \bN\}.
%\]
%We now recover Theorem 2.1.30 from \cite{RamseyThesis} (for the case $\cB = \cM_d$).
%%%%%%%%%%%%%%%%%%%%%%%%%%%%
%\begin{cor}
%For every weak-$*$ closed homogeneous ideal $\cI \triangleleft \cB$ (with $\cB = \cM_d$ or $\cB = H^\infty(\fB_d)$), if we put $V = \{z \in \bB_d : g(z) = 0 \textrm{ for all } g \in \cI\}$, then 
%\[
%\{f \in \cB : f\big|_V = 0\} = \operatorname{rad}_{\cB}(\cI).
%\]
%\end{cor}

%%%%%%%%%%%%%%%%%%%%%%%%%%%
\subsection{Classification of homogeneous commutative nc varieties}

%The main results of \cite[Section 5]{DRS15}, when translated to the language of this paper, can be stated as follows: if $\fV$ and $\fW$ are minimal commutative nc varieties in $\fC\fB_d$, and if $\alpha : H^\infty(\fV) \to H^\infty(\fW)$ is a continuous algebra isomorphism, then there exists a multiplier biholomorphism $G$ from $\fW(1)$ onto $\fV(1)$ that implements $\alpha$.
%It is worth noting that $G$ must then extend to a biholomorphism from $\widetilde{\fW}$ onto $\widetilde{\fV}$, since these are the minimal nc varieties containing $\fW(1)$ and $\fV(1)$, and since $G$ is induced by a bounded map.
%One of the main goals of this paper is to extend the results of \cite[Section 5]{DRS15} to the setting of nc  varieties.
%Under a weak-$*$ continuity assumption this was achieved in Corollary \ref{cor:cb-isom<=>biholo_weakstar}, and under a less stringent weak-$*$ continuity assumption this was achieved in the nc commutative case in Corollary \ref{cor:cb-isom<=>biholo_commutative}.
%In this section, we completely remove all weak-$*$ continuity assumptions, in the homogeneous commutative nc setting. 
%We already obtained this goal for general homogeneous nc varieties in Theorem \ref{thm:homog_no_need_w*}. 
The following theorem is a special case of Theorem \ref{thm:homog_no_need_w*}, when attention is restricted to commutative homogeneous nc varieties; 
we believe that it deserves to be presented separately, because of the importance of the commutative case, and also because we have a different proof. 

%%%%%%%%%%%%%%%%%%%%%%%%%%%
\begin{thm}\label{thm:commutative isom=>biholo2}
Let $\fV \subseteq \fC\fB_d$ and $\fW \subseteq \fC \cB_e$ be two homogeneous commutative nc varieties, and let $\alpha : H^\infty(\fV) \to H^\infty(\fW)$ be a bounded isomorphism.
Then there exists an nc biholomorphism $G: \widetilde{\fW} \to \widetilde{\fV}$
which implements $\alpha$ by the formula
\[
\alpha(f) = \widetilde{f} \circ G
\]
for all $f \in H^\infty(\fV)$ where $\widetilde{f}$ denotes the unique extension of $f$ to $\widetilde{\fV}$. 
Moreover, \[
\sup_{W\in \fW}\|\Phi_{G(W)}\|= \|\alpha\| ,
\] 
and if $\alpha$ is completely bounded, then 
\[
\sup_{W\in \fW}\|\Phi_{G(W)}\|_{cb} = \|\alpha\|_{cb} .
\]
\end{thm}
\begin{proof}
By Corollary \ref{cor:cb-isom<=>biholo_commutative}, it suffices to prove
that $\alpha^*$ restricts to a bijection between the weak-$*$ continuous functionals.
In fact, it suffices to show that $\alpha^*$ maps $W = \fW(1)$ into $V = \fV(1)$, since by symmetry this will also be true for $\alpha^{-1}$.

Suppose, therefore, that $\alpha^*$ maps an evaluation functional $\Phi_{w_0}$ to a functional $\rho = \alpha^*(\Phi_{w_0})$ which is not an evaluation functional. 
By Theorem \ref{thm:finite_dim_reps}, $\rho$ lies in a fiber over a point of the boundary of the ball, that is $\rho \in \pi_1^{-1}(v)$ for some $v \in \partial \bB_d \cap \ol{V}$.
Since $W$ is connected, we can use the proof of \cite[Lemma 5.3]{DRS15} to obtain that $\alpha^*(\Phi_w) \subseteq \pi_1^{-1}(v)$ for all $w \in W$.
We will now show that this leads to a contradiction.

Without loss of generality, we may assume that $v = (1, 0, \ldots, 0)$.
Now let $g = z_1 - 1$, which can be considered as a function in $\cM_d$ and also as a function in $H^\infty(\fW)$.
Then $g^N$ does not vanish on $\fV$ for any $N \in \bN$, since $A - I$ is invertible for any matrix $A$ with $\|A\|<1$.
On the other hand, for every $w \in W$,
\[
\alpha(g) (w) = \alpha^*(\Phi_w)(g) = \alpha^*(\Phi_w)(z_1) - \alpha^*(\Phi_w)(1) = 1-1 = 0.
\]
Note that this does not imply yet that $\alpha(g)$ is zero on $\fW$, just that it vanishes on the first level.
However, by Theorem \ref{thm:null_com_homo}, there exists $N$ such that $\alpha(g^N)  = \alpha(g)^N = 0$, and this contradicts the fact that $\alpha$ is an isomorphism.
This contradiction shows that $\alpha^*$ must map $W$ into $V$, and by the first paragraph of the proof, we are done.
\end{proof}

Combining Theorem \ref{thm:commutative isom=>biholo2} with Propositions \ref{prop:biholo=>0-biholo}, \ref{prop:zeros_to_zeros} and \ref{prop:bdd_linear=>cb_linear}, we recover Theorem \ref{thm:isom_thm_for_homo} for the case of commutative nc varieties (where Theorem \ref{thm:commutative isom=>biholo2} replaces Theorem \ref{thm:homog_no_need_w*} in the proof). 
Proposition \ref{prop:zeros_to_zeros} can be given a slightly different proof in the commutative case, using the joint spectrum instead of the joint spectral radius (we omit the details), and with this approach the proof of the commutative case becomes different in a meaningful way.

We know that for homogeneous nc varieties, and in particular for commutative homogeneous nc varieties, the algebras $H^\infty(\fV)$ and $H^\infty(\fW)$ are boundedly isomorphic if and only if there exists a $\phb$-bi-Lipschitz linear map mapping $\widetilde{\fW}$ onto $\widetilde{\fV}$. 
In Question \ref{quest:bilipLin} we asked whether every bijective linear map taking $\widetilde{\fW}$ onto $\widetilde{\fV}$ is automatically $\phb$-bi-Lipschitz. 
We were unable to resolve the question. 
The answer is known only when the varieties are commutative and minimal in a certain sense, as explained in the following example. 

%%%%%%%%%%%%%%%%%%%%%
\begin{example}
The answer to Question \ref{quest:bilipLin} is affirmative in the case that one (and hence both) of the varieties are the nc zero sets of a radical homogeneous ideal. 
This follows from a difficult result of Hartz \cite{Hartz12} (see Remark \ref{rem:commutative_radical} or \cite[Section 11]{SalShaSha17} for the connection). 
In our notation, Hartz's theorem says that if $A : W \to V$ a bijective linear map between two homogeneous varieties in the unit ball $\bB_d \subseteq \bC^d$, and if $\fV$ and $\fW$ are the minimal nc varieties in the nc ball $\fB_d$ containing $V$ and $W$ as their first levels, then pre-composing with $A$ gives rise to a bounded isomorphism from $H^\infty(\fV)$ onto $H^\infty(\fW)$. 
Recall that every nc holomorphic map $\widetilde{\fW} \to \widetilde{\fV}$ that gives rise to a homomorphism $H^\infty(\fV) \to H^\infty(\fW)$ by pre-composition is Lipschitz (Proposition \ref{prop:4equiv_prime}). 
Thus, Hartz's result provides a solution to Question \ref{quest:bilipLin} in the commutative, radical case, because the restriction to the first level a bijective linear map $A : \widetilde{\fW} \to \widetilde{\fV}$ is a bijective linear map $ A : W \to V$. 
\end{example}

%%%%%%%%%%%%%%%%%%%%%%%%%%%%%%%%%%%%%%%%%%%%%%%
\section{Algebras of continuous multipliers} \label{sec:continuous_case}

Let $\fV \subseteq \fB_d$ be an nc variety. 
Recall that we let $A(\fV)$ denote the algebra of all bounded analytic functions that extend to uniformly continuous functions on $\ol{\fV}$, in the sense that for every $\epsilon > 0$, there exists a $\delta > 0$, such that for every $n \in \N$ we have that if $X, Y \in \ol{\fV}(n)$ are such that $\|X - Y \| < \delta$, then $\|f(X) - f(Y)\| < \epsilon$ (in other words, these are the bounded analytic functions on $\ol{\fV}$ that are uniformly continuous on $\ol\fV(n)$, for all $n$, uniformly in $n$). 
We give $A(\fV)$ the sup norm (which coincides with the multiplier norm), and this gives $A(\fV)$ the structure of an operator algebra. 
By \cite[Section 9]{SalShaSha17}, these algebras can be considered to be continuous multipliers on the nc reproducing kernel Hilbert spaces $\mathcal H_\fV$. 

The algebras of the form $A(\fV)$ include among them many algebras that have been investigated before. 
For example, if $\fV$ is a disc, then $A(\fV) = A(\bD)$ is the disc algebra. 
If $\fV = \fB_d$ is an nc ball, then $A(\fV) = A(\fB_d)$ is the noncommutative disc algebra \cite{Popescu91}. 
If $\fV$ is a homogeneous ideal, then $A(\fV)$ is the tensor algebra associated with a subproduct system (with Hilbert space fibers), as studied in \cite{DRS11}, \cite{KakSh19} and \cite{ShalitSolel}. 

A natural question, in the spirit of our investigations, is: how does the (operator/Banach) algebraic structure of $A(\fV)$ reflect the geometric structure of $\fV$? 
Other natural algebras one might consider are (i) the closure $\fA_\fV$ of the algebra generated by the polynomials, with respect to the supremum norm and (ii) the quotient algebra $A(\fB_d) / \cI_\fV$, where 
\[
\cI_\fV = \{f \in A(\fB_d) : f(X) = 0 \textrm{ for all } X \in \fV\}. 
\]

In the general case, several technical difficulties immediately arise. 
First, it is not true that every $f \in A(\fV)$ has an extension $F \in A(\fB_d)$ such that $\|F\| = \|f\|$. 
Moreover, it might happen that $\cI_\fV = \{0\}$ for a nontrivial $\fV$ (see \cite[Theorem 8.1]{DRS15}). 
In this case, the restriction map $A(\fB_d) \to A(\fV)$ is an isomorphism, hence the algebraic structure is preserved while the geometry is dramatically changed. 
It is also not a simple matter to determine when $A(\fV) = \fA_\fV$. 
We refer the reader to Section 7 of the paper \cite{DRS15} for some discussion of the kind of subtleties involved. 

To recap, in the setting of multiplier algebras, we had the natural and completely isometric identifications 
\[
H^\infty(\fV) = H^\infty(\fB_d) \big|_{\fV} = H^\infty(\fB_d) / \cJ_\fV = \ol{\alg}^{w*} (z_1, \ldots, z_d) ,
\]
where the last algebra denotes the weak-operator (or weak-$*$) closure of the unital algebra generated by the free polynomials in $\mlt \cH^2_\fV$; whereas, when passing to the norm closed algebras, we do not know in general whether 
\begin{equation}\label{eq:Aequalities}
A(\fV) = A(\fB_d)\big|_{\fV} = A(\fB_d) / \cI_\fV = \fA_\fV := \ol{\alg}^{\|\cdot\|} (z_1, \ldots, z_d) 
\end{equation}
holds. In particular, we cannot identify the spaces $\Rep_{k}(A(\fV))$ of finite dimensional representations, which in the previous context has been the starting point of the classification.  
In \cite[Section 9]{SalShaSha17}, we proved that the identifications in \eqref{eq:Aequalities} all hold completely isometrically when $\fV$ is a homogeneous variety (with the provision that the appearance of $A(\fB_d)\big|_\fV$ in the identifications should be interpreted only as equality of sets). 

In the rest of this section we shall consider the isomorphism problem for the algebras of the kind $A(\fV)$ for a homogeneous nc variety $\fV \subseteq \fB_d$, where $d = \infty$ is included. 
We leave the investigation of the function theory, representation theory, and classification of the algebras $A(\fV)$ associated with general varieties for future work. 

In Section 9.3 of \cite{SalShaSha17} we observed that when $\fV$ is homogeneous, every $X \in \ol{\fV}$ gives rise to completely contractive unital representation 
\[
f \mapsto f(X).
\]
Further, for every $\Phi \in \Rep_{k}^{cc}(A(\fV))$ there exists $X \in \ol{\fV}$ such that $\Phi = \Phi_X$. 
Using --- as is done in the proof of Theorem \ref{thm:finite_dim_reps} --- the fact that bounded representations into $M_n$ are completely bounded, together with the fact that completely bounded representations are similar to completely contractive ones, we obtain the following description of the bounded finite dimensional representations of $A(\fV)$.

%%%%%%%%%%%%%%%%%%%%%%%%%%%%%%%%%%%
\begin{prop}\label{prop:finite_dim_reps_AV}
Let $\fV \subseteq \fB_d$ be a homogeneous nc variety. 
For very $k \in \bN$, the natural projection $\pi_{k}$ of $\Rep_{k}(A(\fV))$ into $\widetilde{\ol{\fB_d}}$ given by 
\[
\pi_{k}(\Phi) = (\Phi(z_1), \ldots, \Phi(z_d)) 
\]
is a  bijection onto the similarity invariant envelope $\widetilde{\ol{\fV}}$ of $\ol{\fV}$. 
\end{prop}

Once we know that every bounded finite dimensional representation of $A(\fV)$ is given by evaluation at a point of $\widetilde{\ol{\fV}}$, we can use the methods of this paper (with significantly simpler proofs) to prove the following classification results. 

%%%%%%%%%%%%%%%%%%%%%%%%%%%%%%%%
\begin{thm}\label{thm:isom_thm_for_homoAV1}
Let $\fV \subseteq \fB_d$ and $\fW \subseteq \fB_{e}$ be homogeneous nc varieties.
Let $\alpha : A(\fV) \to A(\fW)$ be a unital bounded homomorphism.   
Then there exists an nc map $G: \widetilde{\ol{\fW}} \to \widetilde{\ol{\fV}}$ such that $G(\widetilde{\ol{\fW}}) = \widetilde{\ol{\fV}}$, which implements $\alpha$ by the formula
\[
\alpha(f) = \widetilde{f} \circ G. 
\]
In this case, $\|\alpha\|=\sup_{W\in \fW}\|\Phi_{G(W)}\|$, and $\|\alpha\|_{cb}=\sup_{W\in \fW}\|\Phi_{G(W)}\|_{cb}$ if $\alpha$ is completely bounded. 
Moreover, if we write $G = (G_1, \ldots, G_d)$, then the component $G_i$ is in $A(\fW)$ for all $i=1, \ldots, d$. 

Consequently, $A(\fV)$ and $A(\fW)$ are (completely) boundedly isomorphic, if and only if there exist ($\ph$-bi-Lipschitz) $\phb$-bi-Lipschitz nc maps $G : \widetilde{\ol{\fW}} \to \widetilde{\ol{\fV}}$ with components in $A(\fW)$, and $H :  \widetilde{\ol{\fV}} \to \widetilde{\ol{\fW}}$ with components in $A(\fV)$, which are mutual inverses of each other. 
\end{thm}

%%%%%%%%%%%%%%%%%%%%%%%%%%%%%%%%
\begin{thm}\label{thm:isom_thm_for_homoAV2}
Let $\fV \subseteq \fB_d$ and $\fW \subseteq \cB_e$ be two homogeneous nc varieties. 
The following are equivalent: 
\begin{enumerate}
\item $A(\fV)$ and $A(\fW)$ are boundedly isomorphic. 
\item $A(\fV)$ and $A(\fW)$ are completely boundedly isomorphic. 
\item There exists a $\phb$-bi-Lipschitz linear map mapping $\widetilde{\ol\fW}$ onto $\widetilde{\ol\fV}$. 
\item There exists a $\ph$-bi-Lipschitz linear map mapping $\widetilde{\ol\fW}$ onto $\widetilde{\ol\fV}$. 
\end{enumerate}
\end{thm}

%%%%%%%%%%%%%%%%%%%%%%%%%%%%%%%%
\begin{rem}
When $\fV$ and $\fW$ are varieties determined by {\em monomials} in (finitely many) noncommuting variables, the algebras $A(\fV)$ and $A(\fW)$ are precisely the tensor algebras studied in \cite{KakSh19} (the algebra $A(\fV)$ appeared in \cite{KakSh19} as $\cA_X$). 
In \cite[Theorem 9.2]{KakSh19} it was shown that in this case $A(\fV)$ and $A(\fW)$ are algebraically isomorphic if and only if they are completely isometrically isomorphic, and this happens if and only if there is a permutation of the variables, such that the defining monomials (and hence the varieties) are the same. 
\end{rem}

\bibliographystyle{abbrv}
\bibliography{nc_bibliography}

\end{document}